\documentclass[11pt, reqno]{amsart}

\usepackage{amsmath}
\usepackage{amsfonts}
\usepackage{amssymb}
\usepackage{amsthm}
\usepackage{mathrsfs}
\usepackage{graphicx}
\usepackage{cite}
\usepackage{caption}
\usepackage[usenames,dvipsnames]{color}
\usepackage[all]{xy}
\usepackage[colorinlistoftodos]{todonotes}
\usepackage{mathtools}
\usepackage{bbm}
\usepackage{enumitem}
\usepackage[colorlinks=true,linkcolor=Red,citecolor=Green]{hyperref}
\usepackage{soul}
\usepackage[normalem]{ulem}

\usepackage[top=3.35cm, bottom=3.35cm, left=2.5cm, right=2.5cm, headsep=0.2in]{geometry}

\DeclareMathOperator{\im}{Im}

\DeclareMathOperator{\ind}{ind}

\DeclareMathOperator{\supp}{supp}
\DeclareMathOperator{\vol}{vol}
\DeclareMathOperator{\Res}{Res}
\DeclareMathOperator{\res}{Res}

\DeclareMathOperator{\id}{Id}
\DeclareMathOperator{\E}{\mathcal{E}}

\DeclareMathOperator{\Op}{Op}

\DeclareMathOperator{\ran}{ran}

\DeclareMathOperator{\divv}{div}

\DeclareMathOperator{\Lie}{\mathcal{L}}
\newcommand{\mc}{\mathcal}
\newcommand{\WF}{\mathrm{WF}}
\newcommand{\llangle}{\langle\!\langle}
\newcommand{\rrangle}{\rangle\!\rangle}

\newcommand{\lk}{\operatorname{lk}}
\newcommand{\re}{\mathrm{Re}}
\newcommand{\SRB}{\Omega_{\mathrm{SRB}}^+}
\newcommand{\SRBs}{\Omega_{\mathrm{SRB}}^-}

\newcommand{\M}{\mc{M}}

\newcommand{\e}{\mathbf{e}}

\theoremstyle{plain}
\newtheorem{theorem}{Theorem}[section]

\newtheorem{lemma}[theorem]{Lemma}
\newtheorem{Remark}[theorem]{Remark}
\newtheorem{prop}[theorem]{Proposition}
\newtheorem{conj}[theorem]{Conjecture}

\newtheorem{corollary}[theorem]{Corollary}

\numberwithin{equation}{section}

\begin{document}

\title{Resonant forms at zero for dissipative Anosov flows}
\author[M. Ceki\'{c}]{Mihajlo Ceki\'{c}}
\date{\today}
\address{Institut f\"ur Mathematik, Universit\"at Z\"urich, Winterthurerstrasse 190,
CH-8057 Z\"urich, Switzerland}
\email{mihajlo.cekic@math.uzh.ch}
\author[G.P. Paternain]{Gabriel P. Paternain}
\address{Department of Pure Mathematics and Mathematical Statistics, University of Cambridge, Cambridge CB3 0WB, UK and Department of Mathematics, University of Washington, Seattle, WA 98195, USA.}
\email{gpp24@uw.edu}

\dedicatory{To the memory of Will Merry}

\maketitle

\begin{abstract}
	We study resonant differential forms at zero for transitive Anosov flows on $3$-manifolds.
We pay particular attention to the dissipative case, that is, Anosov flows that do not preserve an absolutely continuous measure.	
Such flows have two distinguished Sinai-Ruelle-Bowen 3-forms, $\Omega_{\text{SRB}}^{\pm}$, and the cohomology classes $[\iota_{X}\Omega_{\text{SRB}}^{\pm}]$ (where $X$ is the infinitesimal generator of the flow) play a key role in the determination of the space of resonant 1-forms.
When both classes vanish we associate to the flow a {\it helicity} that naturally extends the classical notion associated with null-homologous volume preserving flows. We provide a general theory that includes horocyclic invariance of resonant 1-forms and SRB-measures as well as the local geometry of the maps $X\mapsto [\iota_{X}\Omega_{\text{SRB}}^{\pm}]$ near a null-homologous volume preserving flow. Next, we study several relevant classes of examples. Among these are thermostats associated with holomorphic quadratic differentials, giving rise to quasi-Fuchsian flows as introduced by Ghys \cite{Ghys-92}. For these flows we compute explicitly all resonant 1-forms at zero, we show that $[\iota_{X}\Omega_{\text{SRB}}^{\pm}]=0$
and give an explicit formula for the helicity. In addition we show that a generic time change of a quasi-Fuchsian flow is semisimple and thus the order of vanishing of the Ruelle zeta function at zero is $-\chi(M)$, the same as in the geodesic flow case.
In contrast, we show that if $(M,g)$ is a closed surface of negative curvature, the Gaussian thermostat driven by a (small) harmonic 1-form has a Ruelle zeta function whose order of vanishing at zero is $-\chi(M)-1$.
\end{abstract}

\tableofcontents

\section{Introduction}
If $X$ is a vector field on a closed manifold $\mc{M}$, then typically the $L^2$ spectrum of $X$ is not discrete but it consists of essential spectrum. However, when $X$ is \emph{Anosov} (or \emph{uniformly hyperbolic}), there is a hidden discrete \emph{resonance} spectrum obtained by the action of $X$ on some spaces of \emph{anisotropic regularity}. This spectrum, and the associated \emph{resonant states} (eigenfunctions of $X$) encode many significant properties of the flow, for instance one can read off whether $X$ is ergodic or mixing (with respect to the Sinai-Ruelle-Bowen measure), or even exponentially mixing. Furthermore, the analogously defined space of resonant states $\Res_0^k$ of the Lie derivative $\Lie_X$ (acting on $k$-forms) at zero carries rich topological data, for instance it is related to the Betti numbers of $\mc{M}$ or to dynamical invariants of the flow generated by $X$, and is related to the periodic orbit spectrum (the set of periods of periodic orbits) through the \emph{Ruelle zeta function}. In this paper, we are interested in this spectrum of $\Lie_X$ at zero.

In general, the dependence of the resonant space spectrum of $\Lie_X$ at zero on various topological or geometrical invariants of $\mc{M}$ and $X$ is now understood to be very subtle. The goal of this paper is to cast some light on the more difficult case when $X$ does not preserve a smooth measure and to illustrate the general picture with important classes of examples.

\subsection{Geometric multiplicities of resonant forms}

Assume $\mc{M}$ is a closed smooth $3$-manifold, equipped with a smooth vector field $X$ generating an \emph{Anosov} flow $\varphi_t$. This means that there is a continuous and flow invariant splitting of the tangent space
\[T \mc{M} = \mathbb{R} X \oplus E_s \oplus E_u,\]
into flow, stable, and unstable directions, respectively, and such that for some constants $C, \nu > 0$ and an arbitrary metric $|\bullet|$ on $\mc{M}$, and all $x \in \mc{M}$:
\begin{align}\label{eq:anosov-def}
\begin{split}
|d\varphi_{t}(x) \cdot v| \leq Ce^{-\nu |t|} \cdot |v|, \quad \begin{cases} t \geq 0,\quad v \in E_s(x),\\
t \leq 0,\quad v \in E_u(x).
\end{cases}
\end{split}
\end{align}
It is well-known that the geodesic flow on the unit tangent bundle $\mc{M} = SM$ of a Riemannian manifold $(M, g)$ with negative sectional curvature is Anosov, see \cite{Anosov-67}. Denote $E_{u/s}^* := (\mathbb{R}X \oplus E_{u/s})^\perp \subset T^*\mc{M}$, where $\bullet^\perp$ is the annihilator of $\bullet$.

Denote by $\Omega^k$ the bundle of differential $k$-forms and by $\Omega_0^k = \Omega^k \cap \ker \iota_X \subset \Omega^k$ the bundle of $k$-forms in the kernel of the contraction with the vector field. Write $\mc{D}'(\mc{M}; \Omega^k)$ for the space of distributional sections of $\Omega^k$. Given $\ell \in \mathbb{Z}_{\geq 1}$, introduce the spaces of \emph{generalised resonant $k$-forms at zero} by
\begin{align*}
	\Res^{k, \ell} &:= \{u \in \mc{D}'(\mc{M}; \Omega^k) \mid \Lie_X^{\ell} u = 0, \WF(u) \subset E_u^*\}, \quad &&\Res^{k, \infty} := \bigcup_{\ell \geq 1} \Res^{k, \ell},\\
	\Res_0^{k, \ell} &:= \Res^{k, \ell} \cap \ker \iota_X, \quad &&\Res_0^{k, \infty} := \Res^{k, \infty} \cap \ker \iota_X,
\end{align*}
where $\WF(u) \subset T^*M \setminus 0$ denotes the \emph{wavefront set} of a distribution, see \cite[Chapter 8]{Hormander-90}, and $\Lie_X = \iota_X d + d\iota_X$ is the Lie derivative. These spaces are finite dimensional: this is a non-trivial fact that follows from the construction of \emph{anisotropic Sobolev spaces} tailored to the flow, on which $\Lie_X$ acts as a Fredholm operator (see \cite{Faure-Sjostrand-11} or \cite{Dyatlov-Zworski-16}).

Write $m_{k, 0}^{\infty} := \dim \Res_0^{k, \infty}$ for the \emph{algebraic multiplicity}. When $\ell = 1$, denote $\Res_0^k := \Res_0^{k, 1}$ and $\Res^k := \Res^{k, 1}$, the spaces of \emph{resonant $k$-forms at zero}, and write $m_{k, 0} := \dim \Res_0^{k}$ and $m_{k} := \dim \Res^{k}$ for the \emph{geometric multiplicities}. Say that the action of $\Lie_X$ on $\Omega_0^k$ or $\Omega^k$ is \emph{semisimple} (at zero), if $\Res_0^k = \Res_0^{k, \infty}$ or $\Res^k = \Res^{k, \infty}$, respectively.

The \emph{Ruelle zeta function}
\begin{equation}\label{eq:RZF}
\zeta_{\mathrm{R}}(s):= \prod_{\gamma}{\big(1 - e^{-sT_\gamma}\big)}, \quad \re(s) \gg 1,
\end{equation}
is a converging product for $\re(s)$ large enough and admits a meromorphic continuation to $s \in \mathbb{C}$ by the work of Giulietti-Liverani-Pollicott \cite{Giulietti-Liverani-Pollicott-13} and Dyatlov-Zworski \cite{Dyatlov-Zworski-16}. Here the product is over all primitive closed orbits of the flow. Define the \emph{order of vanishing of the zeta function at zero} to be the unique integer $m_{\mathrm{R}}(0)$ such that $s^{-m_{\mathrm{R}}(0)} \zeta_{\mathrm{R}}(s)$ is holomorphic and non-zero at $s = 0$. According to \cite[Section 3]{Dyatlov-Zworski-17} we have: 
\begin{equation}\label{eq:order-of-vanishing-formula}
	m_{\mathrm{R}}(0) = m_{1, 0}^\infty - m_{0, 0}^\infty - m_{2, 0}^\infty.
\end{equation}
It is known that if $X$ is a \emph{transitive} Anosov vector field (i.e.  it has a dense orbit), then the action of $\Lie_X$ on $\Omega^0$ and $\Omega^2_0$ is semisimple and $m_{0, 0} = m_{2, 0} = 1$ (see Proposition \ref{prop:k=0,2,3} below). Under this assumption, $m_{1, 0}^\infty$ is hence the only remaining unknown in \eqref{eq:order-of-vanishing-formula}.

We now introduce some important dynamical invariants of the $X$. Denote by $\SRB$ the Sinai-Ruelle-Bowen (SRB) measure of $X$, that is, the unique invariant probability measure such that $\WF(\SRB) \subset E_u^*$; similarly introduce the SRB measure $\SRBs$ for the flow $-X$ (see \S \ref{ssec:SRB-entropy} for more details). We may equivalently regard $\Omega_{\mathrm{SRB}}^\pm$ as (distributional) $3$-forms, normalised so that they have integral 1. Since $\omega^\pm := \iota_X \Omega_{\mathrm{SRB}}^\pm$ are closed, the $2$-forms $\omega^\pm$ define de Rham cohomology classes in $H^2(\mc{M})$ that are Poincar\'e dual to the classical \emph{winding cycles} of the SRB measures (also known as \emph{asymptotic cycles} \cite{Schwartzman-57}). If \emph{both} winding cycles vanish, i.e. $[\omega^{+}]_{H^2(\mc{M})} = [\omega^{-}]_{H^2(\mc{M})} = 0$, by Lemma \ref{lemma:hodge-decomposition} below we may write $\omega^\pm = d\tau^\pm$ for some $\tau^+ \in \mc{D}'(\mc{M}; \Omega^1)$ (resp. $\tau^- \in \mc{D}'(\mc{M}; \Omega^1)$) with $\WF(\tau^+) \subset E_{u}^*$ (resp. $\WF(\tau^+) \subset E_{u}^*$) and it is possible to define the \emph{helicity} $\mc{H}(X)$ by
\begin{equation}\label{eq:helicity}
	\mathcal{H}(X) := \int_{\mc{M}} \tau^+(X)~\SRBs = \int_{\mc{M}} \tau^-(X)\, \SRB,
\end{equation}
thanks to the wavefront set conditions. This quantity is independent of any choices (see Section \ref{sec:general}) and it agrees with the well-known concept of helicity in the volume preserving case. The following statement generalises \cite[Theorem 1.2]{Cekic-Paternain-20}:

\def\mystrut{\vrule height13pt depth7pt width0pt}
\begin{theorem}\label{thm:general}
	Assume $X$ generates a transitive Anosov flow on $\mc{M}$. The \emph{\bf geometric} multiplicity for the action of $\Lie_X$ on $\Omega_0^1$ and $\Omega^1$ is determined as a function of $[\omega^\pm]_{H^2(\mc{M})}$ and $\mc{H}(X)$ by the following table:
	\vspace{5pt}
\begin{center}
\begin{tabular}{|l|c|c|c|c|c|}
\hline
 ~~~~~~~~~~{\rm \textbf{Cases}} & \begin{tabular}{@{}c@{}}$[\omega^+] \neq 0$ \\ $[\omega^-] \neq 0$\end{tabular} & \begin{tabular}{@{}c@{}}$[\omega^+] \neq 0$ \\ $[\omega^-] = 0$\end{tabular} & \begin{tabular}{@{}c@{}}$[\omega^+] = 0$ \\ $[\omega^-] \neq 0$\end{tabular} & \begin{tabular}{@{}c@{}}$[\omega^+] = [\omega^-] = 0$ \\ $\mc{H}(X) \neq 0$\end{tabular} & \begin{tabular}{@{}c@{}}$[\omega^+] = [\omega^-] = 0$ \\ $\mc{H}(X) = 0$\end{tabular}\\
\hline\hline
\mystrut $d(\Res_0^1)$ & $0$ & $0$ & $\mathbb{C} \omega^+$ & $0$ & $\mathbb{C} \omega^+$\\
\hline
\mystrut $m_{1, 0} = \dim \Res_0^1$ & $b_1(\mc{M}) - 1$ & $b_1(\mc{M})$ & $b_1(\mc{M})$ & $b_1(\mc{M})$ & $b_1(\mc{M}) + 1$\\
\hline
\mystrut $d(\Res^1)$ & $0$ & $0$ & $\mathbb{C} \omega^+$ & $\mathbb{C} \omega^+$ & $\mathbb{C} \omega^+$\\
\hline
\mystrut $m_1 = \dim \Res^1$ & $b_1(\mc{M}) $ & $b_1(\mc{M})$ & $b_1(\mc{M})+1$ & $b_1(\mc{M})+1$ & $b_1(\mc{M}) + 1$\\
\hline
\end{tabular}
\end{center}
Moreover, the map $\Res^1\cap \ker d\ni u\mapsto [u]\in H^{1}(\mc{M})$ is an isomorphism.
\end{theorem}
\vspace{5pt}

Note that there are two more cases compared to \cite[Theorem 1.2]{Cekic-Paternain-20}, due to the fact that $\SRB \neq \SRBs$ if the flow does not preserve a smooth volume; these are the second and third columns (the contact case belongs to the fourth column). We would like to point out that, to the best of our knowledge, there are no known examples of Anosov flows with zero helicity. Observe that Theorem \ref{thm:general} is really a statement about the 1-dimensional oriented foliation determined by $X$, as the resonant spaces $\Res_{0}^{1}$ and the conditions defining the various cases are all independent of time changes of $X$. The theorem shows that $\dim \Res^1$ is also invariant under time changes. However, it should be noted that in principle some of the resonant forms in $\Res^1$ could be altered under time changes (in the first, third, and fourth cases) but not their key properties. We explain this in more detail in Remark \ref{remark:tc}. 

Finally, we remark that Theorem \ref{thm:general} should have an analogue in the \emph{non-transitive} case (examples of such Anosov flows were constructed in \cite{Franks-Williamas-80}): indeed, then there are finitely many SRB measures in bijection with the kernel of $\Lie_X$ on $\Res^3$ by \cite[Theorem 1]{Butterley-Liverani-07} or \cite[Theorem 3]{GuedesBonthonneau-Guillarmou-Hilgert-Weich-21}, and in principle a similar analysis applies.

Given that $[\omega^{\pm}]$ and $\mathcal H(X)$ feature prominently in Theorem \ref{thm:general} we next give some additional insights into their properties.

\subsection{Helicity of the SRB measures} The helicity is traditionally defined in the context of null-homologous volume preserving flows \cite{AK_98}. One of our ancillary objectives is to explain that in the Anosov case, this quantity can be defined in dissipative situations using the SRB measures and that it captures similar features as in the volume preserving case. For the quantity to be well-defined we require both cohomology classes $[\omega^{\pm}]=0$ as explained above. Equation \eqref{eq:helicity} defines $\mathcal{H}(X)$ taking advantage of the wave front set conditions and hence as a distributional pairing and in this form it is well suited for our proof of Theorem \ref{thm:general}. However, a natural question immediately arises: is it possible to express $\mathcal{H}(X)$ using \emph{linking forms} as in \cite{Contreras-95} or \cite{Kotschick-Vogel-03}? This would interpret $\mc{H}(X)$ introduced by \eqref{eq:helicity} as an averaged linking number.

We show that this is indeed the case. Let $K \in \mc{D}'(\mc{M} \times \mc{M}; \pi_1^*\Omega^1 \otimes \pi_2^*\Omega^1)$ be the Schwartz kernel of the Green operator $G$ of a fixed background Riemannian metric $g$; by definition, if $\Delta_g$ denotes the Hodge Laplacian on $1$-forms, $G$ is given by $0$ and $\Delta_g^{-1}$ on the $\ker \Delta_g$ and its $L^2$ complement, respectively. Here $\pi_1$ and $\pi_2$ denote projections onto the first and second factors of $\mc{M} \times \mc{M}$, respectively. The \emph{linking form} $L \in \mc{D}'(\mc{M} \times \mc{M}; \pi_1^*\Omega^1 \otimes \pi_2^*\Omega^1)$ is defined as the double form
\[L(x, y) := \star_y d_y K(x, y),\]
where $\star$ denotes the Hodge star operator. It satisfies the property that when integrated over two knots in $\mc{M}$ it gives the \emph{linking number} of the two knots, see \cite[Proposition 1]{Kotschick-Vogel-03}. Set $\Lambda(x,y):=L(x,y)(X(x),X(y)) \in  \mc{D}'(\mc{M} \times \mc{M})$; it is not hard to see that $\Lambda$ is smooth outside the diagonal $\Delta(\mc{M})$ and at the diagonal it has a singularity of type $d(x,y)^{-1}$ (where $d(x, y)$ is the distance between $x$ and $y$). Also, as explained in Section \ref{section:helicity}, the wavefront set calculus implies that $\Lambda(x,y) \, \SRB(x) \times \SRBs(y)$ is well defined as a distribution. We have

\begin{theorem}\label{thm:helicity-formula}
The function $\Lambda|_{\mc{M} \times \mc{M} \setminus \Delta(\mc{M})}$ is integrable with respect to $\SRB \times \SRBs$. The following formula holds:
	\begin{align*}
		\mc{H}(X) &= \int_{(x, y) \in \mc{M} \times \mc{M}} \Lambda(x, y)\, \SRB(x) \times \SRBs(y)\\ 
		               &= \lim_{R \to 0^+} \int_{(x, y) \in \mc{M} \times \mc{M} \setminus B_R} \Lambda(x, y)\, \SRB(x) \times \SRBs(y)
	\end{align*}
	where in the last line $(B_R)_{R > 0}$ is any nested family of neighbourhoods of $\Delta(\mc{M}) \subset \mc{M} \times \mc{M}$, such that $\cap_{R > 0} B_R = \Delta(\mc{M})$, and the integral is interpreted classically.
\end{theorem}

We note that the first equality in the theorem is a direct consequence of the wavefront set calculus (somewhat similar to \eqref{eq:helicity}) and of using the Green operator $G$. Indeed, we have
\[
	\mc{H}(X) = \int_{\M} \tau^+ \wedge d\tau^- = \int_{\M} G d^* d\tau^+ \wedge d\tau^- = \int_{(x, y) \in \mc{M} \times \mc{M}} \Lambda(x, y)\, \SRB(x) \times \SRBs(y),
\]
where in the second equality we used that $\tau^+$ is equal to $Gd^*d\tau^+$ up to closed terms (see \eqref{eq:G-formula} below), and in the third one that $\star_y L(x, y)$ is the Schwartz kernel of $Gd^*$ (see \eqref{eq:G-1} below). The second equality of the theorem is its main component and we emphasise that its proof requires some care due to the singular nature of the SRB measures. Moreover, using recent results of Coles-Sharp \cite{Coles-Sharp-21} it is also possible by means of Theorem \ref{thm:helicity-formula} to express the helicity in terms of linkings of closed orbits; see Proposition \ref{prop:classical-helicity}. Finally we observe that the expression
\[\mc{H}(X) = \int_{\M} \tau^+ \wedge \omega^{-}= \int_{\M} \tau^- \wedge \omega^{+}\]
allows us also to directly interpret the helicity as a linking between the two SRB measures when both are homologically trivial (i.e. when $[\omega^\pm]_{H^2(\M)} = 0$).

\subsection{Perturbation theory}\label{subsec:perturbation} Given the relevance of the classes $[\omega^{\pm}]$ it is natural to try to understand in more detail the structure of those Anosov vector fields $X$ for which $[\omega^{\pm}(X)]=0$ at least for $X$ close to $X_{0}$, where $X_{0}$ preserves a volume form $\Omega$ and is null-homologous (i.e. $[\omega^{\pm}(X_0)]=0$). For what follows we fix such a vector field $X_0$ and we work in a fixed $C^{N}$-neighbourhood $\mathcal U$ of $X_0$ for $N$ sufficiently large. Set
\[\mathcal W^{\pm}:=\{X\in C^{N}(\mc{M},T\mc{M}) \mid \;[\omega^{\pm}(X)]=0\}\cap \mathcal U,\;\;\;\mathcal W:=\mathcal W^{+}\cap\mathcal W^{-}.\]
We shall show that $\mc{W}^\pm$ are $C^1$ Banach submanifolds near $X_0$ of codimension $b_1(\M)$, the first Betti number of $\M$. The intersection $\mc{W}^+ \cap \mc{W}^-$ is transversal, that is, it is a $C^1$ Banach submanifold of codimension $2b_1(\M)$. For the proof, see Lemma \ref{lemma:banach-manifold} below. These results can be complemented with results on helicity: assuming additionally $\mathcal H(X_0)=0$, the set $\{X \mid \mathcal H(X)=0\}\subset \mathcal W$ is locally a $C^1$ Banach submanifold of codimension 1 (see Proposition \ref{prop:helicity-zero} and Figure \ref{figure:W-plus-minus}). We reiterate that we do not know if such an $X_0$ exists.

\begin{center}
\begin{figure}[htbp!]
\includegraphics[scale=0.75	]{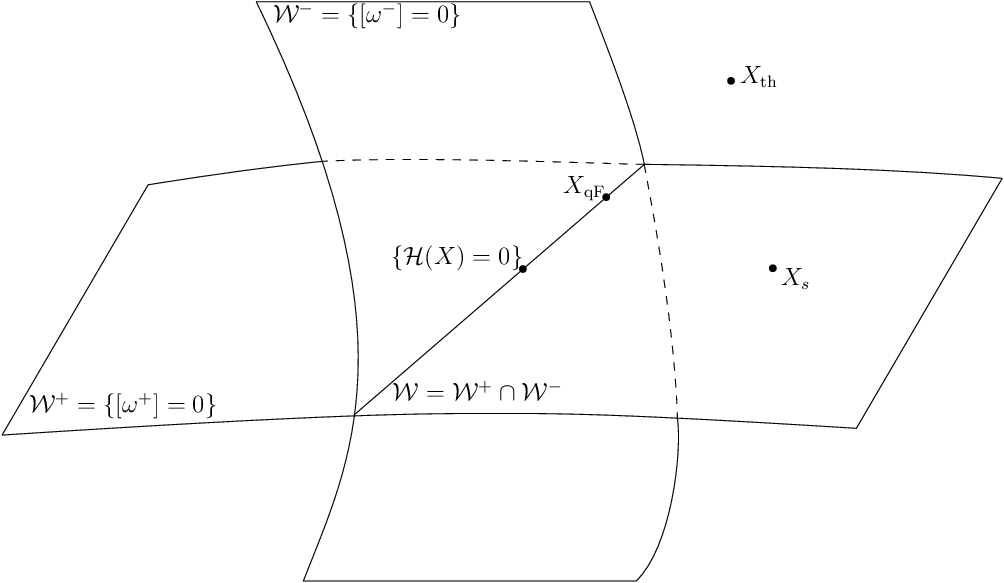}
\caption{Local geometry of vector fields according to the value of the winding cycle and the helicity. Flows $X_s \in \mc{W}^+$, but $X_s \not \in \mc{W}^-$, are constructed in Proposition \ref{prop:newexamples}, quasi-Fuchsian flows $X_{\mathrm{qF}} \in \mc{W}$ are studied in Section \ref{sec:qhd}, and thermostat flows satisfying $X_{\mathrm{th}} \not \in \mc{W}^+ \cup \mc{W}^-$ are provided by Theorem \ref{thm:beta}.}
\label{figure:W-plus-minus}
\end{figure}
\end{center}

\subsection{Horocyclic invariance of resonant forms} Resonant states and forms often exhibit an additional kind of invariance. This is most evident on constant negatively curved surfaces, where a resonant state $u$ at $s \in \mathbb{C}$ satisfying $(X + s)u = 0$ (and $\WF(u) \subset E_u^*$) also satisfies $(X + s + 1) U_-u=0$, where $U_-$ is the smooth vector field spanning the unstable foliation $E_u$. Since there are no resonances in the right half-plane, we conclude that $U_-^ku = 0$ for $k$ large enough (this is studied in detail in \cite{Dyatlov-Faure-Guillarmou-15}).

When we change the constant curvature setting to an arbitrary Anosov flow on $\M$, the vector field $U_-$ spanning $E_u$ becomes only H\"older regular in general and the definition of the derivative $U_-u$ is not immediately clear. However, when $X$ is a contact vector field and $|\re(s)|$ is sufficiently small, $u$ and $U_-$ have enough regularity for $U_-u$ to exist as was observed by Faure-Guillarmou \cite{Faure-Guillarmou-18} who showed $U_-u = 0$ in this situation. This observation was later used to show the existence of the first \emph{spectral band} for resonances by Guillarmou and the first author \cite{Cekic-Guillarmou-21} (see Remark \ref{rem:horocyclic-invariance-general} for a further discussion). Let us also mention that \cite{Tsujii-18} and \cite{Faure-Tsujii-21} introduce a more abstract notion of `horocycle operators'.

In this paper, for \emph{arbitrary} transitive Anosov flows in dimension $3$, we are able to prove a version of horocyclic invariance for closed resonant $1$-forms $u \in \Res_0^1$ (see Lemmas \ref{lemma:horo-I} and \ref{lemma:horocyclic-invariance-new}), including the fact that they are zero on the weak unstable bundle $\mathbb{R}X \oplus E_u$. The idea is to work with the regularity of the \emph{weak unstable bundle}, which is $C^{1 + \alpha}$-regular for some $\alpha > 0$. Similar results were proved in the constant negative curvature setting, see \cite{Cekic-Dyatlov-Delarue-Paternain-22} and \cite{Kuster-Weich-20}. We note that in contrast to \cite{Faure-Guillarmou-18} where the horocyclic invariance holds near the imaginary axis, the resonant forms at zero are ``deep inside" the resonance spectrum (and as is well-known the leading resonance is given by the topological entropy). There are several reasons making this possible: 1) contracting with vector fields costs less derivatives than differentiation, and 2) $d\Res_0^1$ is at most $1$-dimensional, so elements $\Res_0^1$ can be shown to have slightly more regularity than expected.

In addition, we are able to show that the SRB measures also satisfy a form of horocyclic invariance under the hypothesis that one weak bundle is smooth, see Lemma \ref{lemma:horocyclic-invariance-SRB-smooth}.

\subsection{Semisimplicity} Having sorted out the resonant forms at zero, the only remaining obstacle to compute the order of vanishing at zero of the Ruelle zeta function using geometric multiplicities is semisimplicity for $\Lie_{X}$ acting on $\Omega_{0}^{1}$. As explained in \cite{Cekic-Paternain-20} this is a subtle issue that could in principle depend on the parametrisation of the flow. For example \cite[Theorem 1.4]{Cekic-Paternain-20} shows that there are time changes of geodesic flows of hyperbolic surfaces that are not semisimple.
Here we propose a conjecture in this direction.

\begin{conj} Let $X$ be a transitive Anosov vector field on a closed 3-manifold $\mc{M}$.
There exists an open and dense set $\mathcal O\subset C^{\infty}(\mc{M}; \mathbb{R}_{>0})$ such that
for $f\in \mathcal O$, the action of $\Lie_{fX}$ on $\Omega_{0}^{1}$ is semisimple.
\end{conj}

A couple of remarks are in order: the conjecture is equivalent to showing that there is just {\it one} time change that is semisimple. This is easily seen using a pairing between resonant and \emph{coresonant spaces} (resonant spaces for $-X$ are decorated with an additional subscript $*$) and noticing that semisimplicity is characterised in terms of the non-degeneracy of this pairing, see Lemma \ref{lemma:semisimple}. Moreover, using the pairing the conjecture is implied by a positive answer to the following intriguing question:

\medskip

\noindent{\bf Question.} Given $u\in \Res_{0}^{1}$ and $u_{*}\in \Res_{0*}^{1}$ with $u\wedge u_{*}=0$, is it true that $u=0$ or $u_{*}=0$?

\medskip

Lemma \ref{lemma:product} provides a positive answer to this question for an interesting class of flows, namely those Anosov flows with smooth weak bundles, see Theorem \ref{thm:QFF} below. Note that semisimplicity is known in a limited number of circumstances, e.g. for Reeb Anosov flows by \cite{Dyatlov-Zworski-17} and for perturbations by \cite{Cekic-Paternain-20}.

\subsection{Examples} Let us illustrate Theorem \ref{thm:general} with some specific examples.

\subsubsection{Suspensions} 
Let us first consider Anosov flows that are topologically orbit equivalent to the suspension of a linear hyperbolic toral automorphism $A\in \mathrm{SL}(2,\mathbb{Z})$. The suspension of $A$ lives on a solvable manifold $\mc{M}_{A}$ with $b_{1}(\mc{M}_{A})=1$. 
We have the following:

\begin{corollary} Let $X$ be any Anosov vector field on $\mc{M}_{A}$. Then $[\omega^{\pm}]\neq 0$, $\Res^{1}_{0}=\Res^{1}_{0*}=\{0\}$
and  $\zeta_{\text{\rm R}}(s)$ has a pole of order 2 at $s=0$, that is, $m_{\mathrm{R}}(0) = -2$.
\label{cor:suspensions}
\end{corollary}

\begin{proof} By \cite[Theorem B]{Plante_81}, any Anosov flow on $\mc{M}_{A}$ is topologically orbit equivalent to the suspension of $A$. Hence, it follows that the trivial homology class does not contain closed orbits of $X$ and thus $X$ is not homologically full. By \cite[Theorem 1]{Sharp_93} we must have $[\omega^{\pm}]\neq 0$. Since $b_{1}(\mc{M}_{A})=1$, Theorem \ref{thm:general} gives that $\Res_{0}^{1}=\Res_{0*}^{1}=\{0\}$ and therefore $\Lie_{X}$ is trivially semisimple on $\Omega_{0}^{1}$. Using \eqref{eq:order-of-vanishing-formula} we deduce that $\zeta_{\text{\rm R}}(s)$ has a pole of order 2 at $s=0$.
\end{proof}

\subsubsection{Quasi-Fuchsian flows and the coupled vortex equations} An interesting class of Anosov flows arises from the data given by a closed Riemann surface and a holomorphic differential of degree $m$ on the surface as we now explain.

Let $(M, g)$ be a closed oriented Riemannian surface and denote by $K_g$ its Gauss curvature. Introduce the \emph{unit sphere bundle} as
\[SM := \{(x, v) \in TM \mid |v|_g = 1\},\]
which carries a natural smooth measure $\Omega$ given locally by the product of the volume form on $M$ and the Lebesgue measure in the spherical fibres. Denote by $X$ the geodesic vector field and by $V$ the generator of (oriented) rotations in the fibres of $SM$. Define the canonical bundle $\mc{K} := (T_{\mathbb{C}}^*M)^{1, 0}$ to be the holomorphic part of the complexified cotangent bundle $T_{\mathbb{C}}^*M$. For $m \in \mathbb{Z}_{\geq 0}$, there is a natural map
\[\pi_m^*: C^\infty(M; \mc{K}^{\otimes m}) \to C^\infty(SM), \quad \pi_m^*T(x, v):= T_x(v, v, \dotsc, v).\]
If $A \in C^\infty(M; \mc{K}^{\otimes m})$, we may set $\lambda = \im(\pi_m^*A)$ and we will be interested in the flow generated by the vector field $F := X + \lambda V$. Say that the pair $(g, A)$ satisfies the \emph{coupled vortex equations} if:
\begin{equation}\label{eq:coupled-vortex-intro}
	\bar{\partial} A = 0, \quad K_g = -1 + (m - 1)|A|_{g}^2.
\end{equation}
If \eqref{eq:coupled-vortex-intro} are satisfied, \cite[Theorem 5.1]{Mettler-Paternain-19} shows that $F$ is Anosov, and by \cite[Remark 5.3]{Mettler-Paternain-19} for any $A$ as above there is a unique metric in the conformal class $[g]$ that solves \eqref{eq:coupled-vortex-intro}. Hence, this class of flows is parametrised by the data $([g],A)$. Moreover, by \cite[Theorem 5.5]{Mettler-Paternain-19} $F$ preserves an absolutely continuous measure if and only if $A=0$, showing that when $A \neq 0$ we are in a genuinely dissipative scenario.

The case $m=2$ is special and in some sense these flows occupy a distinguished place among dissipative Anosov flows, much in the same way as geodesic flows of hyperbolic surfaces are special among contact/volume preserving flows. The main reason they are special is that they have both weak bundles of class $C^\infty$, but the SRB measures are singular as long as $A\neq 0$ and hence they are not algebraic.
Ghys introduced and studied this type of flows in \cite[Th\'eor\`eme B]{Ghys-92}.  His construction produces a smooth orientable Anosov foliation $\phi_{[g_{1}],[g_{2}]}$ on the bundle of positive half-lines tangent to the surface $M$. Here, $g_{i}$ for $i=1,2$ are metrics of constant curvature $-1$ and $[g_{i}]$ are the corresponding points in Teichm\"uller space $\mathcal T(M)$. Any flow parametrising $\phi_{[g_{1}],[g_{2}]}$  has smooth weak stable foliation $C^{\infty}$-conjugate to the weak stable foliation of the constant curvature metric $g_{1}$ and smooth weak unstable foliation $C^{\infty}$-conjugate to the weak unstable foliation of the constant curvature metric $g_{2}$. Moreover, a volume form is preserved if and only if $[g_{1}]=[g_{2}]$. Ghys named these flows \emph{quasi-Fuchsian flows}, because his construction is analogous to the construction of quasi-Fuchsian groups obtained by coupling two Fuchsian groups. Theorem 4.6 in \cite{Ghys_93} ensures that for $m=2$, the flow of $F$ is $C^{\infty}$-orbit equivalent to a quasi-Fuchsian flow, so we shall call such an $F$ also a quasi-Fuchsian vector field. As we noted above these are parametrised by $T^*\mathcal T(M)$ and in forthcoming work \cite{CP_25} we will show that they exhaust all possible $\phi_{[g_{1}],[g_{2}]}$. Working with $F$ instead of $\phi_{[g_{1}],[g_{2}]}$ has many advantages for our purposes as it gives us access to the vertical Fourier analysis introduced by Guillemin-Kazhdan \cite{Guillemin-Kazhdan-80} which in turn will allow us to compute all resonant 1-forms at zero and the helicity of the SRB measures. As far as we are aware this is the first explicit helicity calculation in a dissipative context.

\begin{theorem}\label{thm:QFF}
	Let $F$ be a quasi-Fuchsian vector field. Then $m_{1, 0} = b_1(M)$, $[\omega^{\pm}] = 0$, and the helicity is given by
	\[\mc{H}(F) = \frac{1 + \frac{1}{2} \e^+ (F)}{2\pi \vol_g(M) + \int_{SM} a^2\, \Omega}.\]
	The function $a$ is the unique H\"older solution to the equation
	\[Fa + \Big(1+\frac{1}{2}V\lambda\Big)a = -\lambda\]
	and $\e^+ (F)$ is the entropy production of the SRB measure which is given explicitly by
	\[\e^+ (F)=-\int_{SM}V\lambda\,\SRB.\]
	Moreover, for an open and dense set of $f \in C^\infty(SM; \mathbb{R}_{>0})$, the action of $\Lie_{fF}$ on $\Omega_0^1$ is semisimple and so for the time-changed vector field $fF$, we have $m_{\mathrm{R}}(0) = -\chi(M)$.
\end{theorem}

Curiously, we do not know if the action of $\Lie_{F}$ on $\Omega_0^1$ is semisimple, except for $A$ small.

\subsubsection{Flows with $[\omega^{+}]=0$ and $[\omega^{-}]\neq 0$.} The next interesting class of examples that arises in relation to Theorem \ref{thm:general} are flows in the second and third columns of the table. In fact, if $X$ is an example with $[\omega^{+}]=0$ and $[\omega^{-}]\neq 0$, then $-X$ is an example with $[\omega^{+}]\neq 0$ and $[\omega^{-}]=0$. They have the special feature that $\Res^{1}_{0}$ contains non-closed 1-forms, something unseen in the volume preserving case. That such examples must exist can now be easily inferred from the perturbation picture described in \S \ref{subsec:perturbation}, however we prefer to go one step further and give explicit smooth examples with the additional feature that the weak stable bundle is $C^{\infty}$. This is carried out in Proposition \ref{prop:newexamples}.
Incidentally, for these examples we show that semisimplicity for $\Lie_{X}$ acting on $\Omega^1$ holds but fails for $\Lie_{-X}$.
Also semisimplicity holds for $\Lie_{\pm X}$ acting on $\Omega_{0}^{1}$.

\subsubsection{Homologically full flows with $[\omega^{\pm}]\neq 0$} The case of suspensions studied in Corollary \ref{cor:suspensions} certainly provides examples of flows with $[\omega^{\pm}]\neq 0$. However, here we provide such examples which in addition are \emph{homologically full}, i.e. every homology class contains a closed orbit. The examples live on the unit tangent bundle of a surface with negative curvature (and hence they are topologically orbit equivalent to a geodesic flow) and are in fact \emph{Gaussian thermostats} or \emph{W-flows} as introduced in \cite{Wojtkowski_00} (they are reparametrisations of the geodesics of a Weyl connection).

Let $(M,g)$ be a closed oriented surface of negative Gaussian curvature and let $\rho$ be a smooth closed 1-form. Consider $\lambda := \pi_1^*\rho \in C^\infty(SM)$ defined by lifting the $1$-form $\rho$ via $\pi_1^*\rho(x,v) = \rho_{x}(v)$, and $F := X + \lambda V$. By \cite[Theorem 5.2]{Wojtkowski_00} the flow of $F$ is Anosov. Moreover, by \cite{Dairbekov-Paternain-07} the flow  preserves an absolutely continuous measure if and only if $\rho = 0$. Observe that these flows are reversible, i.e. the flip $(x,v)\mapsto (x,-v)$ conjugates the flow of $F$ with the flow of $-F$.

\begin{theorem} 
	Let $F = X + \pi_1^*\rho\,V$ be the generator of a Gaussian thermostat with $\rho$ \emph{harmonic} and non-zero. Then $[\omega^{+}]=-[\omega^{-}]\neq 0$ and thus $m_{1, 0} = b_1(M) - 1$. Moreover, for $\rho$ small enough, the action of $\Lie_{F}$ on $\Omega_{0}^{1}$ is semisimple and the order of vanishing of the Ruelle zeta function of $F$ is $m_{\mathrm{R}}(0) = -\chi(M)-1$.	
\label{thm:beta}	
	
\end{theorem}

A Gaussian thermostat as in Theorem \ref{thm:beta} may also be seen as the geodesic flow of the unique affine connection compatible with $g$ and with torsion $\mathcal T(Y,Z)=-\star \rho(Y)Z+\star\rho(Z)Y$ \cite{PW_08}.
Interestingly enough, we will also show that it is possible to produce flows with similar features when $\rho$ is \emph{exact} and non-zero, see Proposition \ref{prop:betaexact}.

\subsection{Outline of the paper}
\begin{itemize}

\item In Section \ref{sec:preliminaries} we recall some preliminary notions from microlocal analysis and resonances for Anosov flows. 
\item In Section \ref{sec:general} we study general $3$-dimensional Anosov flows that do not necessarily preserve a smooth measure and we determine the dimension of the resonant spaces of Lie derivatives at zero in terms of the helicity and the winding cycles of the associated SRB measures, proving Theorem \ref{thm:general}. 
\item In Section \ref{sec:horocyclic} we show horocyclic invariance of resonant states and in Section \ref{sec:perturbation} we study the local manifold structure of $\mathcal W^{\pm}$.
\item Section \ref{section:onesmooth} discusses examples of Anosov flows with smooth weak stable bundle, $[\omega^{+}]=0$ and $[\omega^{-}]\neq 0$.
\item Section \ref{sec:thermo} introduces thermostats and provides a proof of Theorem \ref{thm:beta}.
\item Section \ref{sec:qhd} is devoted to quasi-Fuchsian flows arising from the coupled vortex equations determined by a pair $([g],A)$, where $A$ is a quadratic holomorphic differential. Theorem \ref{thm:QFF} is proved here.
\item Section \ref{section:helicity} establishes the relation between the helicity and the linking form given by Theorem \ref{thm:helicity-formula}.

\item Finally, there are two appendices: Appendix \ref{app:A} deduces some standard properties of the ladder operators $\eta_\pm$, while Appendix \ref{app:B} studies the behaviour of the geodesic vector field under re-scaling.

\end{itemize}

\medskip

\noindent {\bf Acknowledgements.} 
MC has received funding from an Ambizione grant (project number 201806) from the Swiss National Science Foundation. During the course of writing this paper, he was also supported by the European Research Council (ERC) under the European Union’s Horizon 2020 research and innovation programme (grant agreement No. 725967). GPP thanks the University of Z\"urich for hospitality while this work was in progress; he was supported by NSF grant DMS-2347868 during completion of this work. The authors are grateful to Thibault Lefeuvre for pointing out the reference \cite{GuedesBonthonneau-Lefeuvre-19}, to Sebasti\'an Mu\~noz-Thon for spotting several typos, and to the referee for numerous comments and suggestions that considerably improved the presentation.

\section{Preliminaries}\label{sec:preliminaries}

Let $X$ generate an Anosov flow $\varphi_t$ on a closed manifold $\mc{M}$. Denote by $\Omega^k$ the vector bundle of differential $k$-forms, and by $\Omega_0^k = \Omega^k \cap \ker \iota_X$ the bundle of $k$-forms that are in the kernel of the contraction $\iota_X$ with the vector field $X$.

Let $\mc{E}$ be a vector bundle over $\mc{M}$. Denote by $\mc{D}'(\mc{M}; \mc{E})$ the space of distributional sections of $\mc{E}$. Given a closed conic set $\Gamma \subset T^*M$, denote
\[\mc{D}'_{\Gamma}(\mc{M}; \mc{E}) = \{u \in \mc{D}'(\mc{M}; \mc{E}) \mid \WF(u) \subset \Gamma\}.\]
For $m \in \mathbb{R}$, denote the space of \emph{pseudodifferential operators} of order $m$ acting on sections of $\E$ over $\M$ by $\Psi^m(\M; \E)$; when this action is on functions, write simply $\Psi^m(\M)$. For background on distribution theory and the wavefront set, and pseudodifferential operators, see \cite[Chapters II and VIII]{Hormander-90}, and \cite[Chapter XVIII]{Hormander-07-III} or \cite[Appendix E]{Dyatlov-Zworski-19}, respectively.

Let us note that the treatment of the Pollicott-Ruelle resonances as eigenvalues of the flow generator 	in anisotropic Banach or Hilbert spaces has been developed in the past twenty years by Baladi \cite{Baladi-05}, Baladi-Tsujii \cite{Baladi-Tsujii-07}, Blank-Keller-Liverani \cite{Blank-Keller-Liverani-02}, Gou\"ezel-Liverani \cite{Gouezel-Liverani-06}, and Liverani \cite{Liverani-04}. In this paper (see \S \ref{ssec:resolvent} below) we take the viewpoint of microlocal methods for dynamical resonances, introduced by Faure-Sj\"ostrand \cite{Faure-Sjostrand-11} and Dyatlov-Zworski \cite{Dyatlov-Zworski-16}.

\subsection{The resolvent and Pollicott-Ruelle resonances}\label{ssec:resolvent}

For $s \in \mathbb{C}$ with sufficiently large real part, we may define the resolvent $R_k^+(s) = (\Lie_X + s)^{-1}$ acting on $L^2(\mc{M}; \Omega^k)$ by the expression:
\begin{equation}\label{eq:resolvent-identity}
	(\Lie_X + s)^{-1} = \int_0^\infty e^{-st} \varphi_{-t}^*\, dt: L^2(\mc{M}; \Omega^k) \to L^2(\mc{M}; \Omega^k),
\end{equation}
where $\varphi_{-t}^*$ denotes the pullback operation by $\varphi_{-t}$. It has been shown by \cite{Faure-Sjostrand-11} and \cite{Dyatlov-Zworski-16} that $R_k^+(s)$ admits a meromorphic extension to the entire complex plane as a map $R_k^+(s): C^\infty(\mc{M}; \Omega^k) \to \mc{D}'(\mc{M}; \Omega^k)$. At a pole $s_0 \in \mathbb{C}$, the resolvent admits a Laurent expansion (see \cite{Dyatlov-Zworski-16}):
\begin{equation}\label{eq:laurent}
	R_k^+(s) = R^{+, H}_k(s; s_0) + \sum_{i = 1}^{J_k(s_0)}\frac{(-(\Lie_X + s_0))^{i - 1}\Pi_k^+(s_0)}{(s - s_0)^{i}}.
\end{equation}
Here $\Pi_k^+(s_0)$ is a finite rank operator that extends to a map $\Pi_k^+(s_0): \mc{D}'_{E_u^*}(\mc{M}; \Omega^k) \to \mc{D}'_{E_u^*}(\mc{M}; \Omega^k)$ and satisfies $\Pi_k^+(s_0)^2 = \Pi_k^+(s_0)$. Moreover, $J_k(s_0) \geq 1$ is the least integer such that we have $(\Lie_X + s_0)^{J_k(s_0)} \Pi_k^+(s_0) = 0$ and $R^{+, H}_k(s; s_0)$ is a holomorphic function defined near $s = s_0$. The map $R_k^{+, H}(s; s_0)$ also extends to a map $R_k^{+, H}(s; s_0): \mc{D}'_{E_u^*}(\M; \Omega^k) \to \mc{D}'_{E_u^*}(\M; \Omega^k)$ by \cite{Dyatlov-Zworski-16}. We call the poles of $R_k^+(s)$ \emph{resonances}.

Given $s_0 \in \mathbb{C}$, and $\ell \in \mathbb{Z}_{\geq 1}$ we introduce (similarly as in the introduction)
\[\Res^{k, \ell} (s_0) = \{u \in \mc{D}'_{E_u^*}(\mc{M}; \Omega^k) \mid\,(\Lie_X + s_0)^{\ell}u = 0\}.\]
Denote $\Res^k(s_0) := \Res^{k, 1}(s_0)$ and $\Res^{k, \infty}(s_0) := \cup_{\ell \geq 1} \Res^{k, \ell}(s_0)$, and call the elements of $\Res^{k}(s_0)$ and $\Res^{k, \infty}(s_0)$ \emph{resonant states} and \emph{generalised resonant states}, respectively. It can be checked that $s_0 \in \mathbb{C}$ is a pole of $R_k^+(s)$ if and only if $\Res^k(s_0)$ is non-trivial, and that $\ran \Pi^+_k(s_0) = \Res^{k, \infty}(s_0)$. By the residue theorem, we have
\begin{equation}\label{eq:projector+}
	\Pi_k^+ = \frac{1}{2\pi i} \oint_{s_0} R_k^+(s)\, ds, 
\end{equation}
where $\oint_{s_0}$ denotes integration along a small contour around $s = s_0$. Say that \emph{semisimplicity} holds for the action of $\Lie_X$ on $\Omega^k$ at $s_0$ if $\Res^{k, \infty}(s_0) = \Res^k(s_0)$.

Similarly, the resolvent for the backwards flow $R_k^-(s) = (-\Lie_X + s)^{-1}$ is well-defined on $L^2(\mc{M}; \Omega^k)$ for all $s \in \mathbb{C}$ with large real part. It also admits a meromorphic extension as a map $R_k^-(s): C^\infty(\mc{M}; \Omega^k) \to \mc{D}'(\mc{M}; \Omega^k)$. For $s_0 \in \mathbb{C}$, call the resonant states of $-\Lie_X + s_0$ \emph{coresonant states} and decorate the analogous spaces with an additional subscript $*$: $\Res^{k, \ell}_*(s_0)$, $\Res^{k, \infty}_*(s_0)$, $\Res^{k}_*(s_0)$. The projector at $s = s_0$ is denoted by $\Pi_k^-(s_0)$.

One may define analogous objects for the action of $\Lie_X$ on the bundle $\Omega_0^k$. In that case, denote the resulting objects with an additional zero subscript:
\[\Res^{k, \ell}_{0(*)}(s_0), \Res^{k, \infty}_{0(*)}(s_0), \Res^{k}_{0(*)}(s_0), \Pi_{k, 0}(s_0)^{\pm}.\]
For $s_0 = 0$, in order to simplify the notation we omit the $0$ in parentheses and denote the resulting objects as
\[\Res^{k, \ell}_{0(*)}, \Res^{k, \infty}_{0(*)}, \Res^{k}_{0(*)}, \Pi_k^\pm, \Pi_{k, 0}^{\pm}.\]
Since $\Pi_k^+$ is given by the contour integral \eqref{eq:projector+} (and similarly for $\Pi_k^-$), and the Lie derivative $\Lie_X$ commutes with the exterior differential $d$ and the contraction $\iota_X$, we have on the domain of $\Pi_k^\pm$:
\begin{equation}\label{eq:commute-projector}
	d \Pi_k^\pm = \Pi_{k + 1}^\pm d, \quad \Pi_{k - 1}^\pm \iota_X = \iota_X \Pi_{k}^\pm. 
\end{equation}

\subsection{The pairings}\label{ssec:pairing}
Let $\alpha \in C^\infty(\mc{M}; \Omega^1)$ be an arbitrary $1$-form such that $\alpha(X) = 1$. Introduce the following non-degenerate bilinear pairing $\llangle{\bullet, \bullet}\rrangle$ (similar to the contact case studied in \cite{Cekic-Dyatlov-Delarue-Paternain-22}):
\[u \in C^\infty(\mc{M}; \Omega^{k}_0),\,\, u_* \in C^\infty(\mc{M}; \Omega^{2 - k}_0),\,\, \llangle{u, u_*}\rrangle := \int_{\mc{M}} \alpha \wedge u \wedge u_*.\]
Note that the pairing does not depend on the choice of $\alpha$, and that by the wavefront set calculus it extends to $\mc{D}'_{E_u^*}(\M; \Omega_0^k) \times \mc{D}'_{E_s^*}(\M; \Omega_0^{2 - k})$.

If $A: C^\infty(\mc{M}; \Omega^{k}_{0}) \to \mc{D}'(\mc{M}; \Omega^{k}_0)$ is a continuous operator, denote by $A^T: C^\infty(\mc{M}; \Omega^{2 - k}_{0}) \to \mc{D}'(\mc{M}; \Omega^{2 - k}_0)$ its transpose with respect to $\llangle{\bullet, \bullet}\rrangle$. In particular, observe it holds that:
\begin{equation}\label{eq:transpose-easy}
	(\Lie_X + s)^T = -\Lie_X + s, \qquad s \in \mathbb{C}.
\end{equation}
It follows that, using \eqref{eq:projector+} and meromorphic continuation from $s \in \mathbb{C}$ with large real part:
\begin{equation}\label{eq:proj-tranpose}
	(\Pi_{k, 0}^+)^T = \Pi_{2 - k, 0}^-.
\end{equation}
More precisely, from \eqref{eq:transpose-easy}, for large $\re (s)$ we have $\llangle{(\Lie_X + s)^{-1}u, u_*}\rrangle = \llangle{u, (-\Lie_X + s)^{-1}u_*}\rrangle$, and by meromorphic continuation for all $s \in \mathbb{C}$. Using \eqref{eq:projector+} proves the claim. Moreover we can relate semisimplicity for $\Lie_X$ on $\Omega_0^k$ with the pairing $\llangle{\bullet, \bullet}\rrangle$:

\begin{lemma}\label{lemma:semisimple} The Lie derivative $\Lie_X$ acting on $\Omega^1_0$ is semisimple at $s = 0$ if and only if $\llangle \bullet,\bullet\rrangle$ is non-degenerate on $\Res_0^1 \times \Res^{1}_{0*}$.
\end{lemma}
\begin{proof} 
	Assume first that the pairing is non-degenerate and $\Lie_X^2 u = 0$ for some $u \in \mc{D}'_{E_u^*}(\mc{M}; \Omega_0^1)$. Set $v := \Lie_X u$. Then $v \in \Res_0^1$ and for any $u_* \in \Res_{0*}^1$ we have
	\[\llangle{v, u_*}\rrangle = \llangle{\Lie_X u, u_*}\rrangle = -\llangle{u, \Lie_Xu_*}\rrangle = 0.\]
	By non-degeneracy $\Lie_X u = v = 0$ and the semisimplicity condition is verified.
	
	For the other direction, assume that the semisimplicity condition holds. First we check that semisimplicity holds also for coresonant states. Indeed, since by assumption $\Lie_X \Pi^+_{1, 0} = \Pi^+_{1, 0} \Lie_X \equiv 0$, transposing we get $\Lie_X \Pi_{1, 0}^- \equiv 0$, which proves the claim. 
	
	Next, assume $u \in \Res_0^1$ satisfies $\llangle{u, u_*}\rrangle = 0$ for all $u_* \in \Res_{0*}^1$. Take any $\varphi \in C^\infty(\mc{M}; \Omega^1_0)$ and compute:
	\[0 = \llangle{u, \Pi^{-}_{1, 0}\varphi}\rrangle = \llangle{u, \varphi}\rrangle,\]
	where in the second equality we used \eqref{eq:proj-tranpose} and $\Pi^{+}_{1, 0}u = u$. By the non-degeneracy of $\llangle{\bullet, \bullet}\rrangle$, we conclude that $u \equiv 0$, proving the result.
\end{proof} 

There is another non-degenerate pairing on $C^\infty(\M; \Omega^k) \times C^\infty(\M; \Omega^{3 - k})$, with an extension to $\mc{D}'_{E_u^*}(\M; \Omega^k) \times \mc{D}'_{E_s^*}(\M; \Omega^{3 - k})$ given by
\[\langle{u, u_*}\rangle := \int_{\M} u \wedge u_*.\]
Note that similarly to $\llangle{\bullet, \bullet}\rrangle$, with respect to this pairing the transpose satisfies 
\begin{equation}\label{eq:proj-transpose-2}
	(\Pi_k^+)^T = \Pi_{3 - k}^-.
\end{equation}
In fact, we have:
\begin{prop}\label{prop:non-degenerate}
 The pairing $\langle{\bullet, \bullet}\rangle$ is non-degenerate on $\Res^{k, \infty} \times \Res^{3 - k, \infty}_*$. Moreover, the action of $\Lie_X$ on $\Omega^k$ is semisimple if and only if the pairing $\langle{\bullet, \bullet}\rangle$ is non-degenerate on $\Res^{k} \times \Res^{3 - k}_*$.
\end{prop}
\begin{proof}
	It suffices to consider some $u \in \Res^{k, \infty}$ such that 
\[\langle{u, u_*}\rangle = 0, \quad \forall u_* \in \Res_*^{k, \infty},\]
and show $u = 0$. In particular, for any $\varphi \in C^\infty(\M; \Omega^{3-k})$ it holds that
\[0 = \langle{u, \Pi_{3-k}^- \varphi}\rangle = \langle{u, \varphi}\rangle,\]
where we used \eqref{eq:proj-transpose-2} and $\Pi_k^+ u = u$ in the second equality. Therefore $u = 0$, by the fact that the pairing on $C^\infty(\M; \Omega^k) \times C^\infty(\M; \Omega^{3 - k})$ is non-degenerate.

The second claim is proved analogously to Lemma \ref{lemma:semisimple}, using \eqref{eq:proj-transpose-2}.
\end{proof}

\subsection{Mapping properties of the resolvent}

For $s \in \mathbb{R}$, denote by $C_*^s(\mc{M})$ the H\"older-Zygmund space of regularity $s$ on $\M$. See \cite[Section 2]{GuedesBonthonneau-Lefeuvre-20} or \cite[Section 1.8]{Taylor-III-96edition} for a definition and background. In particular, a pseudodifferential operator of order $m$ is a bounded map $C_*^s(\M) \to C_*^{s - m}(\M)$. For $s$ integer and non-negative, these spaces agree with the usual H\"older spaces, and satisfy $C^k(\mc{M}) \subsetneqq C^k_*(\mc{M})$ for $k \in \mathbb{Z}_{\geq 0}$. Let us denote $C_*^{s-}(\M) := \cap_{t < s} C_*^{t}(\M)$. We need the following statement about the mapping properties of the resolvent on H\"older-Zygmund spaces:

\begin{prop}\label{prop:mapping}
	Let $X$ generate an Anosov flow. Then:
	\begin{enumerate}[itemsep=5pt]
	\item[1.] Let $r \in C^{\alpha}(\mc{M})$ for some $\alpha > 0$, such that $\liminf_{t \to \infty}\inf_{x\in\mc{M}} \frac{1}{t} \int_0^t \varphi_{-p}^*r(x)\, dp > \nu$ for some $\nu > 0$. Then there exist $\delta, \delta_1 > 0$, such that $(X + r + s)^{-1}$ exists and is holomorphic as a map $C^0(\M) \to C^0(\M)$ and $C^{\delta}(\M) \to C^{\delta}(\M)$ in the region $\re(s) > -\delta_1 > - \nu$.
	\item[2.] There exists $C > 0$, such that for any $\alpha > 0$ the resolvent $R_0^+(s) = (X + s)^{-1}: C^\alpha_*(\mc{M}) \to C^{-\alpha}_*(\mc{M})$ is meromorphic for $\re(s) > -C\alpha$. In particular, the holomorphic part of the resolvent at zero $R^{+,H}_0: C^\infty(\mc{M}) \to \mathcal{D}'(\mc{M})$ extends as an operator $C^\alpha_*(\mc{M}) \to C^{-\alpha}_*(\mc{M})$ for all $\alpha > 0$.
	\end{enumerate}
\end{prop}
\begin{proof}
	For Item 1, it suffices to use the resolvent formula
	\begin{equation}\label{eq:resolvent-formula}	
		(X + r + s)^{-1} = \int_0^\infty e^{-st} e^{-\int_0^t \varphi_{-p}^*r\, dp} \varphi_{-t}^*\, dt.
	\end{equation}
	The conclusion on the mapping properties as a map $C^0(\M) \to C^0(\M)$ follows directly from the formula, and the case of $C^\delta(\M) \to C^\delta(\M)$ follows for small $\delta > 0$ by using interpolation between $C^0(\M)$ and $C^1(\M)$.
	
	Item 2 follows from the recent result \cite[Theorem 4.2]{GuedesBonthonneau-Lefeuvre-19} which was proved in the setting of the unit tangent bundle of a cusp manifold, but the proof carries over to our setting. Let us sketch the proof for completeness. Pick $A \in \Psi^0(\M)$ such that $A = \mathbbm{1}$ on a conical neighbourhood $\mc{C}_u^0$ of $E_u^*$ and $\WF(A)$ is contained in a larger conical neighbourhood $\mc{C}_u^1$ which does not intersect $E_s^* \oplus E_0^*$ (by definition, $E_0^*$ is the annihilator of $E_u \oplus E_s$). For $\alpha > 0$ introduce the anisotropic norm
	\[\|u\|_{\star} := \|Au\|_{C_*^{-\alpha}} + \|(\mathbbm{1} - A)u\|_{C_*^\alpha}.\]
	Choose $T > 0$ such that for all $t \geq T$ we have $\Phi_t(\mc{C}_u^1) \subset \mc{C}_u^0$, where $\Phi_t(x, \xi) = (\varphi_t x, \xi \circ (d\varphi_{t}(x))^{-1})$ is the Hamiltonian lift on $\varphi_t$ to $T^*\M$ (this choice is possible by the Anosov condition). Finally, as in \cite[equation (4.7)]{GuedesBonthonneau-Lefeuvre-19} introduce the norm:
	\[\|u\|_{\mathbf{C}^\alpha} := \int_0^T \|\varphi_{-t}^*u\|_{\star}\, dt,\]
	which defines a Banach space $\mathbf{C}^\alpha(\M)$. It is straightforward to show the continuous inclusions 
	\[C_*^{\alpha}(\M) \subset \mathbf{C}^\alpha(\M) \subset C_*^{-\alpha}(\M),\]
	so our conclusion will follow if we can show that $R_0^+(s)$ extends as a family of meromorphic operators acting on $\mathbf{C}^\alpha(\M)$ in the required region. As in \cite[Lemma 4.3]{GuedesBonthonneau-Lefeuvre-19}, it is possible to show there is a constant $C_0 > 0$ such that the propagator $\varphi_{-t}^*$ satisfies the bound
	\begin{equation}\label{eq:propagator-estimate}
		\|\varphi_{-t}^*u\|_{\mathbf{C}^\alpha} \leq C_0e^{C_0t\alpha} \|u\|_{\mathbf{C}^\alpha}, \quad \alpha > 0.
	\end{equation}
	Next, by the radial sink/source estimates in H\"older-Zygmund spaces proved in \cite{GuedesBonthonneau-Lefeuvre-20}, following the proof of \cite[Proposition 4.3]{GuedesBonthonneau-Lefeuvre-19} and using \eqref{eq:propagator-estimate}, it is possible to show that there exists a $C > 0$ such that for $\re(s) > -C\alpha$:
	\[\|u\|_{\mathbf{C}^\alpha} \leq C(\|(X + s)u\|_{\mathbf{C}^\alpha} + \|Ku\|_{\mathbf{C}^\alpha}), \qquad u \in C^\infty(\M),\]
	where $K$ is a smoothing operator. From this estimate, it follows that $X + s$ is a semi-Fredholm operator (it has finite dimensional kernel and closed range) with domain $\mc{D}(X) = \{u \in \mathbf{C}^\alpha(\M) \mid Xu \in \mathbf{C}^\alpha(\M)\}$ and hence $R_0^+(s): \mathbf{C}^\alpha(\M) \to \mathbf{C}^\alpha(\M)$ admits a meromorphic continuation by the analytic Fredholm theorem, completing the proof.
\end{proof}

\subsection{Geometry of surfaces}\label{ssec:geometry-surfaces}
For details about the content of this section, see \cite{Guillemin-Kazhdan-80, Merry-Paternain-11, Singer-Thorpe-76}. Let $(M, g)$ be an oriented closed surface and let $X$ be the geodesic vector field on the unit sphere bundle $SM$. Denote by $V$ the generator of the fibrewise rotation action and by $H$ the horizontal vector field. Let $K_g$ be the Gauss curvature of $(M, g)$. Then:
\begin{align}\label{eq:surface-geometry}
\begin{split}
	[H, V] &= X,\\
	[V, X] &= H,\\
	[X, H] &= KV.
\end{split}
\end{align}
Dually to the global frame $\{X, H, V\}$ define a global frame of $1$-forms $\{\alpha, \beta, \psi\}$. Set $\Omega := \alpha \wedge \psi \wedge \beta = \alpha \wedge d\alpha$ to be the canonical volume form on $SM$. For $k \in \mathbb{Z}$ and $x \in M$, introduce
\[\mho_k(x) = \{f \in C^\infty(S_xM) \mid V f = ik f\}.\]
It is straightforward to see that $\mho_k \to M$ has the structure of a smooth line bundle; a smooth section of this bundle may be seen as a smooth function on $SM$. In fact if $\mc{K} := (T_{\mathbb{C}}^*M)^{1, 0}$ is the associated \emph{canonical} bundle and $\mc{K}^{-1} := (T_{\mathbb{C}}^*M)^{0, 1}$, then for any $k \in \mathbb{Z}$, $\mho_k$ may be identified with the tensor power $\mc{K}^{\otimes k}$. For any $k \in \mathbb{Z}_{\geq 0}$ and $x \in M$, there are natural maps
\[\pi_{k}^*: (\otimes_S^k T_{\mathbb{C}}^*M)_x \to C^\infty(S_xM), \quad \pi_k^*T(v) = T(v, \dotsc , v),\]
where $\otimes_S^k T_{\mathbb{C}}^*M$ denotes the bundle of symmetric tensors of degree $k$. Denote by $\otimes_S^k T_{\mathbb{C}}^*M|_{0-\mathrm{tr}}$ the sub-bundle of \emph{trace-free} symmetric tensors, where a symmetric $k$-tensor $T$ at $x$ is trace free if for an orthonormal basis $(\mathbf{e}_i)_{i = 1}^2$ of $T_xM$, $T(\mathbf{e}_1, \mathbf{e}_1, \dotso) + T(\mathbf{e}_2, \mathbf{e}_2, \dotso) = 0$. It is straightforward to see that there are vector bundle isomorphisms
\[\forall k \in \mathbb{Z}_{> 0}, \quad \otimes_S^k T_{\mathbb{C}}^*M|_{0-\mathrm{tr}} = \mc{K}^{\otimes k} \oplus \mc{K}^{-\otimes k}; \quad \forall k \in \mathbb{Z}, \quad \pi_{|k|}^*: \mc{K}^{\otimes k} \xrightarrow{\cong} \mho_k.\]
For $k \in \mathbb{Z}$, set $H_k := C^\infty(M; \mho_k) \subset C^\infty(SM)$. Then $\pi_{|k|}^*$ identifies smooth sections of $\mc{K}^k$ with $H_k$.

By decomposing into eigenstates of $V$, it is straightforward to see that:
\[L^2(SM) = \bigoplus_{k \in \mathbb{Z}} L^2(M; \mho_k),\]
and for $f \in L^2(SM)$, let us write $f = \sum_{k \in \mathbb{Z}} f_k$ for the corresponding decomposition. We refer to $f_k$ as the \emph{Fourier mode of degree $k$} of $f$. Define the \emph{degree} of $f$ to be the maximal $k \in \mathbb{Z}_{\geq 0} \cup \{\infty\}$ such that $f_k \neq 0$ or $f_{-k} \neq 0$. Introduce the \emph{raising/lowering} operators $\eta_\pm$:
\[\eta_+ = \frac{X - iH}{2}, \quad \eta_- = \frac{X + iH}{2}.\]
 By \eqref{eq:surface-geometry} it follows that
\[[\eta_\pm, V] = \mp i \eta_\pm,\]
and therefore $\eta_\pm: H_{k} \to H_{k \pm 1}$. Another consequence of \eqref{eq:surface-geometry} is the following commutation relation:
\begin{equation}\label{eq:eta+eta-commutator}
	[\eta_+, \eta_-] = \frac{i}{2} K_gV.
\end{equation}

The Hodge star operator $\star: \Omega^1 \xrightarrow{\cong} \Omega^1$ is given by (oriented) rotation by $\frac{\pi}{2}$, and $\star: \Omega^0 \xrightarrow{\cong} \Omega^2$ by $\star(\mathbf{1}) = d\vol_g$, the volume form of $g$, where $\mathbf{1}$ is the constant function equal to $1$ everywhere. It can be checked that the co-differential on $1$-forms is given by $d^* = -\star d \star$. Moreover, if $X_-$ is defined on $H_{-1} \oplus H_1$ by $X_- (f_{-1} + f_1) = \eta_+ f_{-1} + \eta_- f_1$, by Proposition \ref{prop:X-} we have:
\begin{equation}\label{eq:X-}
	X_- \pi_1^* \gamma = -\frac{1}{2} \pi_0^* d^* \gamma, \quad \forall \gamma \in C^\infty(M; \Omega^1).
\end{equation}
Finally, it is standard and follows from the expressions for $X$ and $H$ in local isothermal coordinates (see \eqref{eq:isothermal-X-H}) that:
\begin{equation}\label{eq:X-simple}
		X \pi_0^* f = \pi_1^*(d f),\quad H \pi_0^* f = -\pi_1^*(\star d f), \quad \forall f \in C^\infty(M),
\end{equation}
and also similarly that:
\[V\pi_1^*\gamma = -\pi_1^*(\star \gamma), \quad \iota_H\pi^*(\star \gamma) = \pi_1^*\gamma, \quad \forall \gamma \in C^\infty(M; \Omega^1).\]

\subsection{Decomposition of differential forms}

In this section we give a standard result about decomposition of differential forms, similar to \cite[Lemma 2.1]{Dyatlov-Zworski-17}. Denote by $d^*$ the co-differential on $(\mc{M}, g)$, where $g$ is an arbitrary Riemannian metric on $\mc{M}$. Denote by $\Delta_k := d^* d + d d^*$ the \emph{Hodge Laplacian} on $\Omega^k$.

\begin{lemma}\label{lemma:hodge-decomposition}
	Let $\Gamma \subset T^*\mc{M} \setminus 0$ be a closed conic set. Let $u \in \mc{D}'_{\Gamma}(\mc{M}; \Omega^k) \cap C_*^s(\mc{M}; \Omega^k)$ for some $s \in \mathbb{R}$. Then there exist $v \in \mc{D}'_{\Gamma}(\mc{M}; \Omega^{k - 1}) \cap C_*^{s + 1}(\mc{M}; \Omega^{k-1})$, $w \in \mc{D}'_{\Gamma}(\mc{M}; \Omega^{k + 1}) \cap C_*^{s + 2}(\mc{M}; \Omega^{k+1})$, and $\theta \in C^\infty(\mc{M}; \Omega^{k})$, with $d^* v = 0$, $d w = 0$, such that
	\[u = dv + d^* w + \theta.\]
Moreover, if $du = 0$ then we may assume $w = 0$, thus we have $d\theta = 0$, and the cohomology class $[\theta] \in H^k(\M)$ is well-defined.
\end{lemma}
\begin{proof}
	Let $Q_k \in \Psi^{-2}(\mc{M}; \Omega^k)$ be an elliptic parametrix of $\Delta_k$, i.e. so that
	\[Q_k \Delta_k - \id,\, \Delta_k Q_k - \id \in \Psi^{- \infty}(\mc{M}; \Omega^k).\]
	Setting $\theta := u - \Delta_k Q_k u$, $v := d^* Q_k u$, and $w := d Q_k u$, concludes the proof of the first claim.
	
	For the second claim, notice that using $[\Delta_k, d] = 0$ and the properties of the parametrix
		\[
			dQ_k \equiv Q_k \Delta_k d Q_k \equiv Q_k d \Delta_k Q_k \equiv Q_k d \mod \Psi^{-\infty}(\M).
		\]
		Thus, using the proof of the first claim, we may assume that $w = 0$. If we write $u = dv_i + \theta_i$ for some $v_i, \theta_i$ as above for $i = 1, 2$, then $\theta_1 - \theta_2 = d(v_2 - v_1) \in C^\infty(\M; \Omega^k)$. Using $d^*(v_2 - v_1) = 0$, elliptic regularity implies $v_2 - v_1$ is smooth and thus $[\theta]$ is well-defined, finishing the proof.
\end{proof}

\subsection{SRB measures and entropy production}\label{ssec:SRB-entropy} A probability measure $\SRB$ on $\mathcal{M}$ is called a Sinai-Ruelle-Bowen (SRB) measure if there is a probability volume form $\Omega$ on $\mathcal{M}$ such that (see \cite{Young-02})
\begin{equation}
	\SRB = \lim_{T \to \infty} \frac{1}{T} \int_0^T \varphi_{-t}^*\Omega\, dt,
\end{equation}
in the weak limit sense (i.e. when paired with any smooth function). It is known that for \emph{transitive} Anosov flows (recall that by definition transitive flows have a dense orbit), there exists a unique SRB measure, see for instance \cite{Butterley-Liverani-07}. This measure may be seen as a resonant state for the Lie derivative acting on differential forms of top degree. For a microlocal point of view, see \cite[Theorem 3]{GuedesBonthonneau-Guillarmou-Hilgert-Weich-20} who characterise SRB measures as invariant measures whose wavefront set lies in $E_u^*$. In this paper, we will use the latter description. Similarly, there is a unique SRB measure for the flow $-X$ that we denote by $\SRBs$ (it has wavefront set in $E_s^*$).

Next we examine the transformation law for SRB measures under a (positive) time-change $f \in C^\infty(\M)$. Assume $X$ is transitive with the unique SRB measure $\SRB$. Then the unique SRB measure $\SRB(fX)$ of $fX$ is given by the formula:
\[\SRB(fX) = \frac{f^{-1}\SRB}{\int_{\M} f^{-1}\SRB}.\]

Given an arbitrary smooth volume form $\Omega$ on $\M$, the \emph{entropy production} of the SRB measure is given by :
\[\e^{+}(X) := -\int_{\M} \divv_\Omega(X)\, \Omega_{\mathrm{SRB}}^+.\]
This quantity does not depend on the chosen volume form $\Omega$, since given a smooth function $f>0$, $\divv_{f\Omega}(X)-\divv_{\Omega}(X)=X(\log f)$.
Note that
\[\e^{+}(-X) = \int_{\M}\divv_\Omega(X)\, \Omega_{\mathrm{SRB}}^{-} =: \e^{-}(X).\]
An important property of entropy production was proved by Ruelle \cite[Theorem 1.2]{Ruelle-96}, who shows
\[\e^{+}(X) \geq 0,\] 
with equality if and only if $X$ preserves a smooth measure. Later in the article we shall see several interesting classes of Anosov vector fields with $\e^{+}(X)>0$.

\section{Dissipative Anosov flows -- general case}\label{sec:general}

In this section we discuss resonant forms at zero for an arbitrary topologically transitive Anosov flow $X$ on an oriented $3$-manifold $\mathcal{M}$, equipped with a Riemannian metric $g$ and a smooth volume form $\Omega$ that integrates to one.

As a first step, we analyse the spaces $\res_0^k$ for $k = 0, 2$ and $\Res^3$. The classification for $k = 1$ is relegated to Theorem \ref{thm:general}.

\begin{prop}\label{prop:k=0,2,3}
	We have:
	\begin{enumerate}
		\item[1.] $\dim \Res^{0, \infty} = \dim \Res ^{3, \infty} = 1$. The space $\Res^0 = \Res^{0, \infty}$ is spanned by constant functions, and $\Res^3$ is spanned by the SRB measure $\Omega_{\mathrm{SRB}}^+$;
		\item[2.] $\dim \Res_0^{2, \infty} = 1$. The space $\Res_0^2$ is spanned by $\iota_X\Omega_{\mathrm{SRB}}^+$.
	\end{enumerate}
\end{prop}

\begin{proof}
	The proof is a straightforward consequence of well-known facts. By \cite[Theorem 3]{GuedesBonthonneau-Guillarmou-Hilgert-Weich-21} (or \cite[Theorem 1]{Butterley-Liverani-07}), the space $\Res^3$ for a transitive flow is spanned by the SRB measure $\SRB$. Moreover, semisimplicity holds: if $\Lie_X^2 \widetilde{\Omega} = 0$ for some $\widetilde{\Omega} \in \mc{D}'_{E_u^*}(\M; \Omega^3)$, then $\Lie_X \widetilde{\Omega} = c \SRB$ for some $c \in \mathbb{C}$. Stokes' theorem implies that $c = 0$, showing semisimplicity for the action of $\Lie_X$ on $\Omega^3$ and the semisimplicity for the action of $X$ on $\Omega^0$ (functions) follows by Proposition \ref{prop:non-degenerate}. By the latter fact it follows that $\Res^{0, \infty}$ is spanned by constant functions.
	
	Next, if $\Lie_X u = 0$ for some $u \in \mathcal{D}'_{E_u^*}(\M; \Omega_0^2)$, then $du \in \mc{D}'_{E_u^*}(\M; \Omega^3)$ and so $du = c\SRB$ for some $c \in \mathbb{C}$. By Stokes' theorem we get $c = 0$ and thus $du = 0$. Write $u = h \cdot \iota_X \Omega$ for some $h \in \mathcal{D}'_{E_u^*}(\mc{M})$. This implies $\Lie_X (h\Omega) = 0$ and so $h \Omega \in \Res^3$. This gives that $\Res_0^2$ is spanned by $\iota_X\Omega_{\mathrm{SRB}}^+$. Finally, to show that $\dim \Res_0^{2, \infty} = 1$, consider $u\in \mathcal{D}'_{E_u^*}(\M,\Omega_0^2)$ with $\Lie_{X}^2u=0$ and set
$v=\Lie_{X}u$. We wish to show that $v=0$. Since $v\in  \Res_0^2$, there is a constant $c$ such that $v=c\iota_X\Omega_{\mathrm{SRB}}^+$. Thus
\[c=c\int_{\M}\Omega_{\mathrm{SRB}}^+= c\int_{\M}\alpha\wedge \iota_X\Omega_{\mathrm{SRB}}^+=\int_{M}\alpha\wedge v   =\int_{\M}\alpha\wedge \iota_{X}du=\int_{\M}du=0\]
and therefore $v=0$ as desired. (Alternatively, by the version of Lemma \ref{lemma:semisimple} for $\Res^2_0 \times \Res_{0*}^0$, semisimplicity follows from the semisimplicity of $-\Lie_X$ on $\Omega^0$ established in the previous paragraph.)
	
\end{proof}

\subsection{Characterisation of $\Res_0^1$} Let us follow a similar strategy as in \cite[Section 4]{Cekic-Paternain-20}. By Proposition \ref{prop:k=0,2,3}, the two SRB measures are given by
\[\Omega_{\mathrm{SRB}}^\pm = \Pi_{3}^\pm (\Omega).\]
Recall that here $\Pi_{k}^+$ and $\Pi_{k}^-$ are the projections onto $\Res^{k, \infty}$ and $\Res_*^{k, \infty}$, respectively, introduced in \S \ref{ssec:resolvent}. Next, introduce the notation
\[\omega := \iota_X \Omega, \quad \omega^\pm := \iota_X \Omega_{\mathrm{SRB}}^\pm.\]
Define the \emph{winding cycle $W^+$ (resp. $W^-$)} by setting, for any $u \in \mathcal{D}'_{E_s^*}(\M; \Omega^1)$ (resp. $u \in \mathcal{D}'_{E_u^*}(\M; \Omega^1)$)
\[W^\pm(u) = \int_{\M} u(X)\, \Omega_{\mathrm{SRB}}^\pm.\]
When we additionally impose that $du = 0$, the cohomology class $[u]_{H^1(\mc{M})}$ of $u$ is well defined by Lemma \ref{lemma:hodge-decomposition}, $W^\pm(u)$ only depend on $[u]_{H^1(\mc{M})}$ since $\Omega_{\mathrm{SRB}}^\pm$ are flow invariant, and so in particular $W^\pm$ descend to $H^1(\mc{M})$. In what follows for simplicity we will often drop the index under the cohomology class. Note that $W^\pm \equiv 0$ (on $\ker d$) if and only if $[\omega^\pm] = 0$, by Lemma \ref{lemma:hodge-decomposition} and Poincar\'e duality.

Say that the SRB measure $\SRB$ (resp. $\SRBs$) is \emph{exact} or \emph{null-homologous} if $[\omega^+] = 0$ (resp. $[\omega^-] = 0$). Equivalently, by Lemma \ref{lemma:hodge-decomposition} this means that $\omega^+ = d\tau^+$ (resp. $\omega^- = d\tau^-$) for some $\tau^+ \in \mathcal{D}'_{E_u^*}(\mc{M}; \Omega^1)$ (resp. $\tau^- \in \mathcal{D}'_{E_s^*}(\mc{M}; \Omega^1)$). Still equivalently, applying the projector $\Pi_2^+$ (resp. $\Pi_2^-$) and using \eqref{eq:commute-projector}, we may assume that $\tau^+ \in \Res^{1, \infty}$ (resp. $\tau^- \in \Res_{*}^{1, \infty}$). Note that $\iota_X \tau^+ \in \Res^{0, \infty}$ (resp. $\iota_X \tau^- \in \Res_*^{0, \infty}$), and so by Proposition \ref{prop:k=0,2,3} we know $\iota_X \tau^+ = c^+$ (resp. $\iota_X \tau^- = c^-$) for some constants $c^\pm$, which implies $\tau^+ \in \Res^1$ (resp. $\tau^- \in \Res_*^1$). It may be assumed that $c^\pm \in \mathbb{R}$.

Let us introduce the \emph{helicity $\mathcal{H}(X)$} with respect to the SRB measures. Assume that \emph{both} $\Omega_{\mathrm{SRB}}^\pm$ are null-homologous, that is, $[\omega^+] = [\omega^-] = 0$. Introduce
\[\mathcal{H}(X) := c^+ = \int_{\mc{M}} \tau^+(X)\, \SRBs = \int_{\mc{M}} \tau^+ \wedge d\tau^- = \int_{\mc{M}} \tau^- \wedge d\tau^+ = \int_{\mc{M}} \tau^-(X)\, \SRB = c^-,\]
where we used Stokes' theorem in the fourth equality. It also follows from Stokes' theorem that $\mc{H}(X)$ is independent of the choices of both primitives $\tau^\pm$ of $\omega^\pm$, so the helicity $\mc{H}(X) \in \mathbb{R}$ is a well-defined quantity.

\begin{prop}\label{prop:existence}
	Let $f \in \mathcal{D}'_{E_u^*}(\mc{M})$ and assume $\int_{\mc{M}} f\, \SRBs = 0$. Then there exists $u \in \mathcal{D}'_{E_u^*}(\mc{M})$ such that $Xu = f$.
\end{prop}
\begin{proof}
		By Proposition \ref{prop:k=0,2,3} and \eqref{eq:laurent}, near zero we may write
		\begin{equation*}
		 	R_0^+(s) = R^{+, H}_0(s) + \frac{\Pi_{0}^+}{s}.
		\end{equation*}
		Therefore, by applying $X + s$ to this equation we obtain close to zero
		\begin{equation}\label{eq:resolvent+}
				(X + s) R^{+, H}_0(s) + \Pi_{0}^+ = \id.
		\end{equation}
		Introduce $u := R^{+, H}_0f$, which lies in $\mathcal{D}'_{E_u^*}(\mc{M}; \mathcal{E})$ by the mapping properties of $R^{+, H}_0$. Then, assuming $\Pi_{0}^+ f = 0$ we have by \eqref{eq:resolvent+} evaluated at $s = 0$:
		\begin{equation*}
				f = f - \Pi_{0}^+f = XR_0^{+, H}f = X u. 
		\end{equation*}
		Now we prove that $\Pi_{0}^+ f  = 0$. For this, use that by \eqref{eq:proj-transpose-2} we have $(\Pi_{0}^+)^T = \Pi_{3}^-$ and
		\begin{equation*}
			\Pi_0^+f = \langle{\Pi_{0}^+f, \SRBs}\rangle = \langle{f, \Pi_{3}^-(\SRBs)}\rangle = \langle{f, \SRBs}\rangle = 0,
		\end{equation*}
		where in the first equality we used that $\Pi_0^+f$ is constant and that $\SRBs$ integrates to $1$, and in the third equality that $\Pi_3^-\SRBs = \SRBs$ by definition, thus completing the proof.
\end{proof}

We proceed with an auxiliary lemma:

\begin{lemma}\label{lemma:auxiliary-tau}
	There exists $\widetilde{\tau} \in \Res^{1}$ with $\iota_X \widetilde{\tau} = 1$ and $d\widetilde{\tau} = 0$, if and only if $[\omega^-] \neq 0$.
\end{lemma}
\begin{proof}
	Assume first that such $\widetilde{\tau}$ exists. Then $W^-(\widetilde{\tau}) = 1$ and so $[\omega^-] \neq 0$ (recall that $W^- \equiv 0$ if and only if $[\omega^-] = 0$, by Lemma \ref{lemma:hodge-decomposition} and Poincar\'e duality), proving one of the implications. To see the other direction, assume $[\omega^-] \neq 0$. There is a closed, smooth $1$-form $\eta$ such that $W^-(\eta) = 1$. By Proposition \ref{prop:existence} this implies we can solve
	\[XF = 1 - \iota_X \eta, \quad F \in \mc{D}'_{E_u^*}(\mc{M}).\]
	Set $\widetilde{\tau} = dF + \eta$. Then $\widetilde{\tau}$ satisfies all of the assumptions of the lemma, completing the proof. 
\end{proof}

The next lemma determines $d(\Res_0^1)$ as a function of $[\omega^\pm]$ and $\mc{H}(X)$.
\begin{lemma}\label{lemma:T} There is a linear map $T:\Res_0^1 \to \mathbb{C}$ such that $du=T(u)\omega^+$, where $u\in \Res_0^1$. The map $T$ satisfies the following:
\begin{enumerate}
\item[1.] If $[\omega^+]\neq 0$, or $[\omega^+] = [\omega^-] = 0$ and $\mc{H}(X) \neq 0$, then $T$ is trivial;
\item[2.] If $[\omega^+] = 0$, and either $[\omega^-] = 0$ with $\mc{H}(X) = 0$, or $[\omega^-] \neq 0$, then $T$ is surjective.
\end{enumerate}
\end{lemma} 

\begin{proof} Let $u \in \res_0^1$. Then $du \in \Res_0^2$, so $du = c \omega^+$ for some $c \in \mathbb{C}$ by Proposition \ref{prop:k=0,2,3}. Setting $T(u) := c$ defines a linear map such that $du = T(u)\omega^+$. It clearly follows that $T\equiv 0$ if $[\omega^+]\neq 0$.

If $[\omega^+] = 0$, then $\omega^+ = d\tau^+$ for some $\tau^+ \in \Res^{1}$ with $\iota_X \tau^+ = c^+$, where $c^+ \in \mathbb{C}$, and we have
\[d(u-T(u)\tau^+)=0.\]
Now note that, using $\iota_X u = 0$, and the invariance of $\SRB$
\begin{equation}\label{eq:Tu}
	-T(u) c^+ = W^-(u-T(u)\tau^+).
\end{equation}
Assume first that $[\omega^-] = 0$; then also $W^- \equiv 0$. If $c^+ = \mc{H}(X) \neq 0$, \eqref{eq:Tu} implies $T \equiv 0$; if $c^+ = \mc{H}(X) = 0$, then $\tau^+ \in \Res_0^1$ and $d\tau^+ = \omega^+ \neq 0$, so $T \not \equiv 0$.

Next, assume that $[\omega^-] \neq 0$. By Lemma \ref{lemma:auxiliary-tau}, there is $\widetilde{\tau} \in \Res^1$ such that $d\widetilde{\tau} = 0$ and $\iota_X \widetilde{\tau} = 1$. Therefore $\tau^+ - c^+ \widetilde{\tau} \in \Res_0^1$ and $d(\tau^+ - c^+ \widetilde{\tau}) = \omega^+ \neq 0$, so $T \not \equiv 0$. This completes the proof.
\end{proof}

Next, we compute the dimension of the space of closed elements of $\Res_0^1$.

\begin{lemma}\label{lemma:mapS}
	There is an injection
	\begin{equation}
			S: \ker T \hookrightarrow H^1(\mc{M}).
	\end{equation}
	The injection can be described as follows: let $u \in \ker T$. Then there exists $F \in \mathcal{D}'_{E_u^*}(\mc{M})$, such that
	\begin{equation}
		u - dF \in C^\infty(\mc{M}; \Omega^1)
	\end{equation}
	and also $d(u - dF) = 0$. The injection map is given by 
	\begin{equation}
		S: \ker T \ni u \mapsto [u - d F] \in H^1(\mc{M}).
	\end{equation}
An element $[\eta] \in H^1(\mc{M})$ is in the image of $S$ if and only if
\[W^-(\eta) = \int_{\mc{M}} \eta(X)\, \SRBs = 0.\]
Finally, we have:
\begin{enumerate}
	\item[1.] $\dim S(\ker T) = b_{1}(\mc{M})$ if $[\omega^-] = 0$;
	\item[2.] $\dim S(\ker T) = b_{1}(\mc{M}) - 1$ if $[\omega^-] \neq 0$.
\end{enumerate}
\end{lemma}
\begin{proof}
  Let $u \in \ker T$, so that $d u = 0$. By Lemma \ref{lemma:hodge-decomposition} there is an $F \in \mathcal{D}'_{E_u^*}(\mc{M})$ such that $u - d F \in C^{\infty}(\mc{M}; \Omega^1)$ is closed. We claim that the class $S(u) = [u - dF] \in H^1(\mc{M})$ is independent of our choice of $F$. Let $G \in \mathcal{D}'_{E_u^*}(\mc{M})$ be arbitrary such that $u - dG$ is smooth and closed. Then $d(F-G) \in C^\infty(\mc{M}; \Omega^1)$, so by Lemma \ref{lemma:hodge-decomposition}, $F-G$ is smooth and thus $[u-dF] = [u-dG]$.
	
	For injectivity, let us assume that $u - dF$ is exact. Moreover, without loss of generality we may assume $u=dF$. Then $\iota_X u = 0$ implies $XF = 0$, so by Proposition \ref{prop:k=0,2,3} we have that $F$ is constant, so $u = 0$. 

If $[\eta]$ is in the image of $S$, then $\eta=u-dF$ for some $F \in \mathcal{D}'_{E_u^*}(\mc{M})$. Contracting with $X$, we see that $\eta(X)=-XF$, so integrating gives
\[W^-(\eta) = \int_{\mc{M}} \eta(X)\, \SRBs = 0.\]
Conversely, if the last integral is zero, Proposition \ref{prop:existence} gives an $F \in \mathcal{D}'_{E_u^*}(\mc{M})$ such that $\eta(X) = -XF$, so $u := \eta + dF \in \ker T$ and $S(u) = [\eta]$.   

Finally, observe $S(\ker T) = \ker W^-|_{H^1(\mc{M})}$. Since $W^- \equiv 0$ if and only if $[\omega^-] = 0$, the conclusions follow.
\end{proof}

Let us now put everything together and compute $\Res_0^1$ in terms of $[\omega^\pm]$ and $\mc{H}(X)$.

\begin{proof}[Proof of Theorem \ref{thm:general} for $\Res_{0}^{1}$]
	This is a direct consequence of the rank-nullity theorem and Lemmas \ref{lemma:T} and \ref{lemma:mapS}, since we may write
	\begin{align*}
		\dim \Res_0^1 = \dim S(\ker T) + \dim \ran T.
	\end{align*}
\end{proof}

\subsection{Characterisation of $\Res^{1}$} In this subsection we compute  the dimension of $\Res^1$ and we complete the proof of Theorem \ref{thm:general}.
To this end we start with a simple lemma.

\begin{lemma} An element $u\in \Res^1$ if and only if $\iota_{X}du=0$ and $\iota_{X}u$ is constant. If there exists an element $\sigma\in \Res^{1}$ with $\iota_{X}\sigma=1$, then
\[\Res^1=\Res^{1}_{0}\oplus \mathbb{R}\sigma.\]
If no such element exists, then $\Res^1=\Res^{1}_{0}$.
\label{lemma:auxres}

\end{lemma}

\begin{proof} An element $u\in \Res^1$ if and only if ${\mathcal L}_{X}u=\iota_{X}du+d\iota_{X}u=0$. From this it follows that $\iota_{X}d\iota_{X}u=0$ and thus $\iota_{X}u\in \Res^{0}$, hence $\iota_Xu$ is constant by Proposition \ref{prop:k=0,2,3}.

Suppose there exists $\sigma\in \Res^1$ such that $\iota_{X}\sigma=1$. Given any $u\in \Res^1$, we have $u-(\iota_{X}u)\sigma\in \Res_{0}^{1}$
and thus $\Res^1=\Res^{1}_{0}\oplus \mathbb{R}\sigma$. If no such $\sigma$ exists, then given any $u\in \Res^1$ we must have $\iota_{X}u=0$ and $u\in \Res_{0}^{1}$.
\end{proof}

We are now ready to prove the following proposition, which completes the proof of Theorem \ref{thm:general}:

\def\mystrut{\vrule height13pt depth7pt width0pt}

\begin{prop}
Assume $X$ generates a transitive Anosov flow on $\mc{M}$. The resonant 1-forms for $\Lie_X$ on $\Omega^1$ at zero are determined in terms of $[\omega^\pm]_{H^2(\mc{M})}$, $\mc{H}(X)$, and $\Res_{0}^{1}$ by the following table:
\begin{center}
\begin{tabular}{|l|c|c|c|c|c|}
\hline
 ~~~~~~~~~~{\rm \textbf{Cases}} & \begin{tabular}{@{}c@{}}$[\omega^+] \neq 0$ \\ $[\omega^-] \neq 0$\end{tabular} & \begin{tabular}{@{}c@{}}$[\omega^+] \neq 0$ \\ $[\omega^-] = 0$\end{tabular} & \begin{tabular}{@{}c@{}}$[\omega^+] = 0$ \\ $[\omega^-] \neq 0$\end{tabular} & \begin{tabular}{@{}c@{}}$[\omega^+] = [\omega^-] = 0$ \\ $\mc{H}(X) \neq 0$\end{tabular} & \begin{tabular}{@{}c@{}}$[\omega^+] = [\omega^-] = 0$ \\ $\mc{H}(X) = 0$\end{tabular}\\
\hline\hline
\mystrut $\Res^1$ & $\Res^{1}_{0}\oplus\mathbb{R}\sigma$ & $\Res_{0}^{1}$ & $\Res_{0}^{1}\oplus\mathbb{R}\sigma$ & $\Res_{0}^{1}\oplus\mathbb{R}\sigma$ & $\Res_{0}^{1}$\\
\hline
\mystrut $d(\Res^1)$ & $0$ & $0$ & $\mathbb{C}\,\omega^{+}$ & $\mathbb{C}\,\omega^{+}$ & $\mathbb{C}\,\omega^{+}$\\
\hline
\mystrut $\dim \Res^1$ & $b_1(\mc{M}) $ & $b_1(\mc{M})$ & $b_1(\mc{M})+1$ & $b_1(\mc{M})+1$ & $b_1(\mc{M}) + 1$\\
\hline
\end{tabular}
\end{center}
In the table, the form $\sigma\in \Res^1$ satisfies $\iota_{X}\sigma=1$ as in Lemma \ref{lemma:auxres}.
Moreover, the map 
\[\Res^1\cap \ker d\ni u\mapsto [u]\in H^{1}(\mc{M})\]
is an isomorphism.
\label{prop:res1}
\end{prop}

\medskip

\begin{proof} Consider cases as follows which combined give a proof of the various cases in the table.
\begin{enumerate}[itemsep=5pt]
\item[1.] $[\omega^{-}]\neq 0$. By Lemma \ref{lemma:auxres} it suffices to show that there is a form $\sigma$ with $d\sigma=0$ and $\iota_{X}\sigma=1$. This form is given by Lemma \ref{lemma:auxiliary-tau}.
\item[2.] $[\omega^{\pm}]=0$ and ${\mathcal H}(X)\neq 0$. There is a primitive $\tau^{+} \in \Res^1$ with $d\tau^{+}=\omega^{+}$ and $\iota_X \tau^+ = \mc{H}(X)$. Then $\sigma := \frac{\tau^+}{\mc{H}(X)}$ satisfies the conditions of Lemma \ref{lemma:auxres}, completing the proof.
\item[3.] $[\omega^{+}]\neq 0$ and $[\omega^{-}]=0$. Suppose there is $\sigma \in \Res^1$ with $\iota_{X}d\sigma=0$ and $\iota_{X}\sigma=1$. Since $d\sigma\in \Res_{0}^{2}$ and $[\omega^{+}]\neq 0$ we must have $d\sigma=0$ by Proposition \ref{prop:k=0,2,3}. But by Lemma \ref{lemma:auxiliary-tau} this implies $[\omega^{-}]\neq 0$ and hence such $\sigma$ cannot exist. By Lemma \ref{lemma:auxres} it follows that $\Res^1=\Res_{0}^{1}$ as claimed.
\item[4.] $[\omega^{\pm}]=0$ and $\mathcal H(X)=0$. Once again we need to show that there is no $\sigma \in \Res^1$ with $\iota_{X}d\sigma=0$ and $\iota_{X}\sigma=1$. Since $[\omega^{-}]=0$ we must have $d\sigma=c\omega^{+}$, where $c\neq 0$ by Lemma \ref{lemma:auxiliary-tau}. Since $c^{-1}\sigma(X)\neq 0$ this would give non-zero helicity and thus a contradiction.
\end{enumerate}

Finally, the map $\Res^1\cap \ker d\ni u\mapsto [u]\in H^{1}(\mc{M})$ is clearly injective: if $u=df$ for $f\in {\mathcal D}_{E_{u}^{*}}'(\mc{M})$, then $Xf = c$. Integration against $\SRBs$ gives $c=0$ and thus $f$ must be constant by Proposition \ref{prop:k=0,2,3}. A glance at the table in the proposition shows that the map must be an isomorphism.
\end{proof}

\begin{Remark}{\rm An interesting asymmetry arises in the second and third columns of the table in the previous proposition. Consider
the case when we have a flow with $[\omega^{+}]=0$ and $[\omega^{-}]\neq 0$. Then $\dim \Res^1=b_{1}(\M)+1$ but $\dim \Res^1_{*}=b_{1}(\M)$ (by applying the proposition to $-X$). Note that by virtue of Theorem \ref{thm:general}, we always have $\dim \Res_{0}^{1}=\dim\Res_{0*}^{1}$.
Also note that thanks to the paring between $\Res_{0}^1$ and $\Res_{0*}^{1}$ and Lemma \ref{lemma:semisimple}, semisimplicity for $\mathcal L_{X}$ and $\mathcal L_{-X}$ acting on $\Omega_{0}^{1}$ are equivalent. This is no longer the case for the actions of $\mathcal L_{\pm X}$ on $\Omega^{1}$ as we shall see below.

}
\end{Remark}

\begin{Remark}{\rm As we mentioned before, $\dim \Res^1$ is invariant under time changes. In the cases given by the first and third columns of the table there is $\sigma\in \Res^1$ with $d\sigma=0$ and $\iota_{X}\sigma=1$. After a time change this particular form may no longer be in $\Res^1$, however, there is always a form of that type. A similar remark applies to the fourth column, where there is always a form $\sigma\in \Res^1$ such that $d\sigma=c\omega^{+}$ and $\iota_{X}\sigma=1$ with $c\neq 0$. For example, if we begin with a Reeb Anosov vector field, after performing a time change most likely it will not be any longer a Reeb vector field (as the bundle $E_{u}\oplus E_{s}$ stops being smooth), nevertheless it remains ``weakly Reeb" in the sense that there still exists a form in $\Res^1$ playing the same role as the contact form at least from the point of view of resonant states.

}\label{remark:tc}
\end{Remark}

\subsubsection{Semisimplicity on $\Omega^1$}

Finally, we discuss semisimplicity for the action of the Lie derivative on $\Omega^1$. 

\begin{prop}\label{prop:semisimplicity-fails}
	Assume there is $\tau^+ \in \Res_0^1$ such that $d\tau^+ = \omega^+$ (this corresponds to Item 2 in Lemma \ref{lemma:T}). Then the semisimplicity for the action of $\Lie_{-X}$ on $\Omega^1$ fails.
\end{prop}
\begin{proof}
	Let $\alpha^- \in \Res_*^{1, \infty}$ such that $\iota_X \alpha^- = 1$. Such $\alpha^-$ exists since we may take any $\alpha \in C^\infty(\mc{M}; \Omega^1)$ with $\iota_X \alpha = 1$ and set $\alpha^- := \Pi^-_1 \alpha$ (here we use \eqref{eq:commute-projector}). Then:
\begin{align*}
	1 = \int_{\mc{M}} \alpha^- \wedge d\tau^+ = \int_{\mc{M}} d\alpha^- \wedge \tau^+ =
\int_{\mc{M}} \alpha \wedge \iota_X d\alpha^- \wedge \tau^+ = \llangle{\iota_X d\alpha^-,
\tau^+}\rrangle,	
\end{align*}
where in the first equality we used $\iota_X \alpha^- = 1$ (so that $\alpha^- \wedge d\tau^+ = \SRB$), Stokes' theorem in the second one, while in the third equality we used $\iota_X \tau^+ = 0$ and $\iota_X \alpha = 1$. Therefore, $\Lie_X \alpha^- = \iota_X d\alpha^- \neq 0$ and so $\alpha^- \not \in \Res^{1}_*$, and so by definition semisimplicity for the action of $\Lie_{-X}$ on $\Omega^1$ fails.
\end{proof}

\begin{Remark}\label{rem:bla}\rm
	Observe also that the described failure of semisimplicity for $\Lie_{-X}$ acting on $\Omega^1$ persists under time changes since the property that $d(\Res_0^1) \neq 0$ is invariant under time changes, and that the same property is open in the set $\mc{W}^+ \setminus \mc{W}^-$ (defined in Section \ref{sec:perturbation}). Moreover, semisimplicity for the action of $\Lie_{-X}$ on $\Omega_0^1$ and $\Omega^1$ can hold and fail, respectively, see Proposition \ref{prop:newexamples}, Item 4.
\end{Remark}

\medskip

\subsection{Characterisation of $\Res^2$} Here we complete the study of resonant forms at zero and compute the dimension of $\Res^2$.

\begin{lemma} \label{lemma:res2-basic}
	A form $u\in \Res^2$ if and only if $du=0$ and $d\iota_{X}u=0$.
\end{lemma}

\begin{proof} By definition being in $\Res^2$ means that $\Lie_{X}u=0$ and since $\Lie_{X}$ commutes with $d$ and $\iota_{X}$ we must have $du\in \Res^3$ and $\iota_{X}u\in \Res_{0}^{1}$. By Proposition \ref{prop:k=0,2,3} there is a constant $c$ such that $du=c\Omega_{\text{SRB}}^{+}$ and integrating we deduce that $c=0$ and the lemma follows.
\end{proof}

\begin{lemma} The map $\iota_{X}:\Res^2\to \Res_{0}^{1}\cap\,\text{\rm ker}\,d$ is surjective with right inverse $v\mapsto \Pi^{+}_{2}(\alpha\wedge v)$ where $\alpha$ is any smooth 1-form such that $\iota_{X}\alpha=1$. Moreover, $\ker \iota_{X}=\mathbb{C}\,\omega^{+}$.

\label{lemma:res2}

\end{lemma}

\begin{proof} The map is well-defined by Lemma \ref{lemma:res2-basic}. Since $\iota_{X}$ commutes with $\Pi_1^{+}$ and $\Pi_2^{+}$ (see \eqref{eq:commute-projector}) we have for $v \in \Res_0^1$:
\[\iota_{X}\Pi_{2}^{+}(\alpha\wedge v)=\Pi_{1}^{+}(\iota_{X}(\alpha\wedge v))=\Pi_{1}^{+}(v)=v.\]
If in addition $dv=0$, then $\Pi^{+}_{2}(\alpha\wedge v)$ is closed since $d\Pi^{+}_{2}(\alpha\wedge v)=\Pi_{3}^{+}(d\alpha\wedge v)=c\,\Omega_{\text{SRB}}^{+}$ for some constant $c \in \mathbb{C}$ by Proposition \ref{prop:k=0,2,3} and $c=0$ by integration by parts, so $\Pi^{+}_{2}(\alpha\wedge v) \in \Res^2$.

The claim about the kernel of $\iota_{X}$ follows from the second item in Proposition \ref{prop:k=0,2,3}.
\end{proof}

Lemma \ref{lemma:res2} shows that there exists a surjective linear map $C:\Res^{2}\to\mathbb{C}$ such that any $u\in \Res^2$ may be written
as
\begin{equation}
u=\Pi_{2}^{+}(\alpha\wedge v)+C(u)\omega^{+}, \qquad v=\iota_{X}u.
\label{eq:u2}
\end{equation}

From Lemma \ref{lemma:res2} and Theorem \ref{thm:general} applied to $\pm X$, we derive right away the following corollary.

\begin{corollary} Assume $X$ generates a transitive Anosov flow on $\mc{M}$. The resonant 2-forms for the action of $\Lie_X$ on $\Omega^1$ at zero satisfy
\[\dim \Res^2=\dim \Res_{*}^{1}.\]
\end{corollary}

Using \eqref{eq:u2} we may relate the pairings as follows. Given $u\in \Res^2$ and $u_{*}\in\Res^{1}_{*}$ we have
\begin{align}\label{eq:pairings-relate}
\begin{split}
\langle u, u_{*}\rangle&=\langle \Pi_{2}^{+}(\alpha\wedge v)+C(u)\omega^{+},u_{*}\rangle = \llangle \iota_{X}u,u_{*}\rrangle+C(u)\,W^{+}(u_{*}),
\end{split}
\end{align}
where we used \eqref{eq:proj-transpose-2} and $\Pi_1^- u_{*} = u_{*}$ in the second equality. In the case $[\omega^{+}]=0$ and $[\omega^{-}]\neq 0$, by Proposition \ref{prop:res1} (applied to $-X$) we know that  $\Res_{*}^{1}=\Res_{0*}^{1}$  and $d(\Res_{0*}^1) = \{0\}$, hence
$W^{+}(u_{*})=0$ for all $u_{*}\in \Res^{1}_{*}$. Thus $u=\omega^{+}$ satisfies $\langle u, u_{*}\rangle=0$ for all $u_{*} \in \Res^{1}_{*}$
showing that semisimplicity fails for the action of $\Lie_X$ on $\Omega^2$, by Proposition \ref{prop:non-degenerate}. But we expect semisimplicity to hold for the pairing $\llangle \bullet, \bullet\rrangle$ at least under small perturbations as in Section \ref{section:onesmooth}.

In the case $[\omega^{\pm}]=0$ and $\mathcal H(X)=0$, we see also that  $\Res_{*}^{1}=\Res_{0*}^{1}$. There is however a non-closed element of $\Res_{0*}^1$, but we still get $W^{+}(u_{*})=0$ for all $u_{*}\in \Res_{*}^{1}$. Therefore $u=\omega^{+}$ satisfies $\langle u, u_{*}\rangle=0$ for all $u_{*} \in \Res^{1}_{*}$ and again, semisimplicity fails for the action of $\Lie_X$ on $\Omega^2$ by Proposition \ref{prop:non-degenerate}.

\begin{Remark}\rm Observe that the pairings $\llangle{\bullet, \bullet}\rrangle$ and $
\langle{\bullet, \bullet}\rangle$ are between $\Res_{0}^{1}$ and $\Res_{0*}^{1}$, and $\Res^{2}$ and $\Res_{*}^{1}$, respectively, and in both cases dimensions of the corresponding resonant spaces agree (this would be obvious under semisimplicity). We are not aware at the moment of an alternative proof of these facts that does not go through Theorem \ref{thm:general}. 
	
	We note that the two conclusions following \eqref{eq:pairings-relate} about non-semisimplicity follow directly from Propositions \ref{prop:non-degenerate} and \ref{prop:semisimplicity-fails}, however our aim above was to relate the two pairings explicitly by \eqref{eq:pairings-relate}.
\end{Remark}

\section{Horocyclic invariance}\label{sec:horocyclic}

According to Faure-Guillarmou \cite{Faure-Guillarmou-18}, the Ruelle resonant states for contact Anosov flows (in dimension $3$) close to the imaginary axis exhibit horocyclic invariance. In a suitable sense, here we extend this idea to resonant $1$-forms at zero of \emph{transitive Anosov flows}.

Let $X$ generate an arbitrary transitive Anosov flow on a closed orientable $3$-manifold $\mc{M}$. By \cite[Corollary 1.8]{Hasselblatt-94}, it is known that the weak unstable/stable bundles $\mathbb{R}X \oplus E_{u/s}$ are $C^{1 + \alpha}$ H\"older regular for some $\alpha > 0$ (depending on the flow). In particular, if $E_{u/s}$ are orientable, there are vector fields $Y_{u/s} \in C^{1 + \alpha}(\mc{M};T\mc{M})$ such that $\mathbb{R}X \oplus E_{u/s}$ are spanned by $\{X, Y_{u/s}\}$ at every point. Moreover, by \cite{Hasselblatt-94}, there are $C^\beta$ regular vector fields $U^{u/s}$ for some $\beta > 0$, differentiable along the flow, spanning $E_{u/s}$ at every point, as well as $r^{u/s} \in C^\beta(\mc{M})$ such that:
\begin{equation}\label{eq:stable-unstable-flow}
	[X, U^{u/s}] = -r^{u/s} U^{u/s}.
\end{equation} 
Let $\omega, \omega_{u/s} \in C^{1 + \alpha}(\mc{M}; \Omega^1)$ be defined as 
\begin{equation}\label{eq:omega-u-s}
	\omega_{u/s}(Y^{u/s}) = \omega_{u/s}(X) = 0, \quad \omega_{u/s}(Y^{s/u}) = 1, \quad \omega(Y^u) = \omega(Y^s) = 0, \quad \omega(X) = 1.
\end{equation}
Let $H$ and $V$ be smooth vector fields approximating $U^u$ and $U^s$, respectively, in the sense that
\[\|H - U^u\|_{C^0} + \|V - U^s\|_{C^0} < \varepsilon, \quad \|[X, H] - [X, U^u]\|_{C^0} + \|[X, V] - [X, U^s]\|_{C^0} < \varepsilon,\]
for a certain sufficiently small $\varepsilon > 0$, to be specified later. Indeed, such an approximation exists thanks to standard arguments, see e.g. \cite[Lemma E.45]{Dyatlov-Zworski-19} in the related setting of the scale of Sobolev spaces. After possibly time-changing $U^{u/s}$ by some nowhere zero $b_{u/s} \in C^\beta(\mc{M})$, such that $\|b_{u/s} - 1\|_{C^0(\mc{M})} + \|Xb_{u/s}\|_{C^0(\mc{M})} < C\varepsilon$ for some uniform $C > 0$, we may assume that:
\begin{equation}\label{eq:U^u/s-general}
	U^u = H + r_V V - a_u X, \quad U^s = r_H H + V - a_s X,
\end{equation}
for some $r_{H/V}, a_{u/s} \in C^\beta(\mc{M})$. Moreover, since $Y^{u/s} = x_{u/s}X + u_{u/s}U^{u/s}$ for some $x_{u/s}, u_{u/s} \in C^\beta(\mc{M})$ (where $u_{u/s} \neq 0$ pointwise), using \eqref{eq:U^u/s-general} we immediately get that 
\[x_{u/s} - u_{u/s}a_{u/s},\quad u_{u/s},\quad u_u r_V,\quad u_s r_H\quad \in C^{1 + \alpha}(\mc{M}).\] 
In particular $u_{u/s} \in C^{1 + \alpha}(\mc{M})$ implies $r_{H/V} = \frac{u_{s/u} r_{H/V}}{u_{s/u}} \in C^{1 + \alpha}(\mc{M})$. After subtracting the $X$ part in \eqref{eq:U^u/s-general}, we may re-define $Y^{u/s}$ to
\begin{equation}\label{eq:Y^u/s-r_H/V}
	Y^u := H + r_V V, \quad Y^s := r_H H + V, 
\end{equation}
where $r_{H/V} \in C^{1 + \alpha}(\mc{M})$, and $\{X, Y^{u/s}\}$ still spans the weak bundle $\mathbb{R}X \oplus E_{u/s}$ pointwise. In fact, if $\varepsilon > 0$ is small enough, by \cite[Theorem 2]{GuedesBonthonneau-Guillarmou-DePoyferre-21}, $r_{H/V}$ additionally satisfy a wavefront set condition:
\begin{equation}\label{eq:r_H/V-regularity}
	r_V \in \mc{D}'_{E_u^*}(\M) \cap C^{1 + \alpha}(\M),\qquad r_H \in \mc{D}'_{E_s^*}(\M) \cap C^{1 + \alpha}(\M).
\end{equation}
It follows from \eqref{eq:Y^u/s-r_H/V}, \eqref{eq:U^u/s-general}, and \eqref{eq:stable-unstable-flow} that:
\begin{equation}\label{eq:commutator-X-Y}
	Y^{u/s} = a_{u/s}X + U^{u/s}, \qquad [X, Y^{u/s}] = (X + r^{u/s})a_{u/s} X -r^{u/s}Y^{u/s}. 
\end{equation}
Integrating \eqref{eq:stable-unstable-flow}, the function $r^u$ satisfies, for all $t \in \mathbb{R}$, $z \in \mc{M}$:
\begin{equation}\label{eq:U^u}
	d\varphi_{-t}(z) U^u(z) = e^{-\int_0^t r^u \circ \varphi_{-s}(z)\, ds} U^u(\varphi_{-t}z),\end{equation}
which implies in particular that (by definition of Anosov flows \eqref{eq:anosov-def})
\[0 < \nu_{\min} := \lim_{T\to \infty} \inf_{z \in \mc{M}} \frac{1}{T}\int_0^T r^u \circ \varphi_{-s}(z)\, ds\]
is the minimal expansion rate. 

In fact, more is true about the regularity of $r^{u/s}$:

\begin{lemma}\label{lemma:regularity-r^s/u}
	We have $Xr^{u/s}, r^{u/s}, Xr_{H/V}$, and $(X + r^{u/s})a_{u/s}$ all belong to $C^{1 + \alpha}(\mc{M}) \cap \mc{D}'_{E_{u/s}^*}(\M)$. Moreover, $r_{H/V}$ satisfy a \emph{Riccati type} equation.
\end{lemma}
\begin{proof}
In the following, for a vector field $Z$, denote by $Z_{X/H/V}$ the component of $Z$ in the $X/H/V$ direction, respectively. To simplify the notation we consider the unstable quantities only, the case of stable ones follows similarly. Write the commutator formula \eqref{eq:commutator-X-Y} in two ways, using \eqref{eq:Y^u/s-r_H/V}:
\begin{align*}
	[X, Y^u] &= ([X, H]_X + r_V [X, V]_X)X + ([X, H]_H + r_V [X, V]_H)H + \big([X, H]_V + (X + [X, V]_V)r_V\big)V\\
	&= (X + r^u)a_u X - r^u H - r^u r_V V.
\end{align*}
By equating the coefficients next to the vector fields $X, H, V$, we get using \eqref{eq:r_H/V-regularity}:
\begin{align*}
	(X + r^u)a_u &= [X, H]_X + r_V [X, V]_X \in C^{1 + \alpha}(\mc{M}) \cap \mc{D}'_{E_u^*}(\M),\\
	-r^u &= [X, H]_H + r_V [X, V]_H \in C^{1 + \alpha}(\mc{M}) \cap \mc{D}'_{E_u^*}(\M),\\
	-r^ur_V &= [X, H]_V + (X + [X, V]_V)r_V.
\end{align*}
It follows that the left hand side of the last equation is also in $C^{1 + \alpha}(\mc{M})$ and so $X r_V \in C^{1 + \alpha}(\mc{M})$, and thus also by the second equation $Xr^u \in C^{1 + \alpha}(\mc{M})$. This proves the first part of the claim, while the second part follows by substituting the second equation into the third one.
\end{proof}

\subsection{Horocyclic invariance of resonant $1$-forms} Next we show that (closed) elements $u \in \Res_0^1$ vanish on the weak unstable bundle, generalising the analogous claims for geodesic flows in constant negative curvature \cite{Cekic-Dyatlov-Delarue-Paternain-22, Kuster-Weich-20}. In particular, we will show that $u \in C_*^{-1-}(\M; \Omega^1)$ and since $\mathbb{R}X \oplus E_u$ is $C^{1 + \alpha}_*$ regular, the contraction is well-defined. We remark that it suffices to check that $\iota_{Y^u} u = 0$ ($\iota_X u = 0$ by definition), and then $\iota_Y u = 0$ for any $C^{1 + \alpha}_*$ section $Y$ of $\mathbb{R}X \oplus E_u$.

\begin{lemma}[Horocyclic invariance I]\label{lemma:horo-I}
	Let $X$ generate a transitive Anosov flow on a closed $3$-manifold $\mc{M}$, such that $E_u$ is orientable. Then:
	\begin{enumerate}
		\item[1.] If $u \in \Res_0^1 \cap \ker d$, then $u \in C_*^{-1-}(\mc{M};\Omega^1)$ and $\iota_{Y^u} u = 0$;
	
		\item[2.] Assume $X$ preserves a smooth probability volume form $\Omega$. Then \emph{any} $u \in \Res_0^1$ satisfies $u \in C_*^{-1-}(\mc{M};\Omega^1)$ and $\iota_{Y^u} u = 0$.
	\end{enumerate}
\end{lemma}
\begin{proof}
	Observe firstly that $U^u = Y^u - a_uX$ (see \eqref{eq:commutator-X-Y}) extends continuously as a differential operator to some $C_*^{-\varepsilon}(\mc{M}) \cap \mc{D}(X)$, where $\mc{D}(X)$ denotes the domain of $X$, i.e. all $u \in C_*^{-\varepsilon}(\mc{M})$ such that $Xu \in C_*^{-\varepsilon}(\mc{M})$, for any $\varepsilon$ with $\min(\alpha, \beta) > \varepsilon > 0$.
	
	We deal first with the case when $du = 0$. By Lemma \ref{lemma:hodge-decomposition} and Proposition \ref{prop:existence}, we may write $u = \theta + d\varphi$, where $\theta \in C^\infty(\mc{M}; \Omega^1)$, $d\theta = 0$, $X\varphi = -\iota_X \theta$, and $\varphi = -R^{+, H}_0 (\iota_X \theta)$ is obtained by applying (holomorphic part of) the resolvent at zero. By Proposition \ref{prop:mapping} this implies $\varphi \in  C_*^{0-}(\mc{M}) \cap \mathcal{D}'_{E_u^*}(\mc{M})$ and so $u \in C_*^{-1-}(\M; \Omega^1)$. Since $\Pi^+_{0} R^{+, H}_0 = 0$ (see Proposition \ref{prop:existence}), we have $\Pi^+_{0} \varphi = 0$, or equivalently $\int_{\mc{M}} \varphi\, \SRBs = 0$. Clearly $\varphi \in \mc{D}(X)$ and we claim that 
	\[U^u\varphi \in C^0(\mc{M}),\]
	and in fact $U^u \varphi$ is H\"older regular for some positive exponent. To see this, it suffices to show
	\begin{equation}\label{eq:bla}
		U^u\varphi = (X + r^u)^{-1} U^u X\varphi,
	\end{equation}
	since $X\varphi = -\iota_X \theta \in C^\infty(\mc{M})$, $U^uX\varphi \in C^\beta(\mc{M})$, and $(X + r^u)^{-1}$ is a continuous map $C^0(\mc{M}) \to C^0(\mc{M})$ by Proposition \ref{prop:mapping}.
	
	Observe that for $\re(s) > 0$, $f, \phi \in C^\infty(\mc{M})$, we have in the sense of distributions (using the formula \eqref{eq:resolvent-formula} from Proposition \ref{prop:mapping}):
	\begin{align}\label{eq:relation'}
	\begin{split}
		\int_{\M} (X + r^u + s)^{-1}U^u f(x) \phi(x)\, d\vol(x) &= \int_{\M} \phi(x) d\vol(x) \int_0^\infty e^{-\int_0^t (r^u \circ \varphi_{-p} + s)\, dp} U^uf(\varphi_{-t}x)\, dt\\
		&= \int_{\M} \phi(x) d\vol(x) \int_0^\infty e^{-st} df(\varphi_{-t}x) \big(d\varphi_{-t}(x) U^u(x)\big)\, dt\\
		&= \int_{\M} \phi(x) d\vol(x) \int_0^\infty e^{-st} U^u (\varphi_{-t}^*f) (x)\, dt\\
		&= \lim_{N \to \infty} \int_{\M} U^u (X + s)^{-1} (\id - e^{-sN} \varphi_{-N}^*)f(x) \phi(x)\, d\vol(x)\\
		&= \int_{\M} U^u(X + s)^{-1} f(x) \phi(x)\, d\vol(x),
	\end{split}
	\end{align}
	where $d\vol$ is an arbitrary smooth volume form on $\M$. In the second line, we used \eqref{eq:U^u}, in the fourth we re-wrote the integral as the resolvent and used that $\int_0^Ne^{-st}U^u(\varphi_{-t}^*f)\, dt$ converges in $C^0(\M)$, and in the last line we used that $U^u(X + s)^{-1}: C^0(\mc{M}) \to \mc{D}'(\mc{M})$ is continuous to take the limit as $N \to \infty$.
	
	To show \eqref{eq:bla}, observe that for $\re(s) > 0$:
	\begin{equation}\label{eq:relation''}
		(X + r^u + s)^{-1} U^u X\varphi = U^u(X + s)^{-1} X\varphi = U^u \varphi - s U^u(X + s)^{-1} \varphi,
	\end{equation}
	where the first equality holds by \eqref{eq:relation'}. Now for $\re(s) > 0$ and $s$ close to zero, $(X + s)^{-1} \varphi = R^{+, H}_0(s) \varphi$ since $\Pi^+_{0} \varphi = 0$, and recalling that $U^u$ extends continuously as a differential operator to $C^{-\varepsilon}_*(\M) \cap \mc{D}(X)$ (since $U^u = Y^u - a_uX$), and using $R^{+, H}_0(s)\varphi \in C_*^{-\varepsilon}(\M)$ is uniformly bounded (by Proposition \ref{prop:mapping}), taking $s \to 0$ proves the claim.
	
	Now going back to $X\varphi = -\iota_X \theta$, we may apply $U^u$ to get, using $d\theta = 0$ and \eqref{eq:stable-unstable-flow}:
	\[U^uX\varphi = -U^u \iota_X \theta = d\theta(X, U^u) - X \iota_{U^u} \theta + \iota_{[X, U^u]} \theta = -(X + r^u) \iota_{U^u}\theta.\]
	Since $\iota_{U^u} \theta$ and $U^u X\varphi$ belong to $C^0(\mc{M})$, applying $(X + r^u)^{-1}$ we conclude that
	\[\iota_{U^u} \theta = -(X + r^u)^{-1} U^u X\varphi = -U^u\varphi,\]
	where in the last equality we used \eqref{eq:bla}. Using $\iota_{a_uX} \theta = -a_u X \varphi$, we conclude $0 = \iota_{Y^u} \theta + Y^u\varphi = \iota_{Y^u} u$, completing the proof in the case $du = 0$.
		
			By Proposition \ref{prop:k=0,2,3}, since $du \in \Res_0^2$ it remains to deal with the case $du = \iota_X \SRB$. By assumption, $\SRB = \SRBs = \Omega$ since $\Omega$ is a smooth and invariant, so in particular $du \in C^\infty(\mc{M}; \Omega_0^2)$. By Lemma \ref{lemma:hodge-decomposition}, there exist $\varphi \in \mc{D}'_{E_u^*}(\mc{M})$, a closed $\psi \in \mc{D}'_{E_u^*}(\mc{M}; \Omega^2)$, and a smooth $\theta \in C^\infty(\mc{M}; \Omega^1)$ such that
	\begin{equation}\label{eq:hodge}
		u = d\varphi + d^* \psi + \theta.
	\end{equation}
	Since $d \psi = 0$, it follows that $d d^* \psi = \Delta_2 \psi$ (recall $\Delta_2$ is the Hodge Laplacian on $\Omega^2$ of some Riemannian metric). Observe
	\begin{equation}\label{eq:du}
		C^\infty(\mc{M}; \Omega^2) \ni du = \Delta_2 \psi + d\theta,
	\end{equation}
	and by elliptic regularity (equivalently, by applying the operator $Q_2$ from Lemma \ref{lemma:hodge-decomposition}) we get $\psi \in C^{\infty}(\mc{M}; \Omega^2)$. Set $\eta := d^* \psi + \theta \in C^\infty(\mc{M}; \Omega^1)$.
	
	Using $\iota_X u = 0$, we get $X\varphi = -\iota_X \eta$. In particular, we may apply $R_0^{+, H}$ and assume that
	\[\varphi = -R_0^{+, H} \iota_X \eta \in C_*^{0-}(\mc{M}) \cap \mc{D}'_{E_u^*}(\mc{M}),\]
	by Proposition \ref{prop:mapping}. Observe that \eqref{eq:relation'} still holds, and arguing as before we conclude that we may take the limit $s \to 0$ in \eqref{eq:relation''}, so \eqref{eq:bla} holds in this setting as well. Then we compute:
	\begin{align*}
		U^uX\varphi = -U^u \iota_X \eta = d\eta (X, U^u) - X \iota_{U^u} \eta + \iota_{[X, U^u]} \eta = -(X + r^u) \iota_{U^u}\eta,
	\end{align*}
	where in the first equality we used \eqref{eq:hodge} and in the third equality we used that $\iota_X d\eta = \iota_X du = 0$ and \eqref{eq:stable-unstable-flow}. Since $\iota_{U^u}\eta$ and $U^u X\varphi$ are H\"older regular, applying $(X + r^u)^{-1}$ to the previous equation, using $X\varphi = -\iota_X \eta$ and \eqref{eq:bla}, we get:
	\[U^u \varphi + \iota_{U^u}\eta = \iota_{Y^u} u = 0,\]
	which completes the proof.
\end{proof}
\begin{Remark}\rm \label{rem:non-smooth-measure-horo-II}
	The only unclear case in Lemma \ref{lemma:horo-I} is when $du = \iota_X \SRB$ for some $u \in \Res_0^1$ and $X$ does not preserve a smooth measure. Let us follow the notation of the second part of Lemma \ref{lemma:horo-I}. Then $du \in \Big(C^0(\mc{M})\Big)'$ only and by Sobolev embeddings and elliptic regularity we eventually get that $\psi \in W^{2-\varepsilon, \frac{3}{3 - \varepsilon}-}(\mc{M}) \cap \mc{D}'_{E_u^*}(\mc{M})$, for all $0 < \varepsilon < 1$. Then the regularity of $\varphi = -R_0^{+, H} \iota_X \eta$ is not entirely clear: given these estimates, it belongs only to $C_*^{-2-}(\mc{M})$, which is not enough to define $Y^u\varphi$.
\end{Remark}

Next we give a lemma that relates $\Res_0^1$ to a resonant state of $X + Q$ with an additional horocyclic invariance, where $Q$ is some singular potential, generalising the case of geodesic flows in constant negative curvature \cite{Dyatlov-Faure-Guillarmou-15, Guillarmou-Hilgert-Weich-18, Cekic-Dyatlov-Delarue-Paternain-22}.

\begin{lemma}[Horocyclic invariance II]\label{lemma:horocyclic-invariance-new}
	Let $u \in \Res_0^1 \cap \ker d$. Then $c := \iota_V u \in C_*^{-1-}(\M) \cap \mc{D}'_{E_u^*}(\M)$ satisfies:
	\begin{align*}
		(X + [X, V]_H \cdot r_V - [X, V]_V)c &= 0,\\
		(H + [H, V]_H \cdot r_V - [H, V]_V)c + V(r_V c) &= 0.
	\end{align*}
	Similarly, any $u_* \in \Res_{0*}^1 \cap \ker d$ satisfies, if $c_* := \iota_Hu_*$:
	\begin{align*}
		(X + [X, H]_V \cdot r_H - [X, H]_H)c_* &= 0,\\
		H(r_H c_*) + (V - [H, V]_V \cdot r_H + [H, V]_H)c_* &= 0.
	\end{align*}
\end{lemma}
\begin{proof}
Firstly, since $\{X, H, V\}$ form a smooth global frame, there is a smooth global dual co-frame $\{\alpha, \beta, \psi\}$. By Lemma \ref{lemma:horo-I}, we know $\iota_{Y^u} u = 0$. If $u = a\alpha + b\beta + c \psi$ for some $a, b, c \in \mc{D}_{E_u^*}(\mc{M}) \cap C_*^{-1-}(\mc{M})$, it follows that:
\[\iota_X u = 0 \implies a = 0, \quad \iota_{Y^u} u = 0 \implies b + c r_V = 0.\]
Similar argument applies for any $u_* \in \Res_{0*}^1 \cap \ker d$, so there is $\iota_H u_* =: c_{*} \in \mc{D}_{E_s^*}(\mc{M}) \cap C_*^{-1-}(\mc{M})$ such that 
\begin{equation}\label{eq:u-frame}
	u = c(\psi - r_V \beta), \quad u_* = c_{*}(\beta - r_H \psi).
\end{equation}
Now $du = 0$ implies:
\begin{equation}\label{eq:horo-II-1}
	0 = du(X, V) = X \iota_V u - \iota_{[X, V]} u = Xc - c([X, V]_V - r_V[X, V]_H),
\end{equation}
where in the second equality we used $\iota_Xu = 0$, and in the third one we applied \eqref{eq:u-frame}. This gives the first equation for $u$, and the second one follows similarly:
\begin{equation}\label{eq:horo-II-2}
	0 = du(H, V) = H \iota_V u - V \iota_H u - \iota_{[H, V]} u = Hc + V(r_V c) + c([H, V]_H \cdot r_V - [H, V]_V)
\end{equation}
where we used \eqref{eq:u-frame} in the third equality. The case of $u_* \in \Res_{0*}^1$ follows analogously.
\end{proof}

We emphasise that the preceding result was not known even in the case of geodesic flows in variable negative curvature. Let us record this fact separately:
\begin{prop}\label{prop:geodesic-flow-horo-inv}
	Let $(M, g)$ be an Anosov surface, $\M = SM$ the unit sphere bundle, and $X$ the geodesic vector field. Let $\{X, H , V\}$ be the orthonormal frame constructed in \S \ref{ssec:geometry-surfaces}. Then every $u \in \Res_0^1$ satisfies, where $c := \iota_V u$:
	\begin{align*}
		(X - r^u)c &= 0,\\
		Hc + V(r^uc) &= 0,
	\end{align*}	
	where $r^u$ is the unique function such that $H+r^uV\in E_{u}$.
\end{prop}
\begin{proof}
	Note firstly that $d(\Res_0^1) = 0$ in this case as follows from Theorem \ref{thm:general}. Next, we observe that in this frame $r_V = r^u$. Indeed, this is well-known and e.g. follows from Lemma \ref{lemma:commutator-stable/unstable} below for $\lambda = 0$ (i.e. in this case we may take $U^u = Y^u = H + r^uV$). Then the claim follows directly from Lemma \ref{lemma:horocyclic-invariance-new} and \eqref{eq:surface-geometry}.
\end{proof}

\begin{Remark}\rm
	In Lemma \ref{lemma:horocyclic-invariance-new}, the subtlety lies in the divergence type expression $V(r_V c)$ which cannot be expanded since the products $V(r_V) \cdot c$ and $r_V \cdot Vc$ typically do not make sense. However, as soon as the regularity of $r_V$ is $C^{2+}_*(\mc{M})$ (i.e. when the weak unstable bundle is $C^{2+}_*(\mc{M})$-regular), these products make sense and we have a classical interpretation of horocyclic invariance.
\end{Remark}

\begin{Remark}\rm \label{rem:horocyclic-invariance-general}
	Using similar techniques one can easily prove horocyclic invariance for resonant states with resonances with real part sufficiently close to zero for arbitrary Anosov flows in dimension three. We believe this may be useful to study the exponential speed of mixing of $3$-dimensional Anosov flows as in \cite{Cekic-Guillarmou-21}. We leave this for discussion elsewhere.
\end{Remark}

\begin{Remark}{\rm Proposition \ref{prop:geodesic-flow-horo-inv} extends easily to arbitrary Anosov thermostats $F=X+\lambda V$ as introduced in Section \ref{sec:thermo}: if $u\in \text{Res}^{1}_{0}\cap \ker d$ and $c=\iota_{V}u$, then
\[(F-r^{u}+V(\lambda))c=0,\]
\[Hc+\lambda c+V(r^{u}c)=0.\]
The form $u$ can be written as $u=c(-r^{u}\beta-\lambda\alpha+\psi)$.
}
\end{Remark}

\subsection{Horocyclic invariance of SRB measures}
Next, we show that SRB measures exhibit some horocyclic invariance. Let us fix a smooth volume form $\Omega$ on $\M$, and denote $\lambda := -(X + r^u)a_u$ (H\"older regular by Lemma \ref{lemma:regularity-r^s/u}), and recall that by \eqref{eq:commutator-X-Y})
\begin{equation}\label{eq:general-commutation}
	[X, Y^u] = -\lambda X - r^u Y^u.
\end{equation}

\begin{lemma}[Horocyclic invariance III]\label{lemma:horo-SRB}
	Let $W \in C^{\infty}(\mc{M})$. Let $\alpha_W$ be the H\"older continuous solution of
	\begin{equation}\label{eq:alpha-W}
		(X + r^u)\alpha_W = (Y^u - \lambda) W.
	\end{equation}
	Then the following commutation relation holds:
	\[[X + W, Y^u + \alpha_W] = -\lambda (X + W) - r^u (Y^u + \alpha_W).\]
	Consequently, if the SRB measure is given by $\SRB = f\Omega$, for some $f\in \mc{D}'_{E_u^*}(\mc{M})$ satisfying $(X + \divv_{\Omega} X) f = 0$, then we have
	\begin{equation}\label{eq:SRB-half-invariance}
		(X + \divv_{\Omega} X + r^u)(Y^u + \alpha_{\divv_\Omega X})f = 0.
	\end{equation}
\end{lemma}
	\begin{proof}
		The first result is a straightforward computation:
		\begin{align*}
			[X + W, Y^u + \alpha_W] &= - \lambda X - r^u Y^u + X (\alpha_W) - Y^u (W)\\
			&= -\lambda (X + W) - r^u(Y^u + \alpha_W) + \underbrace{\lambda W + r^u \alpha_W + X(\alpha_W) - Y^u(W)}_{= 0,\,\,\mathrm{by}\,\,\eqref{eq:alpha-W}},
		\end{align*}
		where we use \eqref{eq:general-commutation} in the first equality. 
		
		For the conclusion about SRB measures, set $W := \divv_{\Omega} X$. Since $Y^u \in C^{1 + \alpha}(\mc{M}; T\mc{M})$, and $\alpha_W$ is H\"older regular, while $f\Omega$ is a measure, we may apply the commutation above to $f$ to directly obtain \eqref{eq:SRB-half-invariance}.
	\end{proof}
	
	\begin{Remark}\rm
		The operator $X + \divv_{\Omega}X + r^u$ is invertible (e.g. on $L^1(\M, \Omega)$), however the space to which $(Y^u + \alpha_{\divv_{\Omega} X})f$ apriori belongs is very irregular. However, it is expected that \eqref{eq:SRB-half-invariance} implies $(Y^u + \alpha_{\divv_{\Omega} X})f = 0$, and we will indeed see this is the case in the next section under a regularity assumption on the weak unstable bundle.
	\end{Remark}

\subsection{Case of one smooth weak bundle} In this section we assume that the weak unstable vector bundle $\mathbb{R}X \oplus E_u$ is smooth, or equivalently, that the vector field $Y^u$ is smooth. Indeed, this is clear by the construction of $Y^{u}$ at the top of Section \ref{sec:horocyclic}; the fact that $Y^{u}$ is somewhat non-canonical does not play any role and what we care about is its regularity and the fact that it is pointwise linearly independent of $X$. It follows that $r_V \in C^\infty(\M)$ and by inspection of the proof of Lemma \ref{lemma:regularity-r^s/u}, we have $r^u \in C^\infty(\M)$. Then the proof of Lemma \ref{lemma:horo-I} simplifies, and we get horocyclic invariance for resonant $1$-forms $u$ that are not closed, i.e. we have:

\begin{lemma}\label{lemma:horo-smooth}
	Assume $Y^{u} \in C^\infty(\mc{M}; T\M)$. Then:
	\begin{enumerate}
			 \item[1.] Let $u \in \Res^1$ such that $\iota_X u = D \in \mathbb{C}$. Then $\iota_{Y^u} u = D a_u$.
			 \item[2.] Let $u \in \Res_0^1$. Then, there is a constant $D \in \mathbb{C}$ such that for $c:= \iota_V u$:
			 	\begin{align}\label{eq:horo}
					\begin{split}
							(X + [X, V]_H \cdot r_V - [X, V]_V)c &= 0,\\
							(Y^u + [H, V]_H \cdot r_V + V(r_V) - [H, V]_V)c&= -\frac{D}{2}\SRB(X, H, V).
					\end{split}
					\end{align}
					Denote by $\mc{S}$ the set of distributional solutions of \eqref{eq:horo} for $D \in \mathbb{C}$. Then the map $P:\Res_0^1 \ni u \mapsto \iota_{V}u \in \mc{S}$ is an isomorphism.
	\end{enumerate}
\end{lemma}
\begin{proof}
For the first claim, it can be checked that in fact the proof of the second part of Lemma \ref{lemma:horo-I} carries over in this setting, that is, the summary of Remark \ref{rem:non-smooth-measure-horo-II} applies. However, let us give a slightly different and more direct argument.

The idea is to express the contractions in two ways: via meromorphic extension of the resolvent (see \S \ref{ssec:resolvent}) and via the resolvent integral. Let us for simplicity first consider the case $D = 0$. Via the resolvent integral, we consider for large $\re(s)$ and $\gamma \in C^\infty(\mc{M}; \Omega_0^1)$ the expression, for some $x \in \mc{M}$
	\begin{align}\label{eq:integral-extension}
	\begin{split}
		\iota_{Y^u} (\Lie_X + s)^{-1} \gamma (x) &= \int_0^\infty e^{-st} \varphi_{-t}^* \gamma(x)(Y^u(x)) dt\\
		&=  \int_0^\infty e^{-st} \gamma(\varphi_{-t}x)(d\varphi_{-t}(x) Y^u(x)) dt\\
		&\leq C\|\gamma\|_{C^0} \|U^u\|_{C^0} \int_0^\infty e^{-t(\re(s) + \nu_{\min} - \varepsilon)} dt < \infty,
	\end{split}
	\end{align}
where $\varepsilon > 0$, $C = C(\varepsilon) > 0$, we recall by \eqref{eq:commutator-X-Y} that $Y^u = a_uX + U^u$, and by the Anosov property
\[\exists C' = C'(\varepsilon) > 0, \qquad \|d\varphi_{-t}(y) v\| \leq C' e^{-(\nu_{\min} - \varepsilon) t} \|v\|, \qquad \forall y \in \M, \forall v \in E_u(y).\]
Therefore $s \mapsto \iota_{Y^u} (\Lie_X + s)^{-1} \gamma \in C^0(\mc{M})$ is a holomorpic map in the region $\re(s) > - \nu_{\min}$. 

On the other hand, by \S \ref{ssec:resolvent} we know that $R_1^+(s) = (\Lie_X + s)^{-1}: C^\infty(\mc{M}; \Omega^1) \to \mc{D}'(\mc{M}; \Omega^1)$ admits a meromorphic extension to $s \in \mathbb{C}$. Since $Y^u$ is smooth, it follows that $s \mapsto \iota_{Y^u} (\Lie_X + s)^{-1} \gamma \in \mc{D}'(\mc{M})$ is meromorphic. Since the two extensions agree for $\re(s) \gg 1$, it follows they are identical and holomorphic for $\re(s) > -\nu_{\min}$. If $u \in \Res_0^1$, we have $u = \Pi_{1}^+ \gamma$ for some $\gamma \in C^\infty(\mc{M}; \Omega_0^1)$, and since $\Pi_{1}^+ = \frac{1}{2\pi i}\oint_{0} (\Lie_X + s)^{-1}\, ds$, we conclude that $\iota_{Y^u} u = \frac{1}{2\pi i} \oint_0  \iota_{Y^u} (\Lie_X + s)^{-1} \gamma\, ds = 0$, which concludes the proof of the first part.

For the case $D \neq 0$, it suffices to observe that upon replacing $\iota_{Y^u}$ with $\iota_{U^u}$ in \eqref{eq:integral-extension} and taking $\gamma \in C^\infty(\M; \Omega^1)$, we similarly get that $s \mapsto \iota_{U^u} (\Lie_X + s)^{-1}\gamma \in C^0(\M)$ is holomorphic for $\re(s) > -\nu_{\min}$. Next, expressing
\[\iota_{U^u} (\Lie_X + s)^{-1}\gamma = \iota_{Y^u} (\Lie_X + s)^{-1}\gamma - a_u(X + s)^{-1}\iota_X\gamma,\]
we conclude, using also Proposition \ref{prop:mapping} for the second term, that $s \mapsto \iota_{U^u} (\Lie_X + s)^{-1}\gamma$ is meromorphic (as a composition) for $\re(s) > -\varepsilon$ for some $\varepsilon > 0$ small enough. Then the conclusion follows similarly to before, since $u = \Pi_{1}^+ \gamma$ for some $\gamma \in C^\infty(\M; \Omega^1)$.

To derive \eqref{eq:horo}, note firstly that by Proposition \ref{prop:k=0,2,3} we may write $du = -\frac{D}{2} \iota_X \SRB$ for some $D \in \mathbb{C}$. The first equation of \eqref{eq:horo} follows from $\iota_X du = 0$, Item 1, and the computation carried out in \eqref{eq:horo-II-1}. The other equation is a consequence of the computation in \eqref{eq:horo-II-2}; note that the divergence term $V(r_Vc)$ can now be expanded and the leading term becomes $Y^u c$.

Finally, $P$ is clearly injective by Item 1, i.e. $\iota_V u = 0$ implies $u = 0$. In the other direction, if $c$ solves \eqref{eq:horo} we deduce that $\WF(c) \subset E_u^*$, where we also use that $\WF(\SRB) \subset E_u^*$. Moreover, we have $P\big(c(\psi - r_V \beta)\big) = c$ and it is straighforward to check that $c(\psi - r_V \beta) \in \Res_0^1$ by using the computations \eqref{eq:horo-II-1} and \eqref{eq:horo-II-2}.
\end{proof}

\begin{Remark}\rm
	The method of Lemma \ref{lemma:horo-smooth}, Item 1, applies in general to $u \in \Res_0^1$ under the assumption that the resolvent is a meromorphic map $(\Lie_X + s)^{-1}: C^\infty(\mc{M}; \Omega^1) \to C_*^{-1-}(\mc{M}; \Omega^1)$ for $s$ close to zero. In that case one may define the composition $\iota_{Y_{u}} (\Lie_X + s)^{-1}$ and the proof carries over.
\end{Remark}

\begin{lemma}\label{lemma:horocyclic-invariance-SRB-smooth}
	If $\SRB = f\Omega$ for some $f\in \mc{D}'_{E_u^*}(\mc{M})$, then
	\[(Y^u + \alpha_{\divv_{\Omega} X})f = 0.\]
\end{lemma}
\begin{proof}
	Observe firstly by Lemma \ref{lemma:regularity-r^s/u} that $\lambda = -(X + r^u)a_u$ satisfies $\WF(\lambda) \subset E_u^*$. Using that $r^u$ and $Y^u$ are smooth, it follows that the solution $\alpha_{\divv_{\Omega}X}$ of \eqref{eq:alpha-W} with $W = \divv_{\Omega}X$ satisfies $\WF(\alpha_{\divv_{\Omega}X}) \subset E_u^*$. Therefore $(Y^u + \alpha_{\divv_{\Omega} X})f$ has wavefront set in $E_u^*$, so it is a resonant state (at zero) of the operator $X + \divv_{\Omega}X + r^u$ by Lemma \ref{lemma:horo-SRB}. However, $s \mapsto (X + \divv_{\Omega}X + r^u + s)^{-1}: L^1(\M, \Omega) \to L^1(\M, \Omega)$ is holomorphic for $\re(s) > -\nu_{\min}$ and so this resonant state is equal to the zero function, completing the proof.
\end{proof}

\section{Perturbation theory}\label{sec:perturbation}

Let $X_0$ be a fixed smooth Anosov vector field generating a transitive Anosov flow. In this section we study the local manifold structure of flows with vanishing winding cycles using the Implicit Function Theorem.

\subsection{Anisotropic Sobolev spaces}
We will use the perturbation theory developed by Guedes Bonthonneau \cite{Bonthonneau-19}, for which we need to introduce anisotropic Sobolev spaces. Recall that these are given by, for any $k \in \mathbb{N}_0$
\[\mc{H}_{rG, t}(\mc{M}; \Omega^k) := e^{- r \Op(G)} H^t(\mc{M}; \Omega^k), \quad r \geq 0, \quad t \in \mathbb{R}.\]
Here $\Op$ is a quantisation procedure on $\mc{M}$, $G(x, \xi) = m(x, \xi) \log(1 + |\xi|)$ is a logarithmically growing symbol on $T^*\mc{M}$, and $m(x, \xi)$ satisfies certain conditions with respect to any vector field $X$ close to $X_0$ in $C^1(\mc{M}; T\mc{M})$ norm; these conditions are explicitly stated in \cite[eq. (4)]{Bonthonneau-19}. We note that the main theorem of \cite{Bonthonneau-19} features only $\mc{H}_{rG, 0}(\mc{M}; \Omega^k)$; however, it is observed at the end of \cite[Section 2]{Bonthonneau-19} that the methods carry over to $\mc{H}_{rG, t}(\mc{M}; \Omega^k)$ (with a different \emph{threshold}).

When $t = 0$ we will use the notation $\mc{H}_{rG}(\M; \Omega^k) := \mc{H}_{rG, 0}(\M; \Omega^k)$, and when clear from context, we will simply write $\mc{H}_{rG, t}$ for $\mc{H}_{rG, t}(\mc{M}; \Omega^k)$ and $\mc{H}_{rG}$ for $\mc{H}_{rG}(\M; \Omega^k)$. Denote the domain of $X$ by (note that the domain \emph{does} depend on $X$)
\[\mc{D}^X_{rG, t}(\mc{M}; \Omega^k) := \{u \in \mc{H}_{rG, t}(\mc{M}; \Omega^k) \mid \Lie_Xu \in \mc{H}_{rG, t}(\mc{M}; \Omega^k)\}.\]
For $N > 1$  and $\eta > 0$ denote 
\begin{align*}
	\mc{U}_{N, \eta} &:= \{X \in C^N(\mc{M}; T\mc{M}) \mid \|X - X_0\|_{C^N} < \eta\},\\
	\mc{R}_{N, \eta} &:= \{(X, s) \in \mc{U}_{N, \eta} \times \mathbb{C} \mid s\,\,\mathrm{is\,\,a\,\,resonance\,\,of} \Lie_X\,\,\mathrm{on}\,\,\Omega^k,\, \re(s) > - 1\}.
\end{align*}
Here $\eta$ will always be chosen small enough so that the flows in $\mc{U}_{N, \eta}$ are Anosov (this may be done by \cite[Corollary 5.1.11] {Fisher-Hasselblatt-19}).

We summarize the contents of \cite[Theorem 1 and Corollary 2]{Bonthonneau-19} in the following lemma (see also the end of \cite[Section 2]{Bonthonneau-19}):

\begin{lemma}\label{lemma:perturbation-theory}
	There exist $N_0 > 0$, $\eta> 0$, and $C_0 \in \mathbb{R}$ such that the following holds. For any $N > N_0$, for all $X \in \mc{U}_{N, \eta}$, and all $t \in \mathbb{R}$, $r > C_0 + |t|$
	\[\Lie_X + s: \mc{D}^X_{rG, t}(\mc{M}; \Omega^k) \to \mc{H}_{rG, t}(\mc{M}; \Omega^k), \quad \re(s) > -1\]
	is a Fredholm operator and its inverse, when $s$ is not a resonance, is given by $(\Lie_X + s)^{-1}$. Moreover, the set $\mc{R}_{N, \eta} \subset \mc{U}_{N, \eta} \times \mathbb{C}$ is closed and the resolvent $(\Lie_X + s)^{-1}$ is bounded locally uniformly away from $\mc{R}_{N, \eta}$.
\end{lemma}

We remark that Guedes Bonthonneau works in infinite regularity, but the microlocal methods from this reference only require the control of a finite number of derivatives and thus carry over.

\subsection{Derivative of the winding cycle} Now specialise to $\dim \mc{M} = 3$ and fix a smooth probability volume form $\Omega$ on $\M$. Set, for some large $N$ to be specified later
\begin{equation}\label{eq:def-W-W-W+}
	\mc{W}^\pm := \{X \in C^N(\mc{M}; T\mc{M}) \mid [\omega^\pm(X)]_{H^2(\M)} = 0\} \cap \mc{U}_{N, \eta}, \quad \mc{W} := \mc{W}^+ \cap \mc{W}^-,
\end{equation}
where $\omega^\pm(X)$ denotes the winding cycles of $X$. Our goal is to show eventually in \S \ref{ssec:banach-manifold} that $\mc{W}^\pm$ are locally transversal Banach manifolds near certain flows. Of course, for this we will assume that $X_0 \in \mc{W}^\pm$, however there is no need to do so in this subsection since we will only study the dependence of $\omega^\pm(X)$ on $X$.

Recall that $\omega^\pm(X) = \iota_X \Pi_3^{+, X} \Omega$, where for any $k$, $\Pi_k^{+, X}$ is the spectral projector at zero of $\Lie_X$ on $k$-forms. Here, we use that $\Pi_3^{+, X}\Omega =: \Omega_{\mathrm{SRB}}^{+, X}$ is the SRB measure of $X$; it is a probability measure since we assumed that $\Omega$ is a probability measure. For the remainder of this section in order to simplify the notation we drop the superscript $H$ in the holomorphic part at zero of the resolvent $R_k^{+, H, X}$ of $\Lie_X$ acting on $\Omega^k$ and write $R_k^{+, X}$ instead. When clear from context we will also drop the vector field $X$ from the notation. 

Let us first study the $X$ dependence in $\Pi_3^{+, X}$. 

\begin{lemma}\label{lemma:proj-regularity}
	There exist $N_0 > 0$, $\eta > 0$, and $C_0 \in \mathbb{R}$, such that for some $r > C_0 + 3$, $N > N_0$, the following map is $C^1$-regular: 
	\[\mc{U}_{N, \eta} \ni X \mapsto \Pi_{3}^{+, X}: \mc{H}_{rG}(\M; \Omega^3) \to \mc{H}_{rG, -3}(\M; \Omega^3).\]
	Moreover, for $N_0$ large enough, the following map is $C^1$-regular:
	\[\mc{U}_{N, \eta} \ni X \mapsto \iota_X \Pi_{3}^{+, X}: \mc{H}_{rG}(\M; \Omega^3) \to \mc{H}_{rG, -3}(\M; \Omega^2).\]
	Finally, for any $(X, Y) \in \mc{U}_{N, \eta} \times C^N(\mc{M}; T\mc{M})$ and $N > N_0$:
	\[D \big[\omega^+(X)\big] (Y) = \Pi_2^{+, X} \iota_Y \Omega_{\mathrm{SRB}}^{+, X} + dR_1^{+, X} \iota_X\iota_Y \Omega_{\mathrm{SRB}}^{+, X}.\]
\end{lemma}
Here, we remark that $D \big[\omega^+(X)\big] (Y)$ denotes the derivative of $\omega^+$ at $X$ in the direction of $Y$.
\begin{proof}
Let $\gamma$ be a small smooth contour around zero, such that $\Lie_{X_0}$ acting on $\Omega^3$ has no non-zero resonances on a neighbourhood of the closed domain that $\gamma$ bounds. We recall here that since $X_0$ is transitive, the rank of the spectral projector $\Pi_3^{+, X_0}$ is equal to $1$. By Lemma \ref{lemma:perturbation-theory}, there is a neighbourhood $V$ of $\gamma$ such that for some small enough $\eta > 0$, for any $X \in \mc{U}_{N, \eta}$, $\Lie_X$ has no resonances in $V$ and $(\Lie_X + s)^{-1}$ is uniformly bounded as a map on $\mc{H}_{rG}$. Therefore we may define
\begin{equation}\label{eq:projector-flow-neighbourhood}
	\widetilde{\Pi}_{3}^{+, X} := \frac{1}{2\pi i} \oint_\gamma (\Lie_X + s)^{-1} ds, \quad X \in \mc{U}_{N, \eta}.
\end{equation}
By \cite[Lemma 6.1]{Cekic-Paternain-20}, the rank of $\widetilde{\Pi}_{3}^{+, X}$ is constant for $X \in \mc{U}_{N, \eta}$ for small enough $\eta > 0$; therefore $\widetilde{\Pi}_{3}^{+, X} = \Pi_{3}^{+, X}$.

Note that by the resolvent identity, for any $X, Y \in \mc{U}_{N, \eta}$, and $s \in V$:
\begin{equation}\label{eq:resolvent-identity}
	(\Lie_X + s)^{-1} - (\Lie_Y + s)^{-1} = (\Lie_X + s)^{-1} (\Lie_Y - \Lie_X) (\Lie_Y + s)^{-1}.
\end{equation}
By Lemma \ref{lemma:perturbation-theory}, $(\Lie_Y + s)^{-1}$ is uniformly bounded for $s \in V$ as a map on $\mc{H}_{rG}$ and $(\Lie_X + s)^{-1}$ is uniformly bounded as map on $\mc{H}_{rG, -1}$, if $r > C_0 + 1$. We would like to estimate $\Lie_X - \Lie_Y = \Lie_{X - Y}$ as a map $\mc{H}_{rG} \to \mc{H}_{rG, -1}$. Equivalently, we would like to estimate the norm of the operator:
\[P = (1 + \Delta_g)^{-\frac{1}{2}} e^{r \Op(G)} \Lie_{X - Y} e^{- r \Op(G)}: L^2(\mc{M}; \Omega^k) \to L^2(\mc{M}; \Omega^k),\]
By the composition formula for pseudodifferential operators and the Calder\'on-Vaillancourt theorem (see \cite[Theorems 4.14 and 4.23]{Zworski-12} for these statements in the Euclidean setting), it follows that the norm $\|P\|_{L^2 \to L^2}$ depends on a finite number of derivatives of the symbols of the operators in the composition. Therefore if we take $N$ large enough:
\[\|\Lie_X - \Lie_Y\|_{\mc{H}_{rG} \to \mc{H}_{rG, -1}} = \|P\|_{L^2 \to L^2} \leq C\|X - Y\|_{C^N},\]
for some $C > 0$  (depending also on $r$). This estimate and \eqref{eq:resolvent-identity} finally show that 
\[\mc{U}_{N, \eta} \ni X \mapsto (\Lie_X + s)^{-1}: \mc{H}_{rG} \to \mc{H}_{rG, -1}\]
is Lipschitz continuous for $s \in V$. Therefore, taking the limits we get:
\[D \big[(\Lie_X + s)^{-1}\big] (Y) = -(\Lie_X + s)^{-1} \Lie_Y (\Lie_X + s)^{-1}, \quad (X, Y) \in \mc{U}_{N, \eta} \times C^N(\mc{M}; T\mc{M}),\]
as maps $\mc{H}_{rG} \to \mc{H}_{rG, -2}$. In fact, by the above discussion (e.g. we need to use the bound $\|\Lie_{X - Y}\|_{\mc{H}_{rG, -1}, \mc{H}_{rG, -2}} \leq C\|X - Y\|_{C^N}$ for some $C > 0$), the right hand side is only continuous as a map $\mc{H}_{rG} \to \mc{H}_{rG, -3}$, since for each of the operators we lose one derivative. We conclude that 
\[\mc{U}_{N, \eta} \ni X \mapsto (\Lie_X + s)^{-1}:\mc{H}_{rG} \to \mc{H}_{rG, -3}, \quad s \in V\]
is $C^1$ regular. Here throughout we ask that $r > C_0 + 3$.

Since $\Pi_3^{+, X} = \frac{1}{2\pi i} \oint_\gamma (\Lie_X + s)^{-1} ds$, the first claim follows. The second claim follows by the $C^1$ properties of the multiplication $C^{N}(\mc{M}) \times \mc{H}_{rG, -3}(\mc{M}) \to \mc{H}_{rG, -3}(\mc{M})$ for large enough $N$.

Next, we compute the first derivative of the SRB measure:
\begin{align}\label{eq:SRB-derivative-XX}
\begin{split}
	D \big[\Omega_{\mathrm{SRB}}^{+, X}\big](Y) &= -\frac{1}{2\pi i} \oint_{\gamma} (\Lie_X + s)^{-1} \Lie_Y (\Lie_X + s)^{-1}\Omega\\
	&= -(R_3^{+, X} \Lie_Y \Pi_3^{+, X} +  \Pi_3^{+, X} \Lie_YR_3^{+, X})\Omega = -dR_2^{+, X} \iota_Y  \Omega_{\mathrm{SRB}}^{+, X},
\end{split}
\end{align}
where in the second equality we used the residue theorem and the expansion \eqref{eq:laurent}, and in the last equality we used $\Lie_Y = d\iota_Y + \iota_Y d$, as well as $\Pi_3^{+, X} d = 0$, $dR_3^{+, X} = 0$, $d\Pi_{3}^{+, X} = 0$, and $dR_2^{+, X} = R_3^{+, X}d$. The final formula now easily follows:
\begin{align*}
	D \big[\iota_X \Omega_{\mathrm{SRB}}^{+, X}\big] (Y) &= \iota_Y \Omega_{\mathrm{SRB}}^{+, X} - \iota_Xd R_2^{+, X} \iota_Y  \Omega_{\mathrm{SRB}}^{+, X}\\
	&= \Pi_2^{+, X} \iota_Y \Omega_{\mathrm{SRB}}^{+, X} + dR_1^{+, X} \iota_X\iota_Y \Omega_{\mathrm{SRB}}^{+, X},
\end{align*}
where in the second equality we used $\Lie_X = \iota_X d + d\iota_X$, $\Lie_X R_2^{+, X} = \id - \Pi_2^{+, X}$ (which follows from \eqref{eq:laurent}), and $\iota_X R_2^{+, X} = R_1^{+, X} \iota_X$. This completes the proof.
\end{proof}

\begin{Remark}\rm
	In the dynamical systems literature, the regularity of the SRB measure with respect to perturbations has been intensively studied, and statements similar to Lemma \ref{lemma:proj-regularity} but in a different functional setting have been obtained, see \cite[Theorem A]{Ruelle-08}, \cite[Theorem 2]{Butterley-Liverani-07}, and \cite[Theorem 2.7]{Gouezel-Liverani-06}. The point of this lemma is to provide an alternative microlocal proof of the regularity statement, and to compute the first variation of the winding cycle in this setting. We also remark that in the dynamical systems literature the formula \eqref{eq:SRB-derivative-XX} for the derivative of the SRB measure is sometimes known as the \emph{linear response} formula.
\end{Remark}

\subsection{Banach manifolds $\mc{W}^\pm$}\label{ssec:banach-manifold} We are now in shape to study the local Banach manifold structure of the spaces $\mc{W}^\pm$. Assume that $X_0$ preserves a smooth probability volume $\Omega$, so that $\SRB = \SRBs = \Omega$, and assume that $[\omega^+] = [\omega^-] = 0$. We first prove an auxiliary result:

\begin{lemma}\label{lemma:auxiliary-transversal}
	For $Y \in C^\infty(\mc{M}; T\M)$, define $\mc{Z}^\pm(Y) := \Pi_2^\pm \iota_Y \Omega$. Then the following map is surjective, unless $X_0$ is a contact flow:
	\[\mc{Z}: C^\infty(\mc{M}; T\M) \ni Y \mapsto (\mc{Z}^+(Y), \mc{Z}^-(Y)) \in \Res^{2, \infty} \times \Res_*^{2, \infty}.\]
	If $X_0$ is contact, then $\ran(\mc{Z})$ is of codimension $1$, and the projection of $\ran(\mc{Z})$ to $H^2(\M) \times H^2(\M)$ is surjective. In either case, $\mc{Z}^+$ and $\mc{Z}^-$ are surjective onto $\Res^{2, \infty}$ and $\Res^{2, \infty}_*$, respectively, which further project surjectively to $H^2(\M)$.
\end{lemma}
\begin{proof}
	Assume for the sake of contradiction that the map $\mc{Z}$ is not surjective. Then, by Proposition \ref{prop:non-degenerate}, there is non-zero $(u_*, u) \in \Res_*^{1, \infty} \times \Res^{1, \infty}$ such that
	\[\forall Y \in C^\infty(\mc{M}; T\M), \quad \int_{\M} u_* \wedge \Pi_2^+ \iota_Y \Omega + \int_{\M} u \wedge \Pi_2^- \iota_Y \Omega = 0.\]
	It follows by \eqref{eq:proj-transpose-2} that
	\begin{equation}\label{eq:auxiliary-1}
		\forall Y \in C^\infty(\mc{M}; T\M), \quad \int_{\M} \iota_Y(u + u_*)\, \Omega = 0,
	\end{equation}
	which translates into $u + u_* = 0$. By the wavefront set condition, it follows that $u \in C^\infty(\M; \Omega^1)$. By using the contraction/expansion properties, we get that $u|_{E_u \oplus E_s} = 0$, and since $\iota_{X_0}u \in \Res^0$, we have $\iota_{X_0} u$ is a constant. By \cite[Theorem 2.3]{Hurder-Katok-90} it follows that the flow is either a suspension of an Anosov diffeomorphism on a torus contradicting $[\omega^+] = 0$, or $u$ is a contact form. This completes the proof of the first claim.
	
	For the second claim, assume that $X_0$ is contact with contact form $\alpha \in C^\infty(\M; \Omega^1)$, satisfying $\iota_{X_0} \alpha = 1$ and $\Lie_{X_0} \alpha = 0$; then we may assume (up to re-scaling) that $\Omega = \alpha \wedge d\alpha$. Since $\iota_{X_0} \Res^{2, \infty} \subset \Res^1_0$ (here we use that $\Lie_{X_0}$ is semisimple on $\Omega_0^1$ as shown by \cite[Proposition 3.1(3)]{Dyatlov-Zworski-17}), and every $u \in \Res^{2, \infty}$ decomposes invariantly as $u = \alpha \wedge \iota_{X_0} u + (u - \alpha \wedge \iota_{X_0} u)$, we conclude that 
	\begin{equation}\label{eq:res-2}
		\Res^{2, \infty} = \alpha \wedge \Res_0^1 \oplus \Res_{0}^{2} = \alpha \wedge \Res_0^1 \oplus \mathbb{C}d\alpha,
	\end{equation}
	where the semisimplicity of $\Lie_{X_0}$ on $\Omega_0^2$ was proved in Proposition \ref{prop:k=0,2,3}; hence in particular $\Res^{2} = \Res^{2, \infty}$ and similarly $\Res^2_* = \Res^{2, \infty}_*$. Since $(\Pi_2^{\pm})^T = \Pi_1^{\mp}$ by \eqref{eq:proj-transpose-2} and $\Pi_1^\pm \alpha = \alpha$, we have
	\[
		\int_{\M} \alpha \wedge \Pi_2^+ \iota_Y \Omega = \int_{\M} \alpha \wedge \iota_Y \Omega = \int_{\M} \alpha \wedge \Pi_2^- \iota_Y \Omega,
	\]
	and so we get that
	\[\ran(\mc{Z}) \subset \Big\{(u, u_*) \in \Res^2 \oplus \Res_*^2 \mid \int_{\M}\alpha \wedge (u - u_*) = 0\Big\}.\]
	We claim that the equality holds. We have $\mc{Z}(X_0) = (d\alpha, d\alpha)$ and so using \eqref{eq:res-2} it suffices to show that 
	\[ C^\infty(\mc{M}; T\M) \ni Y \mapsto (\iota_{X_0} \mc{Z}^+(Y), \iota_{X_0} \mc{Z}^-(Y)) = -(\Pi_1^+ \iota_Y d\alpha, \Pi_1^- \iota_Y d\alpha) \in \Res_0^1 \times \Res_{0*}^1\]
	is surjective. Here in the equality we used that $\Omega = \alpha \wedge d\alpha$ and $\iota_{X_0} \mc{Z}^\pm(Y) = \Pi_1^\pm \iota_{X_0} \iota_Y \Omega = -\Pi_1^\pm \iota_Y d\alpha$. Again, if not, by Lemma \ref{lemma:semisimple} there exist $u_* \in \Res_{0*}^1$ and $u \in \Res_0^1$ (at least one of which is non-zero) such that for all $Y$
	\[0 =\int_{\M} \alpha \wedge (u_* \wedge \Pi_1^+ \iota_Y d\alpha + u \wedge \Pi_1^- \iota_Y d\alpha) = \int_{\M} \alpha \wedge (u_* + u) \wedge \iota_Y d\alpha = \int_{\M} \iota_Y(u + u_*)\, \Omega,\]
	where we used \eqref{eq:proj-tranpose} in the second equality, and that $(u + u_*) \wedge d\alpha = 0$ in the second equality. This implies $u + u_* = 0$, so $u \in C^\infty(\M; \Omega_0^1)$, which by the same argument as above translates to $u = 0$ (and so also $u_* = 0$). This contradicts the assumption and proves the claim about $\mc{Z}$. 
	
	That the projection onto $H^2(\M) \times H^2(\M)$ is surjective now follows from the observation that for a smooth closed $2$-form $\eta$
	\[\Pi_2^+ \eta = (\id - R_2^+ \Lie_{X_0})\eta = \eta - dR_1^+ \iota_{X_0} \eta,\]
	so $\Pi_2^+$ restricts to an injection on $H^2(\M)$ to $\Res^{2, \infty}$, and similarly for $\Pi_2^-$. Since $d\alpha$ is exact, we conclude that $\ran(\mc{Z})$ contains $\alpha \wedge \Res_0^1 \times \alpha \wedge \Res_{0*}^1$ which projects onto $H^2(\M) \times H^2(\M)$, completing the proof. 
	
	The final claim about the surjectivity of $\mc{Z}^\pm$ follows easily from the method in the first paragraph, and the fact that $\Res^{2, \infty}$ and $\Res_*^{2, \infty}$ surject onto $H^2(\M)$.
\end{proof}

Next, we apply the preceding lemma to show the Banach manifold structure:

\begin{lemma}\label{lemma:banach-manifold}
	Assume $X_0$ preserves a smooth volume and that the winding cycles are zero, i.e. $[\omega^+(X_0)] = [\omega^-(X_0)] = 0$. Then $\mc{W}^\pm$ are $C^1$ Banach submanifolds near $X_0$ of codimension $b_1(\M)$. The intersection $\mc{W}^+ \cap \mc{W}^-$ is transversal, that is, it is a $C^1$ Banach submanifold of codimension $2b_1(\M)$.
\end{lemma}
\begin{proof}
	Let us first show the claims about $\mc{W}^\pm$. Let $(\e_i)_{i = 1}^{b_1(\M)}$ be a basis of $H^1(\mc{M}; \mathbb{R})$ and for $X \in \mc{U}_{N, \eta}$ define the maps
	\begin{equation}\label{eq:F-pm-def}
		\mc{F}^\pm (X) = \Big(\int_{\mc{M}} \e_1 \wedge \iota_X \Pi_{3}^{\pm, X} \Omega, \dotso, \int_{\mc{M}} \e_{b_1(\M)} \wedge \iota_X \Pi_{3}^{\pm, X} \Omega\Big) \in \mathbb{R}^{b_1(\M)}.
	\end{equation}
	In suitable topologies given by Lemma \ref{lemma:proj-regularity}, the maps $\mc{F}^\pm$ are $C^{1}$, and we may compute their derivative along a $C^N$-regular vector field $Y$ as
	\begin{equation}\label{eq:F-pm-derivative}
		D_{X_0} \mc{F}^\pm(Y) = \Big(\int_{\mc{M}} \e_1 \wedge \Pi_2^+\iota_Y \Omega, \dotso, \int_{\mc{M}} \e_{b_1(\M)} \wedge \Pi_2^+\iota_Y \Omega\Big).
	\end{equation}
	We claim that $D_{X_0} \mc{F}^\pm$ are surjective. For that, it suffices to observe by Lemma \ref{lemma:auxiliary-transversal} that $\Res^{2, \infty}$ projects surjectively onto $H^2(\M)$ and $Y \mapsto \Pi_2^+\iota_Y \Omega$ is surjective onto $\Res^{2, \infty}$. Therefore, Poincar\'e duality implies the initial claim. An application of the Implicit Function Theorem \cite[Theorem 5.9]{Lang-99} shows that $\mc{W}^\pm$ are locally $C^1$ Banach manifolds of codimension $b_1(\mc{M})$, as required.

	Next, consider the map $\mc{F}(X) := (\mc{F}^+(X), \mc{F}^-(X)) \in \mathbb{R}^{2b_1(\M)}$. As in the paragraph above, we may compute the derivative of $\mc{F}$ at $X_0$, and conclude by Lemma \ref{lemma:auxiliary-transversal} that $D_{X_0}\mc{F}$ is surjective onto $\mathbb{R}^{2b_1(\M)}$ in both contact and non-contact cases. The surjectivity is then deduced by Poincar\'e duality and the claim follows by applying the Implicit Function Theorem.
\end{proof}

\begin{Remark}\rm
	It is expected that Lemma \ref{lemma:banach-manifold} is valid in the $C^\infty$-regularity setting. For that, we would need to show the assumptions of the Nash-Moser Implicit Function Theorem are satisfied (see \cite[Chapter 4]{Hamilton-82}), that is, that the involved mappings, and the right inverse for the derivative map, are \emph{tame in a neighbourhood} of $X_0$. Since this would produce further technical difficulties, we refrained from this exploration.
\end{Remark}

\begin{Remark}\rm
	Throughout this section we assumed that $X_0$ preserves a smooth volume $\Omega$. The reason is that otherwise, we would not be able to show the surjectivity of the derivative: for instance following the proof of Lemma \ref{lemma:auxiliary-transversal} to show the surjectivity of $\mc{Z}^+$, we would reach the point where $u_* \wedge \iota_{X_0}\SRB = 0$, for some $u_* \in \Res^{1, \infty}_*$. The problem is to show that $u_* = 0$, which is very similar to the issue of semisimplicity raised in Section \ref{sec:thermo} and Lemma \ref{lemma:product}.
\end{Remark}

\subsection{Resonant $1$-forms for nearly hyperbolic metrics}
Resonant 1-forms for hyperbolic metrics are special in the sense that for them it is possible to show that they are in 1-1 correspondence with holomorphic/anti-holomorphic 1-forms on the surface by means of a push forward map to Fourier modes of degree $\pm 1$ given by
\[\Res_{0}^{1}\ni u\mapsto (\iota_{V}u)_{1}+ (\iota_{V}u)_{-1}.\]
(See for example Lemma \ref{lemma:res01} below for the case $A=0$.) It is natural to ask if there is a similar correspondence for an arbitrary negatively curved metric. In this subsection we show that this is {\it not} the case. More precisely, the main result of this section is Proposition \ref{prop:non-holomorphic}, where we construct conformal perturbations for which this correspondence is invalid. As a by-product we derive Proposition \ref{prop:betaexact} to construct \emph{specific} Gaussian thermostats that have both winding cycles non-zero.

We begin by deriving a relation between resonant $1$-forms and Fourier modes. Let $(M, g)$ be a negatively curved surface and let $X$ be the geodesic vector field on the unit sphere bundle $\M = SM$. Let $\{X, H, V\}$ be the global frame introduced in \S \ref{ssec:geometry-surfaces}. For $\theta \in H^1(M)$, a real-valued $1$-form, we set: 
\begin{equation}\label{eq:def-theta-varphi-u}
	u:= \Pi_1^+\pi^*\theta = \pi^*\theta + d\varphi \in \Res_0^1, \quad \varphi := -R_0^{+} \iota_X \pi^*\theta, \quad c: = V\varphi \in C_*^{-1-}(\mc{M}),
\end{equation}
where the regularity claim follows from Lemma \ref{lemma:horo-I}. It follows from Proposition \ref{prop:geodesic-flow-horo-inv} that $u = c(-r^u \beta + \psi)$, where $r^u$ solves a Riccati equation and $Y^u = H + r^u V$, and 
\[(X - r^u)c = 0, \quad Hc + V(r^uc) = 0.\]
Reducing to Fourier modes, this translates to, for each $k \in \mathbb{Z}$ (see Lemma \ref{lemma:recurrence} for the case of the hyperbolic metric)
\begin{align}\label{eq:1-form-fourier}
\begin{split}
	2\eta_- c_{k + 1} - (k + 1)(r^u c)_k &= 0,\\
	2\eta_+ c_{k - 1} + (k - 1)(r^u c)_k &= 0.
\end{split}
\end{align}
Note that the Fourier modes $c_k$ and $(r^u c)_k$ are smooth thanks to the wavefront set condition and the fact that $V \not \in \mathbb{R}X \oplus E_u$. Indeed, by \eqref{eq:r_H/V-regularity} we have that $r^u \in \mc{D}'_{E_u^*}(\M) \cap C_*^{1 + \alpha}(\M)$, and from \eqref{eq:def-theta-varphi-u} it follows that $c \in \mc{D}'_{E_u^*}(\M) \cap C_*^{-1-}(\M)$. Therefore, the product $r^uc$ is well-defined, and by the usual wavefront set analysis moreover we may obtain that $\WF(r^uc) \subset E_u^*$ (see e.g. \cite[Theorem 8.2.10]{Hormander-90}). Let us first assume that $k =0$. On smooth functions $f$, the zeroth Fourier mode is (up to constant) equal to $f_0(x) = \int_{S_xM} f(x, v)\, dv$, where $dv$ denotes the volume form on $S_xM$, i.e. $f_0$ is the pushforward of $f$. The wavefront set calculus, see \cite[Proposition 4.19]{Melrose-03}, then shows that indeed $c_0$ and $(r^uc)_0$ are smooth, proving the claim. The case of non-zero $k$ is similar.

Restricting the first and second equations in \eqref{eq:1-form-fourier} to $k = -1$ and $k = 1$, respectively, we get $\eta_- c_0 = 0$ and $\eta_+ c_0 = 0$, which implies that $c_0$ is constant. Since $c = V\varphi$, integrating by parts implies that $c_0 = 0$. If $g$ is the hyperbolic metric we have $r^u \equiv 1$ and so $\eta_- c_1 = 0$. However, we will see in this section that $\eta_- c_1 \neq 0$ in general, by perturbation theory. 

\begin{prop}\label{prop:non-holomorphic}
	Let $g_0$ be a hyperbolic metric on $M$ and fix a closed non-exact (real) $1$-form $\theta$ on $M$. For an open and dense set of $h \in C^\infty(M)$, there exists an $\varepsilon > 0$ such that for $0 < |s| < \varepsilon$, for the metric $g_s := e^{-2sh} g_0$, the Fourier mode $c_1$ of $c$ defined by \eqref{eq:def-theta-varphi-u} is not holomorphic, i.e. $\eta_- c_1 \not \equiv 0$.
\end{prop}
\begin{proof}
By \eqref{eq:1-form-fourier} we have $\eta_- c_1 = \frac{1}{2} (r^u c)_0 = \frac{1}{2} (X c)_0$, so it suffices to consider the latter quantity. Denote by $SM_s$ the unit sphere bundle of $g_s$, by $\ell_s: SM_0 \to SM_s$, $\ell_s(x, v) = (x, e^{s h}v)$, the natural re-scaling map, and by $X_s$ and $V_s$ the geodesic and vertical vertical vector fields on $SM_s$, respectively; also denote by $r^u(s)$ the solution to the Riccati equation of $g_s$. Then by the third line of Lemma \ref{lemma:rescaling} (and \eqref{eq:X-simple})

\begin{equation}\label{eq:Z_s}
	Z_s := \ell_s^* X_s = e^{s \pi_0^*h} (X - sVX \pi_0^* h \cdot V), \quad \ell_s^*V_s = V.
\end{equation}
If $\pi(s)$ is the footpoint projection $SM_s \to M$, denote $\varphi(s) := -R^{+, X_s}_0 \iota_{X_s} \pi(s)^*\theta$ so that $\varphi(0) = \varphi$ and if we set $c(s) := V_s \varphi(s)$, then:
\begin{equation}\label{eq:d(s)}
	d(s) := \ell_s^*c(s) = - V R^{+, Z_s}_0 \iota_{Z_s} \pi^*\theta.
\end{equation}
In suitable topologies determined by Lemma \ref{lemma:proj-regularity}, we compute the derivative at zero of \eqref{eq:d(s)} (in various expressions, the dot denotes taking the derivative at zero):
\begin{equation}\label{eq:ddot}
\dot{d} = V \big(R^+_0 \dot{Z} \Pi^+_0 + \Pi^+_0 \dot{Z} R^+_0\big) \iota_X \pi^*\theta - V R_0^+ \iota_{\dot{Z}} \pi^*\theta = -V R_0^+ \big(\pi_0^*h \cdot \iota_X \pi^*\theta\big),
\end{equation}
where in the first equality we used the formula for the derivative of the resolvent derived in the proof of Lemma \ref{lemma:proj-regularity} and in the second one we used $V\Pi^+_0 = 0$, $\Pi^+_0 \iota_X \pi^* \theta = 0$, as well as \eqref{eq:Z_s} and $\iota_V \pi^* \theta = 0$. Therefore from \eqref{eq:d(s)}
\begin{align*}
	\dot{\big(Z_s d(s)\big)_0} &= \big((\pi_0^*h \cdot X - VX \pi_0^*h \cdot V)c\big)_0 - \big(XVR_0^+(\pi_0^*h \cdot \iota_X \pi^*\theta)\big)_0\\
	&= -(X \pi_0^*h \cdot c)_0 + \big(H R_0^+ (\pi_0^*h \cdot \iota_X \theta)\big)_0 = (A + B)h,
\end{align*}
where in the first equality we used \eqref{eq:ddot}, in the second equality we used that $c_0 = 0$, $V^2 X\pi_0^*h = -X \pi_0^*h$, and \eqref{eq:surface-geometry}; in the last equality we used $X = \eta_+ + \eta_-$ and introduced the operators
\[A(\bullet) := -2\re\big(c_{-1} \eta_+ \pi_0^* (\bullet)\big), \quad B(\bullet) := \big(H R_0^+ (\pi_0^*(\bullet) \cdot \iota_X \theta)\big)_0,\] 
acting on $C^\infty(M)$. The operator $B$ is actually pseudodifferential of order zero, as follows from \cite[Proposition 4.1]{Cekic-Lefeuvre-21-3} (note that there we deal with $R_0^+ + R_0^-$ instead of $R_0^+$ but the same proof applies). The other operator $A$ is a differential operator of degree $1$. Since it has no degree zero terms, it is clearly non-zero since $c_{-1}$ is non-zero (in fact it is zero only at a finite number of points since $c_{-1}$ is anti-holomorphic). Therefore $A + B \in \Psi^1(M)$ is a non-zero pseudodifferential operator of order $1$ and hence it is non-zero on an open and dense set of $C^\infty(M)$.

We conclude that the first derivative of $\big(Z_s d(s)\big)_0$ at $s = 0$ is non-zero for an open and dense set of $h \in C^\infty(M)$. Since (where we identify the notation for the Fourier modes on $SM_s$ and $SM_0$)
\[\ell_s^*\big(r^u(s) c(s)\big)_0 = \frac{1}{2}\ell_s^*\big(X_s c(s)\big)_0 = \frac{1}{2} \big(Z_s d(s)\big)_0,\] 
it follows that for an open and dense set of $h \in C^\infty(M)$, there exists an $\varepsilon > 0$ small enough such that for $0 < |s| < \varepsilon$, $\big(r^u(s) c(s)\big)_0 \not \equiv 0$, i.e. $\big(c(s)\big)_1$ is \emph{not} holomorphic, which completes the proof.
\end{proof}

We conclude this section with an application to the winding cycles of thermostats:

\begin{prop} \label{prop:betaexact}
	Let $(M, g_0)$ be a hyperbolic surface. For an open and dense set of $h \in C^\infty(M)$, there exist $f \in C^\infty(M)$ and $\varepsilon, \delta > 0$, such that for any metric $g = e^{-2sh}g_0$ with $\frac{\varepsilon}{2} < |s| < \varepsilon$, any vector field $F_t := X_g + tX_g\pi_0^*f \cdot V_g$ with $0 < |t| < \delta$ generates an Anosov flow which has both non-zero winding cycles.
\end{prop}
\begin{proof}
	We adopt the notation of Proposition \ref{prop:non-holomorphic}, which shows that for any fixed non-exact closed (real) $1$-form $\theta$, there is an open and dense set of $h \in C^\infty(M)$, and $\varepsilon = \varepsilon(h) > 0$, such that $\big(r^u(s) c(s)\big)_0 \neq 0$ for $0 < |s| < \varepsilon$. Write $\Omega_s$ for the canonical volume from on $SM_s$.

	By Lemma \ref{lemma:proj-regularity}, the first variation at $t = 0$ of the winding cycle $[\omega^-(F_t)]$ in the direction of $\theta$ is:
	\begin{align*}
		\int_{SM_s} \pi^*\theta \wedge \Pi_2^{-, X_s}\big(X_s \pi_0^*f \cdot \iota_{V_s} \Omega_s) &= \int_{SM_s} \big(\pi^*\theta + d \varphi(s)\big) \wedge X_s \pi_0^*f \cdot \iota_{V_s} \Omega_s\\
		&= \int_{SM_s} c(s) X_s \pi_0^*f \cdot \Omega_s = -\int_{SM_s} \big(X_s c(s)\big)_0 \cdot \pi_0^*f \cdot \Omega_s. 
	\end{align*}
	By the proof of Proposition \ref{prop:non-holomorphic}, $\big(X_sc(s)\big)_0 \not \equiv 0$, so there exists an $f \in C^\infty(M)$ such that the last integral is non-zero, which shows $[\omega^-(F_t)] \neq 0$ for $\frac{\varepsilon}{2} < s < \varepsilon$ and $0 < |t| < \delta$. A similar argument applies to $[\omega^+(F_t)]$, or alternatively this follows from the reversibility of the flow, completing the proof.
\end{proof}

\section{Anosov flows with one smooth weak foliation}
\label{section:onesmooth}

In this section we look at a particular subclass of Anosov of flows on closed 3-manifolds, namely those that have one weak foliation of class $C^{\infty}$. To be definite, let us assume that this is the stable weak foliation.  In this context it is known that the flow is topologically orbit equivalent
to the suspension of a toral automorphism on $\mathbb{T}^2$ or, up to finite covers, to the geodesic flow of a hyperbolic surface \cite[Theorem 4.7]{Ghys_93}. We are mostly interested in flows that are not suspensions (since we already know what happens in that case) and we shall ignore finite covers for the sake of simplicity. Under these assumptions we may as well assume that $\mc{M}$ is the unit tangent bundle of a closed oriented hyperbolic surface $(M,g)$. 

Let $X$ denote an Anosov vector field with $C^{\infty}$ weak stable bundle on $SM$ and let $X_{0}$ denote the geodesic vector field of a hyperbolic metric $g$. By \cite[Theorem 5.3]{Ghys_93}, there is a diffeomorphism $f:SM\to SM$ that conjugates the weak folation of $X$ with the weak folation of $X_{0}$. This implies that there are $a,b\in C^{\infty}(SM)$ such that
\[f_{*}(X)=aX_{0}+b(H-V)\]
since $\{X_{0}, H-V\}$ is a basis for the weak stable bundle of $X_0$ (recall that the global frame $\{X_0, H, V\}$ on $SM$ was introduced in \S \ref{ssec:geometry-surfaces}). Without loss of generality we may remove the diffeomorphism $f$ from the notation and assume that we have an Anosov vector field $X$ on $SM$ of the form
\begin{equation}
X=aX_{0}+b(H-V)
\label{eq:Xsmooth}
\end{equation}
for some $a,b\in C^{\infty}(SM)$. This implies that the bundle spanned by $\{X_{0}, H-V\}$ is invariant under the flow of $X$ and thus it must be one of the weak bundles. By switching the signs of $a,b$ if necessary we may suppose it is the weak stable bundle. It follows that
\begin{equation}
E_{s, X}^* = E_{s, X_0}^*.
\label{eq:sameduals}
\end{equation}

\begin{lemma} We have $\Res_{0*}^1(X_{0})\subset \Res_{0*}^1(X)$ and $[\omega^{+}]=0$, where $\omega^{+}=\iota_{X}\Omega_{\text{\rm SRB}}^{+}$.
\label{lemma:omega+}
\end{lemma}

\begin{proof} Take $u_{*}\in \Res_{0*}^1(X_{0})$. We know that $du_{*}=0$ (for instance, by Lemma \ref{lemma:closed} below for $A = 0$) and $u_{*}(H-V)=0$ by horocyclic invariance, Lemma \ref{lemma:horo-I}. Since $u_{*}(X_{0})=0$ it follows that $u_{*}(X)=0$ and thus $u_{*}\in \Res_{0*}^1(X)$ since \eqref{eq:sameduals} holds. We also know that the map $\Res_{0*}^1(X_{0}) \to H^{1}(SM)$ given by $u_{*}\mapsto [u_{*}]$ is an isomorphism. In other words, given any smooth closed 1-form $\eta$ on $SM$ there is $f\in \mathcal{D}'_{E_{s}^*}(SM)$ such that $\eta+df\in \Res_{0*}^1(X_{0})$. Thus $\eta(X)+df(X)=0$ and
\[\int_{SM}\eta(X)\,\Omega_{\text{SRB}}^{+}=-\int df(X)\,\Omega_{\text{SRB}}^{+}=0,\]
showing $[\omega^+] = 0$ as desired.
\end{proof}

Next, consider the standard volume form $\Omega$ in $SM$. Since $X_{0},H$ and $V$ preserve $\Omega$ we have
\[\text{div}_{\Omega} X= X_{0}(a)+(H-V)(b).\]
Let $\lambda\in C^{\infty}(SM; \mathbb{R})$ be given by $\lambda=\lambda_{-m}+\lambda_{m}\in H_{-m} \oplus H_{m}$ where $\eta_{-}\lambda_{m}=0$ and $m\geq 1$, and $\lambda_{-m} = \overline{\lambda_m}$. By Lemma \ref{lemma:eta-del-bar} we know that this implies that $X_{0}V\lambda=mH\lambda$. We make the following choices for $a$ and $b$:
\[a=1+V(\lambda)/m,\;\;\;\;b=-\lambda.\]
This gives
\begin{equation}
\text{div}_{\Omega} X= V\lambda.
\label{eq:divX}
\end{equation}
The case $m=1$, corresponds to $\lambda=\pi_{1}^*\theta$ (and hence $V\lambda=-\pi_{1}^*\star \theta$) where $\theta$ is a real harmonic 1-form (see the proof of Proposition \ref{prop:H1}). For this case we will prove:

\begin{prop} Let $(M,g)$ be a closed oriented hyperbolic surface and let $\theta$ be a non-zero harmonic 1-form. For $s\in \mathbb{R}$ consider the vector field
\[X_{s}=(1-s\pi_{1}^*\star\theta)X_{0}-s\pi_{1}^*\theta(H-V).\]
Then for all $s\neq 0$ sufficiently small, the Anosov vector field $X_{s}$ has the following properties:
\begin{enumerate}
\item[1.] The weak stable bundle is $C^{\infty}$ but the weak unstable bundle is only of class $C^{1+\alpha}$ for some $\alpha>0$;
\item[2.] $[\omega^{+}]=0$ and $[\omega^{-}]\neq 0$;
\item[3.] There is $u\in \Res_{0}^{1}(X_{s})$ that is not closed, but $\Res_{0*}^{1}(X_{s}) = \Res_{0*}^{1}(X_{0})$ and all its elements are closed;
\item[4.] Semisimplicity for the actions of $\Lie_{X_s}$ and $\Lie_{-X_s}$ on $\Omega^1$ hold and fail, respectively. Semisimplicity for the action of $\Lie_{\pm X_s}$ on $\Omega^1 \cap \ker \iota_{X_s}$ holds.
\end{enumerate}
\label{prop:newexamples}
\end{prop} 

\begin{proof} Note that
\[-s\pi^*\star\theta(X_{s})=sV\lambda(1+sV\lambda)+s^2\lambda^2=sV\lambda+s^{2}[(V\lambda)^2+\lambda^2],\]
where in the second equality we used that $\pi^*\theta(H) = \pi_1^*\theta$, which is a straightforward computation which follows from \eqref{eq:isothermal-X-H}. Thus
\begin{equation}
-s\int_{SM}\pi^*\star\theta(X)\,\Omega_{\text{SRB}}^{-}(s)=s\int_{SM}V\lambda\,\Omega_{\text{SRB}}^{-}(s)+s^2\int_{SM}((V\lambda)^{2}+\lambda^{2})\,\Omega_{\text{SRB}}^{-}(s).
\label{eq:epxs}
\end{equation}
Using \eqref{eq:divX} we see that
\[\mathbf{e}^{-}(s)=s\int_{SM}V\lambda\,\Omega_{\text{SRB}}^{-}(s)\]
is the entropy production of $\Omega_{\text{SRB}}^{-}(s)$. We know that $\mathbf{e}^{-}(s)\geq 0$ with equality if and only if $sV(\lambda)$ is a coboundary for $X_{s}$. Indeed, recall that by \cite[Theorem 1.2]{Ruelle-96}, the entropy production vanishes if and only if the flow of $X_{s}$ preserves a smooth measure. On the other hand, the flow of $X_{s}$ preserves a smooth measure if and only if its divergence $sV(\lambda)$ is a coboundary.
We claim that $sV(\lambda)$ cannot be coboundary for $s\neq 0$. If it were, the integral of $V(\lambda)$ along every closed orbit of $X_{s}$ would be zero. Since $X_{s}$ is topologically conjugate to $X_{0}$ and every homology class contains a closed orbit of $X_0$, this would imply that $[\star\theta]$ pairs to zero with each homology class, which is absurd since
$[\star\theta]\neq 0$. This gives $\mathbf{e}^{-}(s)>0$ for $s\neq 0$ and going back to \eqref{eq:epxs} we deduce that $[\omega^{-}(X_s)]\neq 0$ for $s\neq 0$.

The flow of $X_{s}$ has a smooth weak foliation by construction; hence Lemma \ref{lemma:omega+} gives $[\omega^{+}]=0$ and this shows Item 2.
To complete the proof of Item 1, it suffices to note that if the weak unstable bundle were also $C^{\infty}$, then Lemma  \ref{lemma:omega+} applied to $-X_{s}$ would give that $[\omega^{-}]=0$ which contradicts that $[\omega]^{-}\neq 0$.

Next, Item 3 follows from Theorem \ref{thm:general} and Item 2.

Finally, by Item 3, there is a $u \in \Res_0^1(X_s)$ that is non-closed, so semisimplicity for the action of $\Lie_{-X_s}$ on $\Omega^1$ fails by Proposition \ref{prop:semisimplicity-fails}. For the action of $\Lie_{X_s}$ on $\Omega^1$, we notice that by combining Item 2 and Proposition \ref{prop:res1}, we have $\dim \Res^1(X_s) = b_1(M) + 1 = \dim \Res^1(X_0)$. Therefore by the upper-semicontinuity of $\dim \Res^{1, \infty}$, see \cite[Lemma 6.2]{Cekic-Paternain-20}, we conclude that the required semisimplicity holds (note that $\Lie_{X_{0}}$ is semisimple).

The claim about the action of $\Lie_{X_s}$ on $\Omega^1 \cap \ker \iota_{X_s}$ follows from the fact that the action of $\Lie_{X_s}$ on $\Omega^1$ is semisimple, and the analogous claim about $\Lie_{-X_s}$ then follows directly by Lemma \ref{lemma:semisimple}.
\end{proof}

\begin{Remark} {\rm Anosov flows with weak bundles as in Item 1 of the previous proposition are always dissipative, i.e. they do not preserve a smooth volume. Indeed if a smooth volume form is being preserved, then both weak bundles would be smooth by \cite[Corollary 3.5]{Hurder-Katok-90}.}
\end{Remark}

\section{Thermostats}\label{sec:thermo}

In this section we focus our attention on a particular kind of flows.

Let $(M,g)$ be a closed oriented Riemannian surface and let $\lambda\in C^{\infty}(SM; \mathbb{R})$. We will be interested in the flow of the vector field $F=X+\lambda V$. The integral curves of $F$ have the form $(\gamma,\dot{\gamma})$, where $\gamma:\mathbb{R}\to M$ solves the ODE
\[\ddot{\gamma}=\lambda(\gamma,\dot{\gamma})J\dot{\gamma},\]
where the acceleration is computed using the Levi-Civita connection of $g$ and $J:TM\to TM$ is rotation by $\pi/2$ according to the orientation of the surface. We call these vector fields {\it thermostats}  (sometimes they are referred to as ``$\lambda$-geodesic flows"). The flow of $F$ models the motion of a particle under the influence of a force that is orthogonal to the velocity and with magnitude $\lambda$. For suitable choices of $g$ and $\lambda$ the vector field $F$ will be a dissipative Anosov flow. 

In \cite{Dairbekov-Paternain-07} it was proved that when $F$ is Anosov, $E_{u/s}$ are transversal to the vertical direction and so there are $r^{u/s} \in C^{1 + \alpha}(\mc{M})$ (as the weak stable/unstable bundles $\mathbb{R}F \oplus E_{u/s}$ are $C^{1 + \alpha}$, see Section \ref{sec:horocyclic}) functions such that $Y^{u/s} := H + r^{u/s} V \in \mathbb{R}F \oplus E_{u/s}$. In fact, they satisfy the Riccati equations (see \cite[Lemma 4.3]{Dairbekov-Paternain-07} or \cite[Proposition 8.19]{Merry-Paternain-11}):
\begin{equation}\label{eq:riccati}
	Fr^{u/s} + (r^{u/s})^2 + K - H\lambda + \lambda^2 - V\lambda \cdot r^{u/s} = 0.
\end{equation}
We may therefore compute the following commutators:
\begin{lemma}[Commutator stable/unstable]\label{lemma:commutator-stable/unstable}
	It holds
	\[[F, Y^{u/s}] = -\lambda F - r^{u/s} Y^{u/s}.\]
\end{lemma}
	\begin{proof}
		This is a straightforward computation, which we carry out just for $r^u$ (it is the same for $r^s$):
		\begin{align*}
			[X + \lambda V, H + r^u V] &= [X, H] + Xr^u \cdot V + r^u [X, V] - H\lambda \cdot V - \lambda [H, V] + \lambda Vr^u \cdot V - r^u V\lambda \cdot V\\ 
			&= KV + Xr^u \cdot V - r^u H - H\lambda \cdot V - \lambda X + (\lambda Vr^u - r^u V\lambda) V\\
			&= -\lambda F + (\lambda^2 + K + Fr^u - H\lambda - r^u V\lambda) V - r^u H\\
			&=- \lambda F - r^u(H + r^u V) = -\lambda F - r^u Y^u,		
		\end{align*}
		where we used \eqref{eq:surface-geometry} in the second equality and \eqref{eq:riccati} in the last one. This completes the proof.
	\end{proof}

Write $\SRB = f \Omega$ for the SRB measures of $F$, where $\Omega$ is the canonical volume form on $SM$ (see \S \ref{ssec:geometry-surfaces}) and where $f \in \mc{D}'_{E_u^*}(\mc{M})$ satisfies 
\begin{equation}\label{eq:SRB}
	0=(F + \divv_{\Omega} F)f = (F + V\lambda) f = Xf + V(\lambda f).
\end{equation}
In what follows we will identify $\SRB$ with $f$. Note that all the Fourier modes $f_k$ of $f$ are smooth thanks to the wavefront set condition and the fact that $V \not \in \mathbb{R}F \oplus E_{u/s}$, see the paragraph after \eqref{eq:1-form-fourier} where an argument is given. From \eqref{eq:SRB} it follows that $(Xf)_0 = 0$ and so $\eta_- f_1 + \eta_+ f_{-1} = 0$. Let $\theta$ be the real-valued $1$-form defined by the relation 
\begin{equation}\label{eq:theta}
	\pi_1^*\theta = f_1 + f_{-1}.
\end{equation}
Therefore $X_- \pi_1^*\theta = 0$ and by \eqref{eq:X-} we obtain:
\[d\star \theta = 0,\]
i.e. $\star \theta$ defines a cohomology class in $H^1(M)$. Recall that $\omega^{\pm}=\iota_{F}\Omega_{\text{SRB}}^{\pm}$.
\begin{lemma}\label{lemma:windind-cycle-theta}
	We have $[\star \theta]_{H^1(M)} = 0$ if and only if $[\omega^{+}] = 0$, i.e. the winding cycle of $\SRB$ vanishes.
\end{lemma} 
\begin{proof}
	Observe that $\omega^{+}$ is exact if and only if (by Poincar\'e duality) for every smooth closed real-valued $1$-form $\beta$ on $SM$ we have:
	\[0 = \int_{SM} \beta \wedge \iota_F \SRB = \int_{SM} f\beta(F) \,\Omega.\]
	Since the pullback $\pi^*:H^1(M) \to H^1(SM)$ is an isomorphism (this is well-known, see e.g. \cite[Corollary 8.10]{Merry-Paternain-11}), the previous condition is equivalent to
	\[\forall \beta \in C^\infty(M; \Omega^1) \cap \ker d, \quad 0 = \int_{SM} f \beta(X)\, \Omega.\]
	Since $\beta(X) = \pi_1^*\beta$, this is equivalent to
	\[0 = \int_{SM} (f_1 \beta_{-1} + f_{-1} \beta_1)\, \Omega = \langle{\pi_1^*\theta, \pi_1^*\beta}\rangle_{L^2(\mc{M})} = \pi \langle{\theta, \beta}\rangle_{L^2(M)} = \pi \int_M \beta \wedge \star \theta.\]
	By Poincar\'e duality, we conclude that the last equality is equivalent to $[\star \theta]_{H^1(M)} = 0$, which completes the proof.
\end{proof}

A case of particular interest arises when $\lambda$ is an odd function in the velocity variable as this results in a {\it reversible} vector field $F$.
Let $\mc{J}(x, v) = (x, -v)$ denote the flip map. It satisfies (see e.g. \cite[Proposition 5.1]{Cekic-Paternain-20}) $\mc{J}^*X = -X$, $\mc{J}^*V = V$, and $\mc{J}^*\lambda = -\lambda$,
and therefore $\mc{J}^*F = -F$. It follows that $\mc{J}^*$ sends $E_{u/s}$ to $E_{s/u}$ and $E_{u/s}^*$ to $E_{s/u}^*$. Moreover, the following holds.

\begin{lemma}
	We have $\SRB = \mc{J}^*\SRBs$ and the winding cycles satisfy 
	\[[\omega^{+}] = -[\omega^{-}].\]
	 Moreover, $\mc{J}^* \Res_{0*}^1 = \Res_0^1$. 
	 \label{lemma:reversible}
\end{lemma}
\begin{proof}
	Since $\mc{J}^* E_u^* = E_s^*$ and $\Lie_F \mc{J}^* = -\mc{J}^* \Lie_{F}$, we get $\mc{J}^*\Res_{0*}^k = \Res_0^k$ for any $k = 0, 1, 2, 3$. Since $\mc{J}$ is orientation preserving, the claim about SRB measures follows from Proposition \ref{prop:k=0,2,3}. For the other claim, it follows from $\mc{J}^*\omega^{-} = -\omega^{+}$ and the fact that $\mc{J}$ is homotopic to the identity (via rotations).
	\end{proof}
	
	In particular, this lemma applies to {\it Gaussian} thermostats, that is, when $\lambda=\pi^*_{1}\rho$, where $\rho$ is a smooth 1-form on $M$. Recall that if the curvature of $g$ is negative and $\rho$ is closed, the vector field $F$ is Anosov (see \cite[Remark 4.4]{Mettler-Paternain-19}).

\subsection{Coupled vortex equations} Let $(M, g)$ be a compact oriented surface. Recall that an interesting class of thermostats is obtained when $\lambda$ is generated by a holomorphic differential $A$ of degree $m\geq 2$, that is, a section of $\mc{K}^{\otimes m}$, where $\mc{K} = (T_{\mathbb{C}}^*M)^{1, 0}$ is the holomorphic part of $T_{\mathbb{C}}^*M$, and it satisfies the \emph{coupled vortex equations} (see Appendix \ref{app:A} for the notation)
\begin{equation}\label{eq:coupled-vortex-equations}
	\bar{\partial} A = 0, \quad K_g = -1 + (m - 1)|A|_{g}^2.
\end{equation}
It is known that when \eqref{eq:coupled-vortex-equations} holds, the flow of $F$ for 
\begin{equation}\label{eq:lambda-A}
	\lambda := \im(\pi_m^*A)
\end{equation}
is Anosov, see \cite[Theorem 5.1]{Mettler-Paternain-19}. By \eqref{eq:A-a_m-norm-relation} and \eqref{eq:A-a_m-relation} it holds that
\begin{equation}\label{eq:lambda-A-norm-relation}
	4|\lambda_m|^2 = |A|_{g}^2, \quad \lambda_m = \frac{\pi_m^* A}{2i}.
\end{equation}
Moreover, when $m \geq 2$, $\bar{\partial} A = 0$ is equivalent to by Lemma \ref{lemma:eta-del-bar}:
\begin{equation}\label{eq:eta-lambda}
	\eta_- \lambda_m = \eta_+ \lambda_{-m} = 0.
\end{equation}
We finally note that by \cite[Lemma 5.2]{Mettler-Paternain-19}, we have 
\begin{equation}\label{eq:curvature-condition-coupled-vortex}
	-1 \leq K_g < 0.
\end{equation}

\subsection{Proof of Theorem \ref{thm:beta}} In this case $\lambda = \pi_1^*\rho$, where $\rho$ is harmonic and non-zero.
Let us compute  the entropy production:
		\begin{align*}
			0 \leq \mathbf{e}^{+} &= -\int_{SM} \divv_{\Omega}(F) \,\SRB = -\int_{SM} V\lambda \cdot f\,\Omega = \int_{SM} \pi_1^*(\star \rho) \pi_1^*\theta\, \Omega = \pi \int_{M} \star \rho \wedge \star \theta,
		\end{align*}
		where $\theta$ was defined by \eqref{eq:theta} and we used that $V\pi_1^* = -\pi_1^* \star$ on $1$-forms. Note that the first inequality follows from \S \ref{ssec:SRB-entropy}. If $[\star \theta]_{H^1(M)} = 0$ we have $\mathbf{e}^{+} = 0$ and by \cite[Theorem A]{Dairbekov-Paternain-07} we obtain that $[\star \rho]_{H^1(M)} = 0$, contradiction. By Lemma \ref{lemma:windind-cycle-theta} we conclude that the winding cycle $[\omega^{+}] \neq 0$. The claim that $[\omega^{+}]=-[\omega^{-}]$ is a consequence of Lemma \ref{lemma:reversible}. The claim that $m_{1,0}=b_{1}(M)-1$ follows from the classification in Theorem \ref{thm:general}.

		The remaining claims are consequences of the general theory of perturbations and \eqref{eq:order-of-vanishing-formula} as we now outline, following \cite[Section 7]{Cekic-Paternain-20}. In fact, for $0 \neq \rho \in \mc{H}^1(M)$ small enough, we will show that $\Lie_{F_\rho}$, where $F_\rho := X + \pi_1^*\rho\, V$, has a splitting non-zero resonance when acting on $1$-forms. This implies that the action of $\Lie_{F_\rho}$ on $\Omega^1 \cap \ker \iota_{F_\rho}$ is semisimple, by \cite[Lemma 6.2]{Cekic-Paternain-20} and the fact that $\dim \Res^1(F_\rho) = b_1(M)$ (as follows from the previous paragraph and Theorem \ref{thm:general}). 
		
		For $\rho$ sufficiently small, by Lemma \ref{lemma:perturbation-theory} there exists an anisotropic Sobolev space $\mc{H}_{rG, t} := \mc{H}_{rG, t}(SM; \Omega^1)$, for $t \in \mathbb{R}$, such that by a statement analogous to Lemma \ref{lemma:proj-regularity}, the following map is $C^1$-regular:
		\[\mc{T}: \mc{H}^1(M) \ni \rho \mapsto \Lie_{F_\rho} \widetilde{\Pi}_1^+(\rho)\alpha \in \mc{H}_{rG, -3},\]
where $\widetilde{\Pi}_1^+(\rho)$ is the projector onto resonant $1$-forms of $F_\rho$ near zero, defined analogously to \eqref{eq:projector-flow-neighbourhood}, and $\alpha$ is the contact form of $X$. We compute for $\rho \in \mc{H}^1(M)$ that
\begin{equation}\label{eq:derivative-T_0}
	D_0\mc{T}(\rho) = \partial_s|_{s = 0} \widetilde{\Pi}_1^+(s\rho) \Lie_{F_{s\rho}} \alpha =  \partial_s|_{s = 0} \widetilde{\Pi}_1^+(s\rho) s \Lie_{\pi_1^* \rho V} \alpha = \Pi_1^+(\lambda \beta),
\end{equation}
where in the first equality we used that $\Lie_{F_{s\rho}}$ and $\widetilde{\Pi}_1^+(s\rho)$ commute, in the second one we used $\Lie_X \alpha = 0$, and for the last one we recall that $\lambda := \pi_1^*\rho$. We claim that $D_0\mc{T}$ is injective. Indeed, let $u := \Pi_1^+(\lambda \beta)$ for $\rho \neq 0$. Observe that $\Pi_1^+(V\lambda \cdot \alpha) = 0$, for instance by noting that $\iota_X \Pi_1^+(V\lambda \cdot \alpha) = 0$, and using that the pairing $\llangle{\bullet, \bullet}\rrangle$ is non-degenerate by Proposition \ref{lemma:semisimple} and \cite[Proposition 3.1(3)]{Dyatlov-Zworski-17}. Thus, as $\pi^*(\star \rho) = -V\lambda \cdot \alpha + \lambda \beta$ (since $\pi_1^*(\star \rho) = -V\lambda$ and $\iota_H \pi^* (\star \rho) = \lambda$), we get $u = \Pi_1^+(\pi^*(\star \rho)) \neq 0$ (recall $\Pi_1^+: H^1(M) \to \Res_0^1$ is an isomorphism). Compute the following quantity:
\begin{equation}\label{eq:def-W}
	W := \int_{SM} u(\lambda V)\, \Omega =  \int_{SM} \alpha \wedge u \wedge \lambda \beta = \llangle{u, \lambda \beta}\rrangle = \llangle{u, \Pi_1^-(\lambda \beta)}\rrangle = \llangle{u, \mc{J}^*u}\rrangle,
\end{equation}
where in the last line we used $\Pi_1^+u = u$ and \eqref{eq:proj-tranpose}, and $\mc{J}^*\Pi_1^+ = \Pi_1^- \mc{J}^*$. By \cite[Lemma 7.2]{Cekic-Paternain-20}, $W > 0$ since $u \neq 0$, proving the claim about injectivity of $D_0\mc{T}$. Therefore, by Taylor's expansion $\mc{T}$ is injective close to the zero $1$-form (using that the unit ball in $\mc{H}^1(M)$ is compact). Thus for $\rho \neq 0$ small, we have that $\mc{T}(\rho) \neq 0$, and also, using that $\dim \Res^1(F_\rho) = b_1(M)$ (see previous paragraphs) and \cite[Lemma 6.2]{Cekic-Paternain-20} (the rank of $\widetilde{\Pi}_1^+$ is locally constant), either $\mc{T}(\rho)$ is a splitting resonant state corresponding to a non-zero resonance, in which case $\Lie_{F_\rho} \mc{T}(\rho) \neq 0$, or we loose semisimplicity at zero for the action of $\Lie_{F_\rho}$ on $\Omega^1$, in which case $\Lie_{F_\rho} \mc{T}(\rho) = 0$ (i.e. there is Jordan block of size $2$). In what follows, we show that there is a splitting resonance.

Next, we compute the derivatives of $\Lie_{F_\rho} \mc{T}(\rho)$, which is $C^4$-regular as a map $\mc{H}^1(M) \to \mc{H}_{rG, -t}$ for $\rho$ small enough and for some $t$ large enough (similarly as above). In fact, we have
\[
	D_0[\Lie_{F_\rho} \mc{T}](\rho) = \partial_{s}|_{s = 0} \Lie_{F_{s\rho}} \mc{T}(s\rho) = \Lie_{\pi_1^* \rho V} \underbrace{\mc{T}(0)}_{= 0} + \Lie_X \Pi_1^+(\pi_1^*\rho \cdot \beta) = 0,
\]
where in second equality we used \eqref{eq:derivative-T_0} and $\mc{T}(0) = 0$, and in the third one that $\Lie_X \Pi_1^+ = 0$. The second and third derivatives are computed as in \cite[Section 7.2]{Cekic-Paternain-20}. Indeed, we have
\begin{align*}
	D_0^2 \Lie_{F_\rho} \mc{T}(\rho) = \partial_{s}^2|_{s = 0} \Lie_{F_{s\rho}} \mc{T}(s\rho) = 2 \partial_{s}|_{s = 0} \Lie_{F_{s\rho}} \underbrace{\widetilde{\Pi}_1^+(s\rho)(\pi_1^*\rho \cdot \beta)}_{=: u_s} = 2\dot{z} u = 0,
\end{align*}
at $s = 0$, where $u_s$ satisfies $\Lie_{F_{s\rho}} u_s = z_s u_s$ for some $z_{s} \in \mathbb{R}$ and similarly to \cite[eq. (7.9)]{Cekic-Paternain-20}, the map $s \mapsto z_s$ is $C^2$-regular. Therefore the last equality follows from the first linearisation in \cite[Section 7.2]{Cekic-Paternain-20}, which shows $\dot{z} = 0$. For the third derivative, we similarly have:
\begin{align*}
	D_0^3 \Lie_{F_\rho}\mc{T}(\rho) &= \partial_s^3|_{s = 0} \Lie_{F_{s\rho}} \widetilde{\Pi}_1^+(s\rho) (s\pi_1^*\rho \cdot \beta) = 3 \partial_s^2|_{s = 0} (z_{s} u_{s}) = 3 \ddot{z} u,
\end{align*}
where in the last line we used that $z = \dot{z} = 0$ at $s = 0$. By the same argument as in the second linearisation of \cite[Section 7.2]{Cekic-Paternain-20}, and using that by \eqref{eq:def-W} we have $W > 0$, we conclude that $\ddot{z} < 0$, and since $u \neq 0$ by the discussion above this proves that the third derivative is non-zero when $\rho \neq 0$.

By means of a Taylor expansion, it follows that the map $\mc{H}^1(M) \ni \rho \mapsto \Lie_{F_\rho} \mc{T}(\rho) \in \mc{H}_{rG, -t}$ satisfies the property that $\Lie_{F_\rho} \mc{T}(\rho) \neq 0$ for $\rho \neq 0$ small enough, and so a resonance splits (in fact, as the proof shows necessarily to right half-plane), proving the semisimplicity claim and completing the proof. \qed

\section{Quadratic holomorphic differentials} \label{sec:qhd} In this section we consider the coupled vortex equations \eqref{eq:coupled-vortex-equations} for $m = 2$ and we prove Theorem \ref{thm:QFF}. These thermostats are \emph{quasi-Fuchsian} flows as defined by Ghys \cite{Ghys-92}. It is conjectured in \cite{Paternain-07} that all quasi-Fuchsian flows arise this way. In this case, the weak stable/unstable bundles are \emph{smooth} and moreover by \cite[\S 6.1]{Mettler-Paternain-19}
\begin{equation}\label{eq:ru-rs}
	r^{u} = 1 + \frac{V\lambda}{2}, \quad r^{s} = -1 + \frac{V\lambda}{2}.
\end{equation}
(The case $m=3$ corresponds to Hilbert geodesic flows and the cases $m\geq 4$ are largely unstudied  (see \cite{Mettler-Paternain-19}). The case $m=2$ is the only one for which it is possible to write down explicitly the solutions $r^{u/s}$ to the Riccati equation.)

Recall $\lambda = \lambda_{-2} + \lambda_2$ is a holomorphic differential of degree two, that is, $\eta_- \lambda_2 = \eta_+ \lambda_{-2} = 0$; $\overline{\lambda_2} = \lambda_{-2}$ since $\lambda$ is real valued. Also, recall that $Y^{u/s} = H + r^{u/s}V$ (see Section \ref{sec:thermo}), and that we write $\SRB = f\Omega$ for some $f \in \mc{D}'_{E_u^*}(\M)$ and $\Omega = \alpha \wedge \psi \wedge \beta$.

Let us first restate the horocyclic invariance result of Lemma \ref{lemma:horo-smooth} in this setting.

\begin{lemma}[Horocyclic invariance for holomorphic differentials]\label{lemma:horo-m=2}
	Let $u \in \Res_{0}^1$. Then there is a constant $c \in \mathbb{C}$ such that $du = \frac{c}{2} f \iota_F \Omega$ and $h := \iota_{Y^s} u$ satisfies:
	\begin{align}\label{eq:horo-m=2}
	\begin{split}
		(F + r^s)h &= 0,\\
		(Y^u - \lambda)h &= cf.
	\end{split}
	\end{align}
	Denote by $\mc{S}$ the set of distributional solutions of \eqref{eq:horo-m=2} for $c \in \mathbb{C}$. Then the map $\Res_{0}^1 \ni u \mapsto \iota_{Y^s}u \in \mc{S}$ is an isomorphism.
\end{lemma}
\begin{proof}
	By Lemma \ref{lemma:horo-smooth} we have that $\iota_{Y^u} u = 0$. Next, similarly to Lemma \ref{lemma:horocyclic-invariance-new}, the first equation is derived from:
	\[0 = du(F, Y^s) = Fh - \iota_{[F, Y^s]}u = (F + r^s)h,\]
	where in the second equality we used $\iota_F u = 0$, and in the third one we used Lemma \ref{lemma:commutator-stable/unstable}. For the second equation, we have:
	\begin{equation}\label{eq:blabla}
		\frac{c}{2} f \Omega(F, Y^u, Y^s) = du(Y^u, Y^s) = Y^uh - \iota_{[Y^u, Y^s]}u,
	\end{equation}
	where in the last equality we used $\iota_{Y^u}u = 0$. Then we have the following computation:
	\begin{align}\label{eq:Y^u/s-commutator}
	\begin{split}
		[Y^u, Y^s] &= [H + r^uV, H + r^s V] = H r^s\cdot V + r^s [H, V] - H r^u\cdot V - r^u [H, V] + r^u Vr^s \cdot V - r^sVr^u \cdot V\\
		&= H(r^s - r^u)\cdot V + (r^s - r^u)\cdot X + (r^uVr^s - r^sVr^u) V\\
		&= - 2X + \Big(\Big(1 + \frac{V\lambda}{2}\Big) \frac{V^2\lambda}{2} - \Big(-1 + \frac{V\lambda}{2}\Big) \frac{V^2 \lambda}{2}\Big) V\\
		&= -2X - 4\lambda V = -2F - \lambda(Y^u - Y^s),
	\end{split}
	\end{align}
	where we used \eqref{eq:surface-geometry} in the second line, that $r^s - r^u = -2$ and \eqref{eq:ru-rs} in the third line, and $V^2\lambda = -4\lambda$ and $Y^u - Y^s = 2V$ in the final line. Also, compute
	\[\Omega(F, Y^u, Y^s) = \alpha \wedge \psi \wedge \beta (X + \lambda V, H + r^u V, H + r^s V) = -r^s + r^u = 2,\]
	which combined with \eqref{eq:Y^u/s-commutator} and \eqref{eq:blabla} shows \eqref{eq:horo-m=2} and concludes the proof of the first claim.
	
	For the final claim, the proof is straightforward and analogous to the proof of the final part of Lemma \ref{lemma:horo-smooth} and we omit it.
\end{proof}
\subsection{Recurrence relations}

We will use the ladder operators $\eta_\pm$ to derive from Lemma \ref{lemma:horo-m=2} the recurrence relations for the solutions $h$ of the system in \eqref{eq:horo-m=2}.

\begin{lemma}
	The PDE system \eqref{eq:horo-m=2} is satisfied if and only if, for every $k \in \mathbb{Z}$:
	\begin{align}\label{eq:recurrence-relation}
	\begin{split}
		2\eta_- h_{k + 1} - (k + 1) h_k + 2\lambda_{-2} h_{k + 2} i(k + 1) &= ic f_k,\\
		2\eta_+ h_{k - 1} + (k - 1) h_k + 2\lambda_2 h_{k - 2} i (k - 1) &= -ic f_k.
	\end{split}
	\end{align}
\end{lemma}
\begin{proof}
	To derive these equations, we first rewrite the system \eqref{eq:horo-m=2} more explicitly:
	\begin{align*}
		Xh + \lambda Vh + \Big(-1 + \frac{V\lambda}{2}\Big) h &= 0,\\
		Hh + \Big(1 + \frac{V\lambda}{2}\Big)Vh - \lambda h &= cf.
	\end{align*}
	Multiplying the second equation by $i$, and adding and subtracting from the first one we get:
	\begin{align*}
		2\eta_- h + \Big(\lambda + i + \frac{i V\lambda}{2}\Big) Vh + \Big(-1 + \frac{V\lambda}{2} - i\lambda\Big)h &= icf,\\
		2\eta_+h + \Big(\lambda - i - \frac{iV\lambda}{2}\Big) Vh + \Big(-1 + \frac{V\lambda}{2} + i\lambda\Big)h &= -icf.
	\end{align*}
	Now writing $\frac{V\lambda}{2} = i\lambda_2 - i \lambda_{-2}$, $\lambda = \lambda_{-2} + \lambda_2$, we re-write these equations as:
	\begin{align*}
		2\eta_- h + (i + 2\lambda_{-2}) Vh + (-1 - 2i\lambda_{-2})h &= icf,\\
		2\eta_+h + (- i + 2\lambda_2) Vh + (-1 + 2i\lambda_2)h &= -icf.
	\end{align*}
	Now rewriting these equalities with terms of fixed degree grouped, we get:
	\begin{align*}
		2\eta_- h + i Vh - h + 2\lambda_{-2}(Vh - ih) &= icf,\\
		2\eta_+h -i Vh  - h + 2\lambda_2(Vh  + ih) &= -icf.
	\end{align*}
	Equations in \eqref{eq:recurrence-relation} readily follow by identifying the degree $k$ components.
\end{proof}

We start with the cases $k = \pm 1$ in \eqref{eq:recurrence-relation}.

\begin{lemma}\label{lemma:k=+-1}
	Restricting to $k = \pm 1$ in \eqref{eq:recurrence-relation}, we get:
\begin{align}\label{eq:recurrence-relation-k=+-1}
	2\eta_- h_0 = ic f_{-1}, \quad 2\eta_+ h_0 = -icf_1.
\end{align}
This system of equations has a solution $h_0$, if and only if, $c = 0$ or the winding cycle $[\iota_F \SRB]_{H^2(\mc{M})}$ vanishes. These solutions, if they exist, are unique up to adding a constant.
\end{lemma}
\begin{proof}
	Note that by conjugating the first equation in \eqref{eq:recurrence-relation-k=+-1}, and adding and subtracting from the second one, we get equations for $\re(h_0)$ and $\im(h_0)$, and $c$ is replaced with $\re(c)$ and $\im(c)$, respectively (note that $f$ is real-valued, so $\overline{f_1} = f_{-1}$). This argument allows us to assume that $h_0$ and $c$ are real-valued to begin with. 
	
	Since $\eta_-$ is elliptic acting on $H_0 \cong C^\infty(M)$ and mapping to $H_{-1} \cong C^\infty(M, \mc{K}^{-1})$ (by Lemma \ref{lemma:eta-del-bar}), by Fredholm theory and since $(\eta_+)^* = -\eta_-$, if $c \neq 0$ the first equation has a solution if and only if 
	\begin{equation}\label{eq:fredholm-condition}
		\int_{\mc{M}} f_{-1} g_{1}\, \Omega = 0, \quad \forall g_1 \in \ker \eta_-|_{H_1}.
	\end{equation}
	Introducing $g_{-1} := \overline{g_1}$ for some $g_1 \in \ker \eta_-|_{H_1}$, we may write $\pi_1^*\gamma = g_1 + g_{-1}$ for some real-valued $1$-form $\gamma$ on the base. Therefore, using \eqref{eq:X-} $\gamma$ is co-closed, that is $d \star \gamma = 0$, and similarly $d\gamma = 0$ by using $V\pi_1^* = -\pi_1^* \star$. (In fact, the condition $\eta_- g_1 = 0$ is equivalent to the fact that $\gamma$ is a real-valued harmonic $1$-form on the base.)
	
	Recall $\pi_1^*\theta = f_1 + f_{-1}$ (where $\theta$ was defined in \eqref{eq:theta}), so the condition \eqref{eq:fredholm-condition} is equivalent to
	\[\int_{\mc{M}} (f_1 + f_{-1}) (g_1 + g_{-1})\, \Omega = \int_{\mc{M}} (f_1 + f_{-1}) (g_1 - g_{-1})\, \Omega = 0,\quad \forall g_1 \in \ker \eta_-|_{H_1},\]
	which in turn is simply equivalent to, using $g_1 - g_{-1} = -i V(g_1 + g_{-1}) = i \pi_1^*(\star \gamma)$,
	\[\int_{\mc{M}} \pi_1^*\theta \cdot \pi_1^*\gamma\, \Omega = \int_{\mc{M}} \pi_1^*\theta \cdot \pi_1^*(\star \gamma)\, \Omega = 0, \quad \forall \gamma \in \mc{H}^1(M),\]
	where $\mc{H}^1(M)$ denotes the set of harmonic $1$-forms. Since $\star \gamma \in \mc{H}^1(M)$ if and only if $\gamma \in \mc{H}^1(M)$, the second condition is redundant, and it is equivalent to have
	\[\int_M \star \theta \wedge \gamma = 0, \quad \forall \gamma \in \mc{H}^1(M).\]
	By Hodge decomposition, this is equivalent to $[\star \theta]_{H^1(M)} = 0$, so we conclude by Lemma \ref{lemma:windind-cycle-theta}.
	
	If $c = 0$, set $h_0 = \pi_0^* h_{00}$ for some $h_{00} \in C^\infty(M)$. By \eqref{eq:X-simple}, $\eta_+ h_0 = \eta_- h_0 = 0$ imply that $dh_{00} = 0$, and so $h_{00}$ and $h_0$ are constant functions. Similarly, this argument proves uniqueness of solution to \eqref{eq:recurrence-relation-k=+-1} up to constants, completing the proof.
\end{proof}

We now try to figure out how $h_1$ depends on $h_0$. To do this we have:

\begin{lemma}\label{lemma:recurrence}
	The system \eqref{eq:recurrence-relation} is equivalent to the following set of equations, valid for every $k \in \mathbb{Z}$:
	\begin{align}\label{eq:recurrence-equivalent}
		\begin{split}
		2(\eta_- - 2i \lambda_{-2} \eta_+) h_{k + 1} + (k + 1) (\underbrace{-1 + 4 \lambda_2 \lambda_{-2}}_{=K_g})h_k &= c (i f_k - 2 \lambda_{-2} f_{k + 2}),\\
		2(2i \lambda_2 \eta_- + \eta_+) h_{k + 1} + (k + 1)(\underbrace{-4\lambda_2 \lambda_{-2} + 1}_{=-K_g}) h_{k + 2} &= c(-2\lambda_2 f_k - if_{k + 2}).
		\end{split}
	\end{align}	
\end{lemma}
\begin{proof}
	We use the system \eqref{eq:recurrence-relation}, where we plug $k$ in the first equation and $k+2$ in the second one, to obtain:
	\begin{align*}
	\begin{split}
		2\eta_- h_{k + 1} - (k + 1) h_k + 2\lambda_{-2} h_{k + 2} i(k + 1) &= ic f_k,\\
		2\eta_+ h_{k + 1} + (k + 1) h_{k + 2} + 2\lambda_2 h_{k} i (k + 1) &= -ic f_{k + 2}.
	\end{split}
	\end{align*}
	The first equation of \eqref{eq:recurrence-equivalent} follows by multiplying the second equation by $2i\lambda_{-2}$ and subtracting from the first one; the second one follows from multiplying the first one by $2i\lambda_2$ and adding to the second one.
	
	That the two systems are equivalent can be seen as follows. Denote the first and the second equation of \eqref{eq:recurrence-relation} by $B = 0$ and $C = 0$, respectively. Then \eqref{eq:recurrence-equivalent} take the form 
	\[B - 2i\lambda_{-2}C = 0, \quad 2i\lambda_2 B + C = 0.\]
	From here it is easy to see that $B = C = 0$ is equivalent to this system of equations since
	\[1 - 4\lambda_2 \lambda_{-2} = 1 - |A|^2 = -K_g > 0,\]
	where we use \eqref{eq:lambda-A-norm-relation} in the first equality, \eqref{eq:coupled-vortex-equations} in the second, and \eqref{eq:curvature-condition-coupled-vortex} for the final inequality.
\end{proof}

Note that the leading operators arising in \eqref{eq:recurrence-equivalent} are conjugate to one another, that is:
\begin{equation}\label{eq:def-mu+-}
	\mu_- := \eta_- - 2i \lambda_{-2} \eta_+, \quad \mu_+ := \eta_+ + 2i\lambda_2 \eta_-, \quad \mu_+^* = -\mu_-.
\end{equation}
Here we also use \eqref{eq:eta-lambda}, so that $\lambda_2 \eta_- = \eta_- \lambda_2$, $\lambda_{-2} \eta_+ = \eta_+ \lambda_{-2}$, and $\eta_+^* = -\eta_-$. 

\begin{lemma}\label{lemma:closed}
	There are no solutions of \eqref{eq:recurrence-equivalent} with $c \neq 0$. In particular, we have $d(\Res_0^1) = 0$, i.e. all resonant $1$-forms in the kernel of $\iota_F$ are closed.
\end{lemma}
\begin{proof}
	Using the notation of Lemma \ref{lemma:horo-m=2}, it suffices to show that for $u \in \Res_0^1$ with $du = \frac{c}{2} \iota_F \Omega$, we have $c = 0$. By Lemma \ref{lemma:recurrence}, it is equivalent to show that the system \eqref{eq:recurrence-equivalent} implies that $c = 0$. Applying the first equation in \eqref{eq:recurrence-equivalent} for $k = 0$ and the second one for $k = -2$ we get:
	\begin{align*}
		2\mu_- h_1 + K_g h_0 &= c(if_0 -2\lambda_{-2}f_2),\\
		2\mu_+ h_{-1} + K_g h_0 &= c(-2\lambda_{2}f_{-2} - if_0).
	\end{align*}
	 Subtracting the two equations we obtain:
	 \[2(\mu_-h_1 - \mu_+ h_{-1}) = 2ic(f_0 - i(\lambda_2 f_{-2} - \lambda_{-2} f_2)) = 2ic\Big(f_0 - \frac{1}{2} (V\lambda \cdot f)_0\Big).\]
	 Integrating over $SM$ and using that $\mu_\pm^* = -\mu_{\mp}$, we get that the left hand side is zero, so:
	 \begin{equation}\label{eq:closed}
	 	0 = 2ic \int_{\mc{M}} \Big(f_0 - \frac{1}{2}(V\lambda \cdot f)_0\Big)\, \Omega.
	 \end{equation}
	 We have $\int_{\mc{M}} f_0\, \Omega = \int_{\mc{M}} \SRB = 1$. On the other hand, we have
	 \[-\int_{\mc{M}} (V\lambda \cdot f)_0\, \Omega = - \int_{\mc{M}} V\lambda \cdot f\, \Omega = -\int_{\mc{M}} \divv_{\Omega}(F)\, \SRB = \e^+(F),\]
	 where $\divv_{\Omega} F = V\lambda$ is the divergence of $F$ and we recall the entropy production $\e^+(F)$ was introduced in \S \ref{ssec:SRB-entropy}, and that it satisfies $\e^+(F) \geq 0$. We conclude that the integral on the right hand side of \eqref{eq:closed} is strictly positive, and so $c = 0$, completing the proof.
\end{proof}

Next, we show that the zero order Fourier mode of $h$ always vanishes.

\begin{lemma}\label{lemma:h_0}
	Assume $h$ is real valued and satisfies \eqref{eq:recurrence-equivalent}. Then $h_0 = 0$.
\end{lemma}
\begin{proof}
	By Lemma \ref{lemma:closed}, we know $c = 0$ in \eqref{eq:recurrence-equivalent} and \eqref{eq:recurrence-relation-k=+-1}. In fact, from Lemma \ref{lemma:k=+-1} it follows that $h_0$ is constant. Integrating the first equation of \eqref{eq:recurrence-equivalent}, for $k = 0$, and using $\mu_+^* = - \mu_-$, we get:
	\[0 = -2 \int_{\mc{M}} \mu_- h_1\, \Omega =  h_0 \cdot \int_{\mc{M}} K_g\, \Omega.\]
	By \eqref{eq:curvature-condition-coupled-vortex} we have $K_g < 0$, which implies $h_0 = 0$.
\end{proof}

Now we will construct $h_k$ by hand solving \eqref{eq:recurrence-equivalent} iteratively. For this we need to compute $\ker \mu_-$ on $H_1$, because by Lemma \ref{lemma:h_0}, $\mu_- h_1 = 0$ is the initial equation. This and other needed properties of the operators $\mu_\pm$ are proved in \S \ref{ssec:mu+-} below.

\begin{lemma}\label{lemma:res01}
	The map $\mc{J}: \mc{S} \ni h \mapsto (h_{-1}, h_1) \in \ker \mu_+|_{H_{-1}} \oplus \ker \mu_-|_{H_1}$ is an isomorphism on the set $\mc{S}$ of solutions of the system \eqref{eq:recurrence-equivalent}. Moreover, $\dim \Res_0^1 = b_1(M)$.
\end{lemma}
	\begin{proof}
		Recall first that by Lemma \ref{lemma:closed} for any solution $0 \neq h \in \mc{S}$ we have $c = 0$; also $h_0 = 0$ by Lemma \ref{lemma:h_0}. The map $\mc{J}$ is injective: if $h_1 = h_{-1} = 0$, then by applying first and second equations in \eqref{eq:recurrence-equivalent}, we get $h_k = 0$ for $k < 0$, and $h_k = 0$ for $k > 0$, respectively (here we use that $K_g < 0$ by \eqref{eq:curvature-condition-coupled-vortex}). Therefore $h \equiv 0$.
		
		To show surjectivity, let $h_1 \in \ker \mu_-|_{H_1}$. Then set $h_k = 0$ for $k \leq 0$ and observe that equations in \eqref{eq:recurrence-equivalent} take the form (plugging $k-2$ in the second one and $k$ in the first one), using $K_g = -1 + 4|\lambda_2|^2$:\footnote{Note that this agrees with \cite[Equations (3.12) and (3.13)]{Guillarmou-Hilgert-Weich-18} in the constant curvature case, i.e. $A \equiv 0$.}
		\begin{align}
		\label{eq:A}
			2 \mu_+ h_{k - 1} &= (k - 1) K_g h_k,\\
		\label{eq:B}
			2 \mu_- h_{k + 1} &= -(k + 1) K_g h_k.
		\end{align}
		Next, set $h_2, h_3, \dotso$ to be defined inductively using \eqref{eq:A} for $k = 2, 3, \dotso$ (note that they are all smooth, since $h_1$ is). Therefore \eqref{eq:A} is now satisfied for all $k$, by definition of $h_k$ for $k \leq 0$. Also, \eqref{eq:B} is satisfied for $k \leq -1$ trivially, and for $k = 0$ by definition of $h_1$. Therefore, we are left to check \eqref{eq:B} holds for $k \geq 1$.
		
		We prove this by induction on $k$. If \eqref{eq:B} holds for $k \leq \ell$ (starting with $\ell = 0)$, then we have:
		\begin{align*}
			2\mu_- h_{\ell + 2} &= 2\mu_- \Big(\frac{2\mu_+ h_{\ell + 1}}{K_g(\ell + 1)}\Big) = \frac{4}{\ell + 1} \mu_- \Big(\frac{1}{K_g} \mu_+ h_{\ell + 1}\Big)\\
			&= \frac{4}{\ell + 1} \Big( \frac{i}{2} K_g V h_{\ell + 1} + \mu_+ \Big(\frac{1}{K_g} \mu_- h_{\ell + 1}\Big)\Big)\\
			&= -2 K_g h_{\ell + 1} - 2 \mu_+ h_{\ell}\\
			&= -2 K_g h_{\ell + 1} - \ell K_g h_{\ell + 1} = -(\ell + 2) K_g h_{\ell + 1}. 
		\end{align*}		 
		Here we used \eqref{eq:A} for $k = \ell + 2$ in the first line, Lemma \ref{lemma:commutator} in the second line, the fact that $V h_{\ell + 1} = i(\ell + 1) h_{\ell + 1}$ and \eqref{eq:B} for $k = \ell$ in the third line, and finally \eqref{eq:A} for $k = \ell + 1$ in the last line. This proves that \eqref{eq:B} holds for $k = \ell + 1$, which completes the proof of the induction.
	
		It is left to show that $h := \sum_{k = 1}^\infty h_k$ converges in the distributional sense and for that it suffices to prove $\|h_k\|_{L^2(SM)} = \mc{O}(|k|^N)$ as $k \to \infty$ for some $N > 0$, since the Fourier modes of an arbitrary $\varphi \in C^\infty(SM)$ decay faster than any polynomial. In what follows norms and inner products will be in $L^2(SM)$. Notice that for $k \geq 2$:
		\begin{align*}
			\|h_k\|^2 &= \frac{2}{k - 1} \langle{h_k, \frac{1}{K_g} \mu_+ h_{k - 1}}\rangle = -\frac{2}{k - 1} \Big\langle{\mu_- \Big(\frac{1}{K_g} h_k\Big), h_{k - 1}}\Big\rangle\\
						  &=-\frac{2}{k - 1} \Big(\Big\langle\frac{1}{K_g} \mu_- h_k, h_{k - 1} \Big\rangle + \Big\langle{\mu_- \Big(\frac{1}{K_g}\Big) h_k, h_{k - 1}}\Big\rangle\Big)\\
						  &= \frac{1}{k - 1} \Big(k\|h_{k - 1}\|^2 - 2\Big\langle \mu_- \Big(\frac{1}{K_g}\Big) h_k, h_{k - 1} \Big\rangle\Big).			  
		\end{align*}
		Here we used \eqref{eq:A} and $\mu_+^* = -\mu_-$ in the first line, and \eqref{eq:B} (plugging in $k - 1$) in the third line. Therefore, by Cauchy-Schwarz and AM-GM inequalities:
		\begin{align*}
			\|h_k\|^2 \leq \frac{k}{k - 1} \|h_{k - 1}\|^2 + \frac{1}{k - 1} \Big(\|h_k\|^2 + \Big\|\mu_-\Big(\frac{1}{K_g}\Big)\Big\|_{L^\infty}^2 \|h_{k - 1}\|^2\Big),
		\end{align*}
		which gives, after setting $Z: = \Big\lceil \Big\|\mu_-\Big(\frac{1}{K_g}\Big)\Big\|_{L^\infty}^2 \Big\rceil \in \mathbb{N}$:
		\begin{align*}
			\|h_k\|^2 \leq \frac{k + Z}{k - 2} \|h_{k - 1}\|^2.
		\end{align*}
		Iterating the last inequality, we obtain:
		\begin{align*}
			\|h_k\|^2 \leq \frac{k + Z}{k - 2} \cdot \frac{k - 1 + Z}{k - 3} \cdots \frac{3 + Z}{2} \cdot \|h_{2}\|^2.
		\end{align*}
		It follows that $\|h_k\|^2 = \mc{O}(|k|^{Z + 2})$, which proves the claim and shows $\mc{J}(h) = (0, h_1)$.
		
				If $h_{-1} \in \ker \mu_+|_{H_{-1}}$, then $\overline{h_{-1}} \in \ker \mu_-|_{H_1}$, so the construction above gives an $\overline{h}$ with $\mc{J} \overline{h} = (0, \overline{h_{-1}})$, implying $\mc{J}h = (h_{-1}, 0)$. Thus $\mc{J}$ is surjective, completing the proof.
				
				The final claim now follows from Lemma \ref{lemma:index} below.
	\end{proof}
	
	Finally, we may show that the winding cycle of $\SRB$ is trivial.
	
	\begin{lemma}\label{lemma:winding-cycle-SRB}
		The winding cycle of $\SRB$ vanishes, i.e. $[\iota_F \SRB]_{H^2(\mc{M})} = 0$. Also, the helicity $\mc{H}(F)$ is non-zero.
	\end{lemma}
		\begin{proof}
					 By Lemmas \ref{lemma:closed} and \ref{lemma:res01}, we know $d(\Res_0^1) = 0$ and $\dim \Res_0^1 = b_1(M)$, respectively; in fact, by analogous arguments for coresonant states, i.e. arguing for the flow $-F$, the same lemmas imply $d(\Res_{0*}^1) = 0$ and $\dim \Res_{0*}^1 = b_1(M)$. Therefore, the map $T_*: \Res_{0*}^1 \to \mathbb{C}$ constructed (now for coresonances) in Lemma \ref{lemma:T} is trivial and the map $S_*: \Res_{0*}^1 \to H^1(\mc{M})$ constructed (again, for coresonances) in Lemma \ref{lemma:mapS} is an isomorphism. By the last part of Lemma \ref{lemma:mapS} we conclude that the winding cycle of $\SRB$ vanishes.
					 
					 The final conclusion now follows from the classification in Theorem \ref{thm:general}.
		\end{proof}

		So far, we have established the claims in Theorem \ref{thm:QFF} that assert that $m_{1,0}=b_{1}(M)$ and $[\omega^{\pm}]=0$, modulo properties of the operators $\mu_{\pm}$ to be proved next.

\subsection{Properties of $\mu_\pm$}\label{ssec:mu+-}
	Recall $\mu_\pm$ were defined in \eqref{eq:def-mu+-}. In this section we compute the principal symbols of $\mu_\pm$, size of their kernels, and also obtain the analogue of the formula \eqref{eq:eta+eta-commutator} for the operators $\mu_\pm$. According to \S \ref{ssec:geometry-surfaces}, we will freely identify $\mu_\pm$ with operators acting on sections of tensor powers $\mc{K}^{\otimes m}$ for $m \in \mathbb{Z}$.

\begin{lemma}\label{lemma:mu_pm-elliptic}
	For any $m \in \mathbb{Z}$, the operators $\mu_\pm$ are elliptic.
\end{lemma}
\begin{proof}
	As $\mu_+^* = -\mu_-$, it suffices to consider $\mu_-$ only; it also suffices to consider the case $m \geq 0$. In fact, for each $(x, \xi) \in T^*M$ the principal symbol $\sigma(\mu_-)(x, \xi): \mc{K}^{\otimes m}(x) \to \mc{K}^{\otimes(m - 1)}(x)$ is a linear map and hence may be identified with an element of $\overline{\mc{K}}(x)$ (the dual bundle). Consider local isothermal coordinates on $U \subset M$, that is, such that $g|_{U} = e^{2\psi}(dx^2 + dy^2)$ for some locally defined $\psi$. Fix an arbitrary $(z_0, \xi) \in T^*U$ and take $\chi, S \in C^\infty(M)$ such that $S(z_0) = 0$ and $dS(z_0) = \xi$, and $\supp(\chi) \subset U$ with $\chi = 1$ near $z_0$. Using Lemma \ref{lemma:eta-del-bar} (here we use the fact that the principal symbol of a pseudodifferential operator is recovered by oscillatory testing, see e.g. \cite[Theorem 4.19]{Zworski-12} in the related Euclidean case and when $S$ is linear; see also \cite[Definition 2.1]{Hormander-65})
	\begin{align*}
		\sigma(\mu_-)(z_0, \xi)(dz^m) &= \lim_{h \to 0} h \mu_-(\chi e^{i\frac{S}{h}} dz^m)(z_0)\\
		&= \lim_{h \to 0} h \Big(\frac{\partial (e^{i\frac{S}{h}})}{\partial \bar{z}}e^{-2\psi} \cdot dz^{m-1} + \overline{A_0} d\bar{z}^2 \otimes e^{2m\psi} \frac{\partial (e^{-2m\psi} e^{i\frac{S}{h}})}{\partial z} dz^{m + 1}\Big)(z_0)\\
		&= i\Big(\frac{\partial S}{\partial \bar{z}} + \overline{A_0} e^{-2\psi} \frac{\partial S}{\partial z}\Big)(z_0) \cdot e^{-2\psi(z_0)} dz^{m-1}
	\end{align*}
	where in the second line we wrote $A = A_0 dz^2$, so that by \eqref{eq:lambda-A}, we have $\lambda_2 = \frac{\pi_2^*A}{2i}$ and $\lambda_{-2} = -\frac{\pi_2^*(\overline{A_0} d\bar{z}^2)}{2i}$; in the final line, we also used $\pi_2^*(d\bar{z} \otimes dz)(z_0, v) = d\bar{z}(v) \cdot dz(v) = v_x^2 + v_y^2 = e^{-2\psi}$. Using the identification of the symbol with an element of $\overline{\mc{K}}(z_0)$, we conclude
	\begin{align*}
		\sigma(\mu_-)(z_0, \xi) = d\bar{z} \times i\xi\Big(\frac{\partial}{\partial \bar{z}} + \overline{A_0}(z_0) e^{-2\psi(z_0)}\frac{\partial}{\partial z}\Big).
	\end{align*}
	Recall that $|A| = |A_0| |dz|^2 = |A_0| e^{-2\psi}$ and that by \eqref{eq:coupled-vortex-equations} and \eqref{eq:curvature-condition-coupled-vortex} we have $|A|^2 = 1 + K_g < 1$. Therefore $|A_0| e^{-2\psi} < 1$ and writing $|A_0(z_0)| = r$ and $A_0(z_0) = re^{i\Upsilon}$ for some $\Upsilon \in \mathbb{R}$, the relation $\sigma(\mu_-)(z_0, \xi) = 0$ is equivalent to
	\begin{align*}
		0 &= \xi_x + r e^{-2\psi(z_0)} (\xi_x \cos \Upsilon - \xi_y \sin \Upsilon),\\
		0 &= \xi_y + r e^{-2\psi(z_0)}(-\xi_x \sin \Upsilon - \xi_y \cos \Upsilon). 
	\end{align*}
	If `$\cdot$' denotes Euclidean inner product in $\mathbb{R}^2$, this implies
	\[\xi_x^2 + \xi_y^2 = (r e^{-2\psi(z_0)})^2 (|(\xi_x, \xi_y) \cdot (\cos \Upsilon, -\sin \Upsilon)|^2 + |(\xi_x, \xi_y) \cdot (\sin \Upsilon, \cos \Upsilon)|^2) = (r e^{-2\psi(z_0)})^2(\xi_x^2 + \xi_y^2),\]
	Since $r e^{-2\psi(z_0)} < 1$, this is equivalent to $\xi = 0$ showing $\mu_-$ is elliptic and completing the proof.
\end{proof}

Now we compute the commutator $[\mu_+, \mu_-]$:

\begin{lemma}\label{lemma:commutator}
	We have
	\begin{align}\label{eq:commutator-mu+-1}
		[\mu_+, \mu_-] = - \frac{iK_g^2 V}{2} + \frac{\mu_+K_g}{K_g} \cdot \mu_- - \frac{\mu_-K_g}{K_g} \cdot \mu_+.
	\end{align}
	In fact, this is equivalent to the following formula:
	\[\Big[\frac{\mu_-}{K_g}, \frac{\mu_+}{K_g}\Big] = \frac{i}{2} V,\]
	or once again, equivalently:
	\[\mu_-\Big(\frac{1}{K_g} \mu_+\Big) = \frac{i}{2} K_gV + \mu_+ \Big(\frac{1}{K_g}\mu_-\Big).\]
\end{lemma}
	\begin{proof}
		Using the definition \eqref{eq:def-mu+-}, $[\mu_+, \mu_-]$ equals to:
		\begin{align*}
			[\eta_+ + 2i \lambda_2 \eta_-, \eta_- - 2i \lambda_{-2} \eta_+] &= [\eta_+, \eta_-] - 2i \eta_+ \lambda_{-2} \cdot \eta_+ - 2i \eta_- \lambda_2 \cdot \eta_-\\
			 &+ 4\big(\lambda_2 \eta_- (\lambda_{-2}) \cdot \eta_+ + |\lambda_{2}|^2 \eta_- \eta_+ - \lambda_{-2} \eta_+(\lambda_2) \cdot \eta_- - |\lambda_2|^2 \eta_+ \eta_-\big)\\
			 &= (1 - 4|\lambda_2|^2) [\eta_+, \eta_-] + 4\big(\eta_-(|\lambda_2|^2) \cdot \eta_+ - \eta_+(|\lambda_{2}|^2) \cdot \eta_-\big)\\
			 &= -K_g \cdot \frac{iK_gV}{2} + \eta_-K_g \cdot \eta_+ - \eta_+K_g \cdot \eta_-,
		\end{align*}
		where in the second equality we used $\eta_- \lambda_2 = \eta_+ \lambda_{-2} = 0$ and in the last equality the commutator \eqref{eq:eta+eta-commutator}, as well as \eqref{eq:coupled-vortex-equations}. Coming back to the definition \eqref{eq:def-mu+-} of $\mu_{\pm}$, we express $\eta_\pm$ in terms of $\mu_\pm$:
		\begin{equation}\label{eq:eta+eta-mu+mu-}
				\eta_- = -\frac{1}{K_g} \Big(2i\lambda_{-2} \mu_+ + \mu_-\Big), \quad \eta_+ = \frac{1}{K_g} \Big(2i\lambda_2 \mu_- - \mu_+\Big).
		\end{equation}
		Substituting \eqref{eq:eta+eta-mu+mu-} in the previous equality, $[\mu_+, \mu_-] + \frac{iK_g^2V}{2}$ equals to:
		\begin{align*}
		&-\frac{1}{K_g^2} \Big(2i\lambda_{-2} \mu_+ + \mu_-\Big)K_g \cdot \Big(2i\lambda_2 \mu_- - \mu_+\Big) + \frac{1}{K_g^2} \Big(2i\lambda_2 \mu_- - \mu_+\Big)K_g \cdot \Big(2i\lambda_{-2} \mu_+ + \mu_-\Big)\\
			 &= \frac{1}{K_g^2} \Big((2i \lambda_{-2} \mu_+ + \mu_-)K_g + (-4|\lambda_2|^2 \mu_- - 2i\lambda_{-2} \mu_+K_g)\Big)\cdot \mu_+\\
			 &+ \frac{1}{K_g^2} \Big((2i \lambda_2 \mu_- - \mu_+)K_g + (4|\lambda_2|^2 \mu_+ - 2i\lambda_2 \mu_-K_g)\Big) \cdot \mu_-\\
			 &= -\frac{\mu_- K_g}{K_g} \cdot \mu_+ + \frac{\mu_+ K_g}{K_g} \cdot \mu_-,
		\end{align*}
		where we used that $K_g = -1 + 4|\lambda_2|^2$ in the last line. This proves the first formula.
		
		For the second formula, compute:
		\begin{align*}
			\Big[\frac{\mu_-}{K_g}, \frac{\mu_+}{K_g}\Big] &= \frac{1}{K_g} \mu_- \Big(\frac{\mu_+}{K_g}\Big) - \frac{1}{K_g} \mu_+ \Big(\frac{\mu_-}{K_g}\Big)\\
			&= -\frac{1}{K_g^2}[\mu_+, \mu_-] + \frac{1}{K_g^3} \Big(\mu_+ K_g \cdot \mu_- - \mu_- K_g \cdot \mu_+\Big) = \frac{iV}{2},
		\end{align*}
		where we used \eqref{eq:commutator-mu+-1} in the last equality. The final formula is a straightforward restatement of this one.
	\end{proof}
		
		Finally, we are able to compute the index of $\mu_\pm$ explicitly:
			
	\begin{lemma}\label{lemma:index}
		We have $\ker \mu_+|_{H_m} = \{0\}$ for any $m > 0$. Also, $\ker \mu_+|_{H_0}$ is spanned by constant functions. Moreover, for any $m \geq 0$, the analytical index satisfies $\ind (\mu_+|_{H_m}) = \ind(\eta_+|_{H_m})$.
		
		In particular, $\dim \ker \mu_-|_{H_1} = \frac{1}{2} b_1(M)$.
	\end{lemma}
		\begin{proof}
			Assume $\mu_+ f = 0$, where $f \in H_m$ for some $m \geq 0$. Using the final identity of Lemma \ref{lemma:commutator}, multiplying with $\overline{f}$ and integrating over $\mc{M}$, we get:
			\begin{align*}
				0 &= \frac{i}{2}\int_{\mc{M}} K_gVf \cdot \overline{f}\, \Omega + \int_{\mc{M}} \mu_+\Big(\frac{1}{K_g} \mu_-f\Big) \cdot \overline{f}\, \Omega\\
				   &= -\frac{m}{2} \int_{\mc{M}} K_g |f|^2\, \Omega - \int_{\mc{M}} \frac{1}{K_g} |\mu_-f|^2\, \Omega \geq 0,
			\end{align*}
			since by \eqref{eq:curvature-condition-coupled-vortex}, $K_g < 0$, and we used $\mu_+^* = -\mu_-$. If $m > 0$, this shows $f = 0$ and completes the proof of the first part of the statement.
			
			For the second part, if $m = 0$, from the previous paragraph we get that $\mu_- f = 0$. By \eqref{eq:eta+eta-mu+mu-}, and from the assumption $\mu_+ f =0$, we conclude that $\eta_+ f= \eta_- f = 0$. This implies in particular that $X f = 0$, which by \eqref{eq:X-simple} implies $df=0$, which in turn gives that $f$ is a constant.
			
			Finally, let $P_t := \eta_- - t \cdot 2i\lambda_{-2}\eta_+$ be a continuous deformation of operators for $t \in [0, 1]$. By inspecting the proof of Lemma \ref{lemma:mu_pm-elliptic} the operators $P_t$ are indeed elliptic in this region, when acting on sections of $\mc{K}^{\otimes m}$ (or equivalently, on $H_m$) for any $m \geq 0$. By topological invariance of Fredholm index we conclude that $\mu_- = P_1$ and $\eta_- = P_0$ have identical indices, i.e. $\ind(\mu_-|_{H_m}) = \ind(\eta_-|_{H_m})$. In particular for $m = 1$ we get that, using $\mu_+^* = -\mu_-$
			\[\dim \ker \mu_-|_{H_1} - \dim \ker \mu_+|_{H_0} = \ind(\mu_-|_{H_1}) = \ind(\eta_-|_{H_1}) = \frac{1}{2}b_1(M) - 1,\]
			where in the last equality we used Proposition \ref{prop:H1}. Since $\dim \ker \mu_+|_{H_0} = 1$, the final claim follows.
		\end{proof}
		
		\subsection{Horocyclic invariance of the SRB measure}
		In this section we derive equations for the horocyclic invariance of the SRB measure. Let $a := a_u \in \mathcal{D}'_{E_u^*}(\mc{M})$ be the H\"older regular solution of: 
		\begin{equation}\label{eq:def-a}
			(F + r^u)a = -\lambda.
		\end{equation}
		Note that this $a$ is such that $U^u = Y^u - aF$ (the notation coming from Section \ref{sec:horocyclic}):
		\[[F, U^u] = [F, Y^u] - Fa \cdot F = -r^uY^u - (Fa + \lambda) \cdot F = -r^u(Y^u - aF) = -r^u U^u.\]
		Using the results of Section \ref{sec:horocyclic}, we have:
		\begin{lemma}\label{lemma:horo-SRB-m=2}
			If $\SRB = f \Omega$ is the SRB measure it holds that
			\[(Y^u - 2\lambda + 2a)f = 0.\]			
		\end{lemma}
		Note that $af$ is well-defined as $f\Omega$ is an actual \emph{measure}.
		\begin{proof}
		According to the notation of Lemma \ref{lemma:horo-SRB} and by Lemma \ref{lemma:horocyclic-invariance-SRB-smooth}, it suffices to show:
		\[\alpha_{\divv_{\Omega}F} = 2a - 2\lambda.\]
		Since $\divv_{\Omega} F = V\lambda$, for this it suffices to show the identity
		\[B:= (F + r^u) (2a - 2\lambda) - (Y^u - \lambda)V\lambda = 0.\]
			In fact, we compute by the definition \eqref{eq:def-a} of $a$:
			\begin{align*}
				B &= -2\lambda - 2\Big(X + \lambda V + \Big(1 + \frac{V\lambda}{2}\Big)\Big) \lambda - \Big(H + \Big(1 + \frac{V\lambda}{2}\Big)V\Big)V\lambda + \lambda V\lambda\\
				&= -2\lambda -2X\lambda - 2\lambda V\lambda - 2\lambda - \lambda V\lambda - HV\lambda + 4\lambda + 2\lambda V\lambda + \lambda V\lambda = 0,
			\end{align*}
			where we used that $(2X + HV)\lambda = 0$ (which follows from Lemma \ref{lemma:eta-del-bar}) and $V^2\lambda = -4\lambda$. This proves the claim.
		\end{proof}
		
		Lemma \ref{lemma:horo-SRB} implies that $f$ satisfies the system:
		\begin{align}\label{eq:SRB-system}
		\begin{split}
			(F + V\lambda)f = 0,\\
			(Y^u - 2\lambda + 2a)f = 0.
		\end{split}
		\end{align}
		Under the assumption that $f$ is a measure or more generally, that belongs to some Sobolev space with a small negative exponent (so that $af$ makes sense), this system implies that $f \Omega$ is really the SRB measure.
		
		\begin{Remark}\rm
			One might hope to explicitly solve \eqref{eq:SRB-system} similarly as we did in Lemma \ref{lemma:res01} for elements of $\Res_0^1$ and the system \eqref{eq:horo-m=2}. However, the issue here is that $a$ has \emph{infinite Fourier content}, that is, its degree is infinite (otherwise, it would be smooth, and so the bundle $E_u$ would be smooth, which generically does not happen \cite{Paternain-07}). Therefore, the recurrence relations stemming from \eqref{eq:SRB-system} would now involve \emph{all} Fourier modes of $f$, which complicates the situation significantly compared to \eqref{eq:recurrence-equivalent}.
		\end{Remark}
		
		\begin{Remark}\rm
		Note also that $\lambda_{\pm 2}$ vanish at finitely many points due to holomorphicity of $A$ (exactly at $4g - 4$ of points, the degree of the holomorphic bundle $\mc{K}^{\otimes 2}$). Therefore we see that the solution to $(F + V\lambda) f = 0$ is uniquely determined by $f_0, f_1, f_2$ (the $f_3, \dotsc$ are determined by dividing by $\lambda_{\pm 2}$ outside zeros of $A$ and by continuity on the zeros).
		\end{Remark}
		
		\subsection{Generic semisimplicity}
		
		In this section we tackle the question of semisimplicity for quasi-Fuchsian vector fields. We will need the following claim, proved in \cite[Lemma 4.8]{Cekic-Dyatlov-Delarue-Paternain-22}).
		
\begin{lemma}[Linear algebra lemma]\label{lemma:lin-alg}
	Let $V\subset \mathbb{M}_{n\times n}$ be a (real or complex) linear subspace of the set of $n\times n$ complex matrices such that for each $u,v\in \mathbb{C}^{n}\setminus\{0\}$, there exists $B\in V$ such that $(Bu,v) \neq 0$. (Here $(\bullet, \bullet)$ denotes the canonical inner product on $\mathbb{C}^n$.) Then $V$ contains a dense set of invertible matrices.
\end{lemma}

Next, we need to show that the product of certain distributions equals zero if and only one of the distributions itself is zero.

\begin{lemma}[Product lemma]\label{lemma:product}
	Let $P_\pm = Y + Q^\pm$, $P_{1, 2} = Y_{1, 2} + Q_{1, 2}$, be first order differential operators with $C^\infty$ coefficients, where $Y, Y_1, Y_2 \in C^\infty(\mc{M}; T\mc{M})$ form a smooth global frame and $Q^\pm$, $Q_{1, 2} \in C^\infty(\mc{M})$. Assume $u_\pm \in \mc{D}'(\mc{M})$ satisfy $P_{\pm} u_\pm = P_1 u_+ = P_2 u_- = 0$. Then the product $u_+ u_-$ is well-defined and 
	\[u_+ u_- = 0 \quad \implies \quad \supp (u_+)^{\mathrm{c}} \cup \supp (u_-)^{\mathrm{c}} = \mc{M}.\]
\end{lemma}
\begin{proof}
	The basic idea of the proof is to reduce the product of distributions to a tensor product in suitable coordinates. Since $\WF(u_+) \subset (\mathbb{R}Y \oplus \mathbb{R}Y_1)^\perp \subset T^*\mc{M}$ and $\WF(u_-) \subset (\mathbb{R}Y \oplus \mathbb{R}Y_2)^\perp$ have zero intersection, where we recall $\bullet^\perp$ denotes the annihilator of $\bullet$, the distributional product is well-defined by \cite[Theorem 8.2.10]{Hormander-90}. Let $x \in \mc{M}$ and pick flow-box coordinates $(x_1, x_2, x_3) \in [-\varepsilon, \varepsilon]^3$ for some $\varepsilon > 0$ near $x = (0, 0, 0)$, such that $Y = \partial_{x_1}$ and so
\begin{equation}\label{eq:u_pm}
	(\partial_{x_1} + Q_\pm)u_\pm = 0.
\end{equation}
By \eqref{eq:u_pm} we have $\WF(u_\pm) \subset \mathbb{R}dx_2 \oplus \mathbb{R}dx_3$, so since the conormal bundle of the slice $S_c := \{x_1 = c\}$ for each $c \in [-\varepsilon, \varepsilon]$ is $\mathbb{R}dx_1$, we may restrict $u_\pm$ to $S_c$ by the wavefront set calculus, see \cite[Corollary 8.2.7]{Hormander-90}. On each $S_c$, using \eqref{eq:u_pm} we may write the equations $P_1 u_+ = 0$ and $P_2 u_- = 0$ as 
	\[(\underbrace{a_+ \partial_{x_2} + b_+ \partial_{x_3}}_{=: Z_+} + Q_1') u_+ = 0, \quad (\underbrace{a_- \partial_{x_2} + b_- \partial_{x_3}}_{=: Z_-} + Q_2') u_- = 0,\]
	where $a_\pm, b_\pm, Q_{1, 2}'$ depend on $(c, x_2, x_3)$, and $Z_\pm$ are pointwise linearly independent in $TS_c$. Since the restriction of the product to $S_c$ is the product of restrictions (well-defined similarly to above since $Z_\pm$ are linearly independent), and the invariance \eqref{eq:u_pm} is valid, it suffices to show $u_+|_{S_0}$ or $u_-|_{S_0} = 0$ near $(0, 0)$.
	
	By using suitable smooth integrating factors, without loss of generality assume that $Q_1' = Q_2' = 0$, so that $Z_\pm u_\pm = 0$ on $S_0$. This reduces the problem to a statement in $2$ dimensions. We may put $Z_+$ into flow-box coordinates $(y_1, y_2) \in (-\delta, \delta)^2$ for some $\delta > 0$, defined on $S_0$ such that $x = (0, 0)$, $Z_+ = \partial_{y_1}$, and (possibly after a time-change) $Z_- = d \partial_{y_1} +  \partial_{y_2}$, where $d \in C^\infty((-\delta, \delta)^2)$, so that
	\[\partial_{y_1} u_+ = 0, \quad (d \partial_{y_1} + \partial_{y_2}) u_- = 0.\]
	
	Consider the flow $\alpha_t$ of $Z_-$, and the time $t = t(y_1, y_2) \in \mathbb{R}$, defined on a sub-domain $(-\delta_1, \delta_1)^2$ for some $0 < \delta_1 < \delta$, such that $\mathrm{pr}_2 \circ \alpha_{-t}(y_1, y_2) = 0$, where for $i = 1, 2$, $\mathrm{pr}_i$ denotes the projection onto the coordinate $y_i$. Consider the map
	\[G: (-\delta_1, \delta_1)^2 \to (-\delta, \delta)^2, \quad (y_1, y_2) \mapsto (\mathrm{pr}_1 \circ \alpha_{-t}(y_1, y_2), y_2),\]
	and note that after taking $\delta_1 > 0$ small enough, $G$ is a diffeomorphism onto its image by the inverse mapping theorem (note that $DG(0, 0) = \begin{pmatrix}
	1 & 0\\
	* & 1
	\end{pmatrix}$ on $\{y_2 = 0\}$ since $G(y_1, 0) = (y_1, 0)$). Denote by $(G_1, G_2)$ the components of $G$. By the wavefront set calculus, define $v_+ := u_+|_{\{0\} \times (-\delta, \delta)}$ and $v_- := u_-|_{(-\delta, \delta) \times \{0\}}$. Since pullback of distributions under submersions is well-defined, see \cite[Theorem 6.1.2]{Hormander-90}, the invariance $Z_\pm u_\pm = 0$ translates to $G_2^* v_+ = u_+$ and $G_1^* v_- = u_-$. 
	
	Let $\delta_2 > 0$ be such that $[-\delta_2, \delta_2]^2 \subset G((-\delta_1, \delta_1)^2)$ and $\delta_2 < \delta_1$. Let $\psi_1, \psi_2 \in C_0^\infty((-\delta_2, \delta_2))$ be arbitrary. Then:
	\begin{align*}
		& \int_{G((-\delta_1, \delta_1)^2)} v_-(z_1) v_+(z_2) \psi_1(z_1) \psi_2(z_2)\, dz_1\, dz_2\\
		&= \int_{(-\delta_1, \delta_1)^2} u_-(y_1, y_2) u_+(y_1, y_2) (G_1^*\psi_1)(y_1, y_2) \psi_2(y_2)  J(y_1, y_2)\, dy_1\, dy_2 = 0,
	\end{align*}
	since $u_+ u_- = 0$, and where $J$ is the Jacobian of $G$. Since $\psi_1, \psi_2$ were arbitrary we conclude that either $v_+|_{(-\delta_2, \delta_2)}$ or $v_-|_{(-\delta_2, \delta_2)}$ have to vanish and so either $u_+ = G_2^*v_+$ or $u_- = G_1^* v_-$ vanish in a neighbourhood of $x$, completing the proof.
\end{proof}

Finally, we formulate the main result of this section: for a generic time-change, semisimplicity is valid for $\Res_0^1$
as claimed in Theorem \ref{thm:QFF}.

\begin{lemma}\label{lemma:generic-semisimplicity}
	For an open and dense set of $\mathbf{a} \in C^\infty(\mc{M}; \mathbb{R}_{>0})$, the Lie derivative action $\Lie_{\mathbf{a}F}$ on $\Omega_0^1$ is semisimple.
\end{lemma}
\begin{proof}
For $\mathbf{a} \in C^\infty(\M; \mathbb{R})$ define the pairing
	\[A_{\mathbf{a}}(u, u_*) = \int_{\mc{M}} \mathbf{a} \alpha \wedge u \wedge u_*, \quad (u, u_*) \in \Res_0^1 \times \Res_{0*}^1.\]
	Further identifying $\Res_{0(*)}^1$ with $\mathbb{C}^{b_1(M)}$ (according to Lemma \ref{lemma:res01}), we may identify $A_{\mathbf{a}}$ with an $n \times n$ complex matrix in $\mathbb{M}_{n \times n}$. Introduce the finite dimensional real-linear vector space
	\[
		V:= \{A_\mathbf{a} \in \mathbb{M}_{n \times n} \mid \mathbf{a} \in C^\infty(\mc{M}; \mathbb{R})\}.
	\]	
	
	For $\mathbf{a} \in C^\infty(\mc{M}; \mathbb{R}_{>0})$, denote by $\llangle{\bullet, \bullet}\rrangle_{\mathbf{a}} = \llangle{\frac{1}{\mathbf{a}}\bullet, \bullet}\rrangle$ the pairing associated to $\mathbf{a}F$ introduced in \S \ref{ssec:pairing}. Since the weak stable/unstable bundles of $\mathbf{a}F$ and $F$ agree by \cite[Lemma 2.1]{DeLaLlave-Marco-Moryon-86}, the $\Res_{0(*)}^1$ spaces with respect to $\mathbf{a}F$ and $F$ agree. Thus $\llangle{\bullet, \bullet}\rrangle_{\mathbf{a}}$ may be identified with the matrix $A_{\frac{1}{\mathbf{a}}}$, and by Lemma \ref{lemma:semisimple} to prove the present lemma it suffices to show that for an open and dense set of $\mathbf{a} \in C^\infty(\M; \mathbb{R}_{> 0})$, the matrix $A_{\frac{1}{\mathbf{a}}}$ is invertible. Openness is immediate and it suffices to show density.

	We first show that $V$ contains a dense set of invertible matrices; again, since invertibility is an open condition, this set is also open. Let $0 \neq u \in \Res_0^1$ and $0 \neq u_* \in \Res_{0*}^1$. Then
	\[A_{\mathbf{a}}(u, u_*) = \int_{\mc{M}} \mathbf{a} h h_* \alpha \wedge \omega_u \wedge \omega_s,\]
	where $h = \iota_{Y^s}u$, $h_* = \iota_{Y^u} u_*$, and $\omega_{u/s} \in C^\infty(\mc{M}; T^*\mc{M})$ were defined in \eqref{eq:omega-u-s}; we also used Lemma \ref{lemma:horo-I} and its analogue for co-resonant states. If $A_{\mathbf{a}}(u, u_*) = 0$ for all $\mathbf{a} \in C^\infty(\mc{M}; \mathbb{R})$, then $h h_* \equiv 0$. By Lemmas \ref{lemma:horo-m=2} and \ref{lemma:closed}, $(F + r^s)h = (Y^u - \lambda)h = 0$, and similarly $(-F - r^u)h_* = (Y^s - \lambda)h_* = 0$ and so Lemma \ref{lemma:product} applies to give that $\supp(h)^{\mathrm{c}} \cup \supp(h_*)^{\mathrm{c}} = \mc{M}$. However, \cite[Theorem 1]{Weich-17} shows that $h$ and $h_*$ (being resonant states for $F + r^s$ and $-F -r^u$, respectively) either vanish everywhere or have full support, which implies that either $h \equiv 0$ or $h_* \equiv 0$, contradiction. Therefore, by Lemma \ref{lemma:lin-alg}, the claim about density follows.
	
	Observe that for an arbitrary open set $U \subset C^\infty(\M; \mathbb{R})$, we have $\{A_{\mathbf{a}} \in \mathbb{M}_{n \times n} \mid \mathbf{a} \in U\} \subset V$ is open, and so by density of invertible elements in $V$, we conclude that for a dense set of $\mathbf{a} \in C^\infty(\M; \mathbb{R})$, the matrix $A_{\mathbf{a}}$ is invertible. Since $C^\infty(\M; \mathbb{R}_{> 0}) \subset C^\infty(\M; \mathbb{R})$ is also open, the main claim follows.  
\end{proof}

\begin{Remark}\rm
	In fact, in order to establish generic semisimplicity under time-changes as in Lemma \ref{lemma:generic-semisimplicity} it is equivalent to show semisimplicity for only \emph{one} time-change. This follows from basic linear algebra.
\end{Remark}

\subsection{Computation of the helicity} By Lemma \ref{lemma:winding-cycle-SRB} we know that the winding cycles of $F = X + \lambda V$ are zero; therefore there exists a primitive $\tau^+ \in \Res^1$ such that
\begin{equation}\label{eq:primitive-helicity}
	d\tau^+ = \iota_F \SRB, \quad \iota_F \tau^+ = \mc{H}(F) =: B.
\end{equation}
Recall that $a$ was defined in \eqref{eq:def-a} and that it satisfies $Y^u = U^u + a F$; also recall $\SRB = f\Omega$ for some $f \in \mc{D}'_{E_u^*}(\M)$, where $\Omega$ is the canonical measure on $\M = SM$. Write $f^+ := f$ in this section. We first prove an auxiliary horocyclic invariance result similar to Lemma \ref{lemma:horo-m=2}, but valid for $\tau^+$.
\begin{lemma}
	Let $c := \iota_V \tau^+$. Then:
	\begin{align}\label{eq:tau-+-equations}
	\begin{split}
		(F + r^s)c &= -aB,\\
		(Y^u - \lambda)c &= -f^+ + B(1 + Va).
	\end{split}
	\end{align}
	Moreover, the following identity holds:
	\begin{equation}\label{eq:identity-random-success}
		 -(Y^u - \lambda) r^s = (F + r^u)\lambda.
	\end{equation}
\end{lemma}
\begin{proof}
	By the first part of Lemma \ref{lemma:horo-smooth} we know that $\tau^+$ satisfies 
	\[\iota_H \tau^+ + r^uc = \iota_{Y^u}\tau^+ = aB.\] 
	Using this relation, as well as $[F, V] = -(H + V\lambda \cdot V)$, the first equation of \eqref{eq:tau-+-equations} follows from expanding $d\tau^+(F, V) = 0$ similarly to Lemma \ref{lemma:horocyclic-invariance-new}. The second equation similarly follows by expanding the equation $d\tau^+(H, V) = -f^+$, as in Lemma \ref{lemma:horocyclic-invariance-new}.
	
	The final equation \eqref{eq:identity-random-success} follows from the following computation:
  	\begin{align*}
  	 	 -Y^ur^s + \lambda (r^s - r^u) - F\lambda &= -\Big(H + \Big(1 + \frac{V\lambda}{2}\Big)V\Big)\Big(-1 + \frac{V\lambda}{2}\Big) - 2\lambda - (X + \lambda V) \lambda\\
  	 	&= -\frac{HV\lambda}{2} + 2\lambda + \lambda V\lambda - 2\lambda - X\lambda - \lambda V\lambda\\
  	 	&= -\Big(\frac{HV}{2} + X\Big)\lambda = -2(\eta_- \lambda_{2} + \eta_+ \lambda_{-2}) = 0,
  	 \end{align*}
  	 which completes the proof.
\end{proof}

We are now in shape to compute $B$ in terms of $a$ and the entropy production (defined in \S \ref{ssec:SRB-entropy}). We will denote by $\vol_{\bullet}(M)$ the volume of $M$ equipped with the Riemannian metric $\bullet$.

\begin{lemma}\label{lemma:helicity-explicit}
	The following identity holds:
	\begin{equation}\label{eq:helicity-formula}
		\mc{H}(F) = \frac{1 + \frac{1}{2} \e^+ (F)}{2\pi \vol_g(M) + \int_{\M} a^2\, \Omega}.
	\end{equation}
	Next, assume $A \neq 0$ and set $A(s) = sA$ for $s \in \mathbb{R}$. Consider the Riemannian metric $g_s$ such that $(g_s, A(s))$ solves the coupled vortex equations \eqref{eq:coupled-vortex-equations}, and the associated thermostat vector field $F_s$. Then:
	\[\mc{H}(F_s) = \mc{O}\Big(\frac{1}{s}\Big), \qquad s \to \infty.\]
\end{lemma}
\begin{proof}
	Multiplying the first equation of \eqref{eq:tau-+-equations} by $\lambda$ and integrating with respect to $\Omega$, we get (we omit the volume form for simplicity)
	\begin{align*}
		-B \int_{\M} a \lambda &= \int_{\M} (F + r^s)c \cdot \lambda\\
		&= - \int_{\M} c (F + r^u) \lambda\\
		&= \int_{\M} c (Y^u -\lambda) r^s\\
		&= -\int_{\M} (Y^u - \lambda)c \cdot r^s\\
		&= -\int_{\M} \big(-f^+ + B(1 + Va)\big) r^s\\
		&= -\big(1 + \frac{1}{2}\e^+(F)\big) + B \cdot 2\pi \vol_g(M) - 2B \int_{\M} a\lambda.
	\end{align*}
	In the second line we integrated by parts and used \eqref{eq:ru-rs}, in the third line we used \eqref{eq:identity-random-success}, in the fourth line we integrated by parts and used \eqref{eq:ru-rs} again, in the fifth line we used the second equation of \eqref{eq:tau-+-equations}, in the sixth line we used that $\divv_\Omega F = V\lambda$, \eqref{eq:ru-rs}, the definition of entropy production, and integration by parts. To prove the first claim, it now suffices to observe
	\[-\int_{\M} a\lambda = \int_{\M} (F + r^u)a \cdot a = -\frac{1}{2}\int_{\M} a^2 (-2r^u + V\lambda) = \int_{\M} a^2,\]
	where in the first equality we used \eqref{eq:def-a}, in the second one we integrated by parts, and in the final one we used \eqref{eq:ru-rs}.
	
	Next, consider the family of solutions $g_s := e^{2u(s)} g_0$ to the coupled vortex equations \eqref{eq:coupled-vortex-equations} corresponding to $A(s)$, and giving $\lambda(s) := s \im(\pi_{2, s}^*A) \in C^\infty(SM_s)$, where $\pi_{2, s}^*$ is the pullback on $2$-tensors to the unit sphere bundle $SM_s$ of $g_s$, and $V_s$ is the vertical vector field on $SM_s$. Then the curvature of $g_s$ being negative (see \eqref{eq:curvature-condition-coupled-vortex}) translates to
	\begin{equation}\label{eq:e^2u(s)}
		0 > K_{g_s} = -1 + s^2 e^{-4u(s)} |A|_{g_0}^2 \iff e^{2u(s)} > s|A|_{g_0}.
	\end{equation}
	The volume form scales as $d\vol_{g_s} = e^{2u(s)}\, d\vol_{g_0}$ and so by \eqref{eq:e^2u(s)} we get
	\begin{equation}\label{eq:volume-s}
		\vol_{g_s}(M) = \int_M d\vol_{g_s} = \int_M e^{2u(s)}\, d\vol_{g_0} \geq s \int_M |A|_{g_0}\, d\vol_{g_0}.
	\end{equation}
	Observe next that the entropy production is bounded: indeed pointwise we have
	\[|V_s \lambda(s)| = s \big|\im(2i\pi_{2, s}^*A)\big| \leq 2s |\pi_{2, s}^*A| = 2s |A|_{g_s} = 2s e^{-2u(s)} |A|_{g_0} \leq 2,\]
	where in the second equality we used \eqref{eq:A-a_m-norm-relation}, and in the final estimate we used \eqref{eq:e^2u(s)}. This shows that the entropy production is bounded by
	\begin{equation}\label{eq:entropy-bound}
		|\e^+(F_s)| \leq \int_{SM_s} |V_s\lambda(s)| \, \SRB(F_s) \leq 2.
	\end{equation}
	Finally, it follows from the formula \eqref{eq:helicity-formula}, the estimates \eqref{eq:volume-s} and \eqref{eq:entropy-bound} that
	\[\mc{H}(F_s) \leq \frac{2}{2\pi \vol_g(M)} \leq \frac{1}{s \pi \int_{M} |A|_{g_0}\, d\vol_{g_0}},\]
	which completes the proof.
\end{proof}

Theorem \ref{thm:QFF} now follows directly from Lemmas \ref{lemma:res01}, \ref{lemma:winding-cycle-SRB}, \ref{lemma:generic-semisimplicity}, and \ref{lemma:helicity-explicit}.

\section{Helicity and linking}
\label{section:helicity}

In this final section we give an interpretation of helicity as an averaged quantity with respect to the SRB measures, based on the wavefront set calculus, and the works of Coles-Sharp \cite{Coles-Sharp-21} and Kotschick-Vogel \cite{Kotschick-Vogel-03}. We also give an interpretation of $\mc{H}(X)$ as an asymptotic weighted averaged sum over closed orbits. Throughout, we will assume that $[\omega^+] = [\omega^-] = 0$ and so there are distributional $1$-forms $\tau^\pm$ such that:
\begin{equation}\label{eq:tau-pm}
	\iota_X \Omega_{\mathrm{SRB}}^\pm = d\tau^\pm = \omega^\pm, \quad \tau^+ \in \Res^{1}, \quad \tau^- \in \Res_*^1.
\end{equation}
Throughout the section we will assume that $\dim \M = 3$.

\subsection{The linking form} In \cite[Section 2]{Kotschick-Vogel-03}, the linking from is constructed for any pair $(N_1, N_2)$ of null-homologous submanifolds whose sum of dimensions plus one is equal to the dimension of the ambient space, and the linking number is expressed as its double integral over $N_1 \times N_2$. We proceed in a similar fashion here. 

Fix a Riemannian metric $g$ on $\mc{M}$ and write
\begin{equation}\label{eq:fpm-def}
	\Omega_{\mathrm{SRB}}^\pm = f^\pm\, d\vol_g, \quad f^+ \in \mc{D}'_{E_u^*}(\mc{M}), \quad f^- \in \mc{D}'_{E_s^*}(\mc{M}).
\end{equation}
Denote by $\mc{H}^i$ the space of harmonic $i$-forms on $(\M, g)$ and by $(\mc{H}^i)^{\perp}$ its $L^2$ orthogonal complement; let $\mathcal{P}$ denote the orthogonal projection onto $\mc{H}^i$. Note that the Hodge Laplacian $\Delta = \Delta_i: (\mc{H}^i)^{\perp} \to (\mc{H}^i)^{\perp}$ is an isomorphism, so we may introduce $G \in \Psi^{-2}(\mc{M})$ by asking that $G = 0$ on $\mc{H}^i$ and $G = \Delta^{-1}$ on $(\mc{H}^i)^\perp$. Then, by definition:
\begin{equation}\label{eq:G-def}
	G\Delta = \Delta G = \id - \mathcal{P}, \quad \mathcal{P}G = 0.
\end{equation}
From $G \Delta = \Delta G$ and the fact that $[\Delta, d] = 0$, $[\Delta, d^*] = 0$, and $[\Delta, \star] = 0$, it follows that:
\begin{equation}\label{eq:G-commutes}
	[G, d] = 0, \quad [G, d^*] = 0, \quad [G, \star] = 0.
\end{equation}
Denote by $K \in \mc{D}'(\mc{M} \times \mc{M}; \mathrm{pr}_1^*\Omega^i \otimes \mathrm{pr}_2^*\Omega^i)$ the Schwartz kernel of $G$; here $\mathrm{pr}_1$ and $\mathrm{pr}_2$ denote projections onto the first and second factors of $\mc{M} \times \mc{M}$, respectively. Since $G \in \Psi^{-2}(\mc{M}; \Omega^i)$, it follows that $K$ is smooth outside of $\Delta(\mc{M})$ and by \cite[Theorem 1.5]{Neri-70} that $d(x, y) |K(x, y)|$ is bounded on $\mc{M} \times \mc{M} \setminus \Delta(\mc{M})$ (this uses $\dim \M = 3$), where $d(x, y)$ is the Riemannian distance function and $\Delta(\mc{M}) \subset \mc{M} \times \mc{M}$ is the diagonal.

Then for any $\alpha \in C^\infty(\mc{M}; \Omega^i)$:
\begin{equation}\label{eq:G-0}
	G \alpha(x) = \int_{y \in \mc{M}} \langle{\alpha(y), K(x, y)}\rangle_g \, d\vol_g(y) = \int_{y \in \mc{M}} \alpha(y) \wedge \star_y K(x, y),
\end{equation}
where $\langle{\bullet, \bullet}\rangle_g$ is the natural inner product on the fibres of $\Omega^i$. Specialising to $i = 1$, the \emph{linking form} $L \in \mc{D}'(\mc{M} \times \mc{M}; \mathrm{pr}_1^*\Omega^1 \otimes \mathrm{pr}_2^*\Omega^1)$ is defined as
\[L(x, y) := \star_y d_y K(x, y).\]
It satisfies the property that when integrated over two knots in $\mc{M}$ it gives the \emph{linking number} of the two knots, see \cite[Proposition 1]{Kotschick-Vogel-03} (or \cite[Section 3]{Vogel-03}), i.e.
\begin{equation}\label{eq:linking-form-def-property}
	\lk(K_1, K_2) = \int_{K_1} \int_{K_2} L,
\end{equation}
where $\lk(\bullet, \bullet)$ denotes the linking number (taking values in the rationals $\mathbb{Q}$), $K_1$ and $K_2$ are two submanifolds of dimension $1$ both of which have a trivial Poincar\'e dual over $\mathbb{R}$. Then \cite[Proposition 1]{Kotschick-Vogel-03} shows that $\lk(K_1, K_2)$ does not depend on the choice of $g$ and agrees with the definition through differential forms (also, it takes values in the rationals $\mathbb{Q}$, see the discussion in \cite[Section 2]{Kotschick-Vogel-03}; if we assume furthermore that the Poincar\'e duals of $K_1$ and $K_2$ are trivial over $\mathbb{Z}$, then $\lk(K_1, K_2) \in \mathbb{Z}$). In fact, in this paper, we define $\lk(K_1, K_2)$ as in \eqref{eq:linking-form-def-property}, where $K_1$ and $K_2$ are \emph{any} two submanifolds of dimension $1$. The difference then is that $\lk(K_1, K_2)$ \emph{does depend} on the choice of the metric $g$ (through $L$), as opposed to the case when $K_1$ and $K_2$ are assumed to have trivial Poincar\'e duals. The same definition appears in \cite[Section 7.2]{Coles-Sharp-21}.

Note that $\star_y L(x, y)$ is actually the Schwartz kernel of the operator $Gd^*$, since for any $\alpha \in C^\infty(\mc{M}; \Omega^2)$, by \eqref{eq:G-0} we have
\begin{equation}\label{eq:G-1}
	Gd^*\alpha(x) = \int_{y \in \mc{M}} d^*\alpha(y) \wedge \star_y K(x, y) = \int_{y \in \mc{M}} \alpha(y) \wedge L(x, y).
\end{equation}
It follows that $d(x, y)^{2}|L(x, y)|$ is bounded on $\mc{M} \times \mc{M} \setminus \Delta(\mc{M})$. Set 
\begin{equation}\label{eq:lambda-def}
	\Lambda(x, y) := L(x, y)(X(x), X(y)) \in \mc{D}'(\mc{M} \times \mc{M}).
\end{equation}
This distribution satisfies the following important properties:

\begin{lemma}\label{lemma:Lambda}
	The distribution $\Lambda(x, y)$ is the Schwartz kernel of the operator $P := \iota_X G d^* \iota_X \star$. Moreover, $P \in \Psi^{-2}(\mc{M})$ and thus there is a constant $C_{\Lambda} > 0$ such that
	\[d(x, y)|\Lambda(x, y)| \leq C_{\Lambda}, \quad (x, y) \in \mc{M} \times \mc{M} \setminus \Delta(\mc{M}).\]
	Finally, $\Lambda(x, y) = \Lambda(y, x)$ and $f^+(x) \Lambda(x, y) f^-(y) \in \mc{D}'(\mc{M} \times \mc{M})$ is well-defined as a distribution.
\end{lemma}

We remark that the bound on $\Lambda(x, y)$ was shown \cite[Lemma 9.2]{Coles-Sharp-21} by a computation in local coordinates, whereas here we employ a global approach.

\begin{proof}
	For any $f \in C^\infty(\mc{M})$ we compute:
	\begin{align*}
		Pf(x) &= \int_{y \in \mc{M}} f(y) \,\iota_{X(y)} d\vol_g(y) \wedge \iota_{X(x)} L(x, y)\\
		&= \int_{y \in \mc{M}} f(y)\, d\vol_g(y) \cdot \underbrace{\iota_{X(y)} \iota_{X(x)} L(x, y)}_{= \Lambda(x, y)}.
	\end{align*}
	Here we used \eqref{eq:G-1} in the first equality and the anti-commuting property of the contraction in the second; this completes the proof of the first claim.
	
	For the second claim, simply write:
	\[P = [\iota_X, G] d^* \iota_X \star + G \underbrace{(\iota_X d^* + d^*\iota_X)}_{=: Q} \iota_X \star \in \Psi^{-2}(\mc{M}),\]
	since the commutators $[\iota_X, G]$ and $Q$ are pseudodifferential operators of degrees $-3$ and $0$, respectively. Indeed, the principal symbol of $G$, given by $\sigma(G)(x, \xi) = |\xi|^{-2}_g(x, \xi) \times \id_{\Omega^1}$ is diagonal, and $\sigma(\iota_X)(x, \xi) = \iota_{X(x)}$ and $\sigma(d^*)(x, \xi) = \iota_{\xi^\sharp}$, where $\xi^\sharp$ is obtained by applying the musical isomorphism to $\xi$. Since $\iota_{X(x)} \iota_{\xi^\sharp} = -\iota_{\xi^\sharp} \iota_{X(x)}$, this shows that $Q$ has degree $0$, which completes the proof.
	
	Next, the symmetry $\Lambda(x, y) = \Lambda(y, x)$ is equivalent to $P$ being formally self-adjoint, which follows from integration by parts, by using that $G$ is self-adjoint, and applying \eqref{eq:G-commutes}.
	
	For the final claim it suffices to observe that 
	\begin{equation}\label{eq:wf-normal}
		\WF(\Lambda) \subset \{(x, x, \xi, -\xi) \mid x \in \mc{M},\, \xi \in T_x^*\mc{M}\}
	\end{equation}
	since $\Lambda$ is a kernel of a pseudodifferential operator, and that $f^+$ and $f^-$ have disjoint wavefront sets, hence the wavefront set calculus applies and the product $f^+(x)\Lambda(x, y) f^-(y)$ is well-defined.
\end{proof}

We proceed with the proof of Theorem \ref{thm:helicity-formula}, which is the main result of this section.

\begin{proof}[Proof of Theorem \ref{thm:helicity-formula}]
	
	Observe that for any $\alpha \in \mc{D}'(\mc{M}; \Omega^1)$ it holds that
\begin{equation}\label{eq:G-formula}
	G d^*d \alpha = G\Delta \alpha - G d d^* \alpha = \alpha - \mathcal{P}\alpha - dh,   
\end{equation}
where $h = Gd^*\alpha$, and we used \eqref{eq:G-def} and \eqref{eq:G-commutes} in the second equality. Now we may compute:
\begin{align}\label{eq:G-helicity}
\begin{split}
	\mc{H}(X) &= \int_{\mc{M}} \tau^- \wedge d\tau^+\\
	 			   &= \int_{\mc{M}} Gd^*d\tau^- \wedge d\tau^+\\
	 			   &= \int_{\mc{M}} \iota_X Gd^*\iota_X \star f^- \cdot f^+\, d\vol_g\\
	 			   &= \int_{\mc{M}} Pf^- \cdot f^+\, d\vol_g\\
	 			   &= \int_{(x, y) \in \mc{M} \times \mc{M}} f^+(x) \Lambda(x, y) f^-(y)\, d\vol_g(y) \times d\vol_g(x).
\end{split}
\end{align}
Here in the second line we used \eqref{eq:G-formula} for $\alpha = \tau^-$ and integration by parts, in the third line we used \eqref{eq:tau-pm} and \eqref{eq:fpm-def}, in the fourth we used the definition of $P$, and in the fifth we used that $f^+\Lambda f^-$ is a well-defined distribution by Lemma \ref{lemma:Lambda}. 

Finally, to interpret the integral in the fifth line as a limit of classical integrals, consider the regularisations $\Lambda_\varepsilon := E_\varepsilon \Lambda \in C^\infty(\mc{M} \times \mc{M})$ (where the mollifiers $E_\varepsilon$ are introduced in \S \ref{ssection:regularisation} below). Then we have the following chain of limits:
\begin{align*}
	\mc{H}(X) &= \lim_{\varepsilon \to 0} \int_{(x, y) \in \mc{M} \times \mc{M}} \Lambda_\varepsilon(x, y)\,  \SRB(x) \times \SRBs(y)\\
	&= \lim_{\varepsilon \to 0} \int_{(x, y) \in \mc{M} \times \mc{M} \setminus \Delta(\mc{M})} \Lambda_\varepsilon(x, y)\,  \SRB(x) \times \SRBs(y)\\
	&= \int_{(x, y) \in \mc{M} \times \mc{M} \setminus \Delta(\mc{M})} \Lambda(x, y)\,  \SRB(x) \times \SRBs(y).
\end{align*}
Here in the first equality we used the second part of Lemma \ref{lemma:mollification} to obtain $\Lambda_\varepsilon \to \Lambda$ in $\mc{D}'_\Gamma(\mc{M} \times \mc{M})$ where $\Gamma$ is given by the right hand side of \eqref{eq:wf-normal}, and we used the sequential continuity of multiplication under the wavefront set condition (see \cite[Chapter 8]{Hormander-90}). Next, in the second line we used \cite[Lemma 9.7]{Coles-Sharp-21}, which guarantees that the $\SRB \times \SRBs$ measure of $\Delta(\mc{M}) \subset \mc{M} \times \mc{M}$ is zero. In the final line we used the first part of Lemma \ref{lemma:mollification} for $S = \Delta(\mc{M})$ and $t = 1$, which says that $\Lambda_\varepsilon$ is dominated by the function $C d(x, y)^{-1} \sim Cd\big((x, y), \Delta(\mc{M})\big)^{-1}$ for some uniform $C > 0$ (recall by Lemma \ref{lemma:Lambda} that $\Lambda$ has the singularity $d(x, y)^{-1}$ at the diagonal and is smooth outside of it), which is integrable with respect to $\SRB \times \SRBs$ by \cite[Lemma 9.3(ii)]{Coles-Sharp-21}, as well as \cite[Lemma 9.6]{Coles-Sharp-21} (where the latter lemma is applied to H\"older continuous functions $\psi = -r^u$ and $\psi = r^s$ which give the SRB measures as equilibrium measures -- recall that $r^{u/s}$ were introduced in \eqref{eq:stable-unstable-flow}; see also \S \ref{ssection:classical-helicity}). This also shows the first part of the lemma and moreover, applying the Dominated Convergence Theorem completes the proof.
\end{proof}

\subsection{Regularisation of the kernel}\label{ssection:regularisation} We now discuss the auxiliary mollification result used in the proof of Theorem \ref{thm:helicity-formula}. Given a metric $g_N$ on a manifold $N$ of dimension $n = \dim N$, introduce the family of mollifiers defined for $u \in \mc{D}'(N)$ and $\varepsilon > 0$:
\begin{equation}
	E_\varepsilon u (x) := \frac{1}{F_\varepsilon(x)}\int_N \chi \Big(\frac{d(x, y)}{\varepsilon}\Big) u(y)\, d\vol_g(y), \quad x \in N.
\end{equation}
Here $\chi \in C^\infty(\mathbb{R}_{\geq 0}; [0, 1])$ is a non-increasing cut-off function supported in $[0, 1]$ with values in $[0, 1]$ such that $\chi = 1$ close to zero; also $F_\varepsilon \in C^\infty(N)$ is chosen such that $E_\varepsilon \mathbf{1} \equiv \mathbf{1}$ and for some $C > 1$
\begin{equation}\label{eq:mollifier0}
	\forall \varepsilon > 0, \forall x \in N, \quad C^{-1} \varepsilon^{n} \leq F_\varepsilon(x) \leq C\varepsilon^n.
\end{equation}
By \cite[eq. (2.18)]{Dyatlov-Zworski-16} we know that:
\begin{equation}\label{eq:mollifier}
	E_\varepsilon \in \Psi^{- \infty}(N), \quad E_\varepsilon \to_{\varepsilon \to 0} \id \,\,\,\,\mathrm{in}\,\,\,\, \Psi^{0+}(N),
\end{equation}
where the latter limit is understood in the Fr\'echet topologies of $\Psi^k(N)$ for every $k > 0$. We prove the following approximation result:

\begin{lemma}\label{lemma:mollification}
	Let $S \subset N$ be a smooth submanifold of dimension $s$. Denote by $d(\bullet, S)$ the distance function to $S$. Then for any $t \in (0, n -s)$, there exists $C_t > 0$ such that:
	\[\forall \varepsilon > 0, \forall x \not \in S, \quad \big[E_\varepsilon d(\bullet, S)^{-t}\big](x) \leq C_t d(x, S)^{-t}.\]
	Moreover, if $u \in \mc{D}'_{\Gamma}(N)$ for some closed conic set $\Gamma$, we have:
	\[E_\varepsilon u \to_{\varepsilon \to 0} u \,\,\,\,\mathrm{in}\,\,\,\, \mc{D}'_{\Gamma}(N).\]
\end{lemma}
	The sequential topology in the space $\mc{D}'_{\Gamma}(N)$ is introduced in \cite[Chapter 8]{Hormander-90}.
\begin{proof}
	We first show the second claim. For that, consider an arbitrary $\varphi \in C^\infty(N)$. By \eqref{eq:mollifier}, we have ${}^tE_\varepsilon \varphi \to \varphi$ in $C^\infty(N)$ as $\varepsilon \to 0$, where the superscript ${}^t$ denotes the transposed operator. It follows that 
	\[\langle{E_\varepsilon u, \varphi}\rangle = \langle{u, {}^tE_\varepsilon \varphi}\rangle \to_{\varepsilon \to 0} \langle{u, \varphi}\rangle,\]
	which means that $E_\varepsilon u \to_{\varepsilon \to 0} u$ in $\mc{D}'(N)$. To show convergence in $\mc{D}'_{\Gamma}(N)$, it suffices to show that for an arbitrary $A \in \Psi^0(N)$ with $\WF(A) \cap \Gamma = \emptyset$, $AE_\varepsilon u$ is uniformly (in $\varepsilon > 0$) bounded in $C^\infty(N)$ (see \cite[Definition 8.2.2]{Hormander-90}). Let $B \in \Psi^0(N)$ be such that $\WF(B) \cap \WF(A) = \emptyset$ and $\WF(\id - B)$ is contained in the complement of an open conic neighbourhood of $\Gamma$. Then
	\[AE_\varepsilon u = AE_\varepsilon B u + AE_\varepsilon (\id - B)u,\]
	where by construction and \eqref{eq:mollifier} we have $A E_\varepsilon B \in \Psi^{-\infty}(N)$ with uniformly bounded seminorms, $(\id - B)u \in C^\infty(N)$ (as $\WF(u) \subset \Gamma$) and since again by \eqref{eq:mollifier} we have $AE_\varepsilon \in \Psi^{0+}(N)$ with uniformly bounded seminorms. The claim immediately follows.
	
	For the first claim, from the definition of $E_\varepsilon$ and by the triangle inequality, it follows that
	\begin{equation*}
		\big[E_\varepsilon d(\bullet, S)^{-t}\big](x) \leq (d(x, S) - \varepsilon)^{-t}, \quad d(x, S) > \varepsilon,
	\end{equation*}
	and so we immediately get, setting $u_\varepsilon := E_\varepsilon d(\bullet, S)^{-t}$
	\begin{equation}\label{eq:interpolate-0}
		u_\varepsilon(x) \leq 2^t d(x, S)^{-t}, \quad d(x, S) \geq 2\varepsilon.
	\end{equation}
	
	Next, consider the injectivity radius $\iota_N > 0$ of the metric $g_N$. Identify an $\varepsilon_0$-neighbourhood of $S$ with $T := N_{\leq \varepsilon_0} S = \{(x, \xi) \in NS \mid x \in S,\, |\xi| \leq \varepsilon_0\}$ via the parametrisation by normal geodesics to $S$, where $NS$ denotes the normal bundle to $S$. By taking $\varepsilon_0 \in (0, \iota_N)$ small enough, we may assume there exists $C_0 > 1$ such that for any $(x, \xi), (y, \eta) \in T$ with distance at most $\varepsilon_0$ apart, we have
	\begin{equation}\label{eq:metric-taxi-metric}
		C_0^{-1} d\big((x, \xi), (y, \eta)\big) \leq d(x, y) + |\xi - P_{y \to x} \eta| \leq C_0 d\big((x, \xi), (y, \eta)\big),
	\end{equation}
	where $P_{y \to x}: T_yN \to T_xN$ is the parallel transport with respect to the Levi-Civita connection along the unique short geodesic from $y$ to $x$. For any $(x, \xi) \in T\setminus S$ with $d\big((x, \xi), S\big) \leq 2\varepsilon$ (and $\varepsilon > 0$ small enough) and for any $r \in (0, 1)$, we make the following computation:
	\begin{align*}
		u_\varepsilon(x, r\xi) &= \frac{1}{F_\varepsilon(x, r\xi)} \int_{y \in T_xS,\, |y| \leq C_1 \varepsilon} \int_{\eta \in N_x S,\, |\eta| \leq C_1 \varepsilon} \chi \Big(\frac{d\big((x, r\xi), (y, \eta)\big)}{\varepsilon}\Big) |\eta|^{-t} J_x(y, \eta)\, dy\, d\eta\\
		&\leq C\varepsilon^{-n} \int_{y \in T_xS,\, |y| \leq C_1 \varepsilon} \int_{\eta \in N_x S,\, |\eta| \leq C_1 \varepsilon} \chi\Big(\frac{C_0^{-1} \big(\frac{1}{r}|y| + |\xi - \frac{1}{r} \eta| \big)}{\varepsilon/r}\Big) |\eta|^{-t}J_x(y, \eta)\, dy\, d\eta\\
		&= C\varepsilon^{-n}r^{n - t} \int_{y' \in T_xS,\, |y'| \leq C_1 \frac{\varepsilon}{r}} \int_{\eta' \in N_x S,\, |\eta'| \leq C_1 \frac{\varepsilon}{r}} \chi\Big(\frac{C_0^{-1} \big(|y'| + |\xi - \eta'| \big)}{\varepsilon/r}\Big) |\eta'|^{-t}J_x(ry', r\eta')\, dy'\, d\eta'\\
		&\leq C\varepsilon^{-n}r^{n - t} \|J_x\|_{\infty} (2C_1)^{s} \Big(\frac{\varepsilon}{r}\Big)^s \int_{\eta' \in N_x S,\, |\eta'| \leq C_1 \frac{\varepsilon}{r}} |\eta'|^{-t}\, d\eta'\\
		&\leq \vol(\mathbb{S}^{n - s - 1}) C\varepsilon^{-n}r^{n - t} \|J_x\|_{\infty} (2C_1)^{s} \Big(\frac{\varepsilon}{r}\Big)^s \int_0^{C_1\frac{\varepsilon}{r}} \rho^{-t + n - s - 1}\, d\rho\\
		&\leq \frac{\vol(\mathbb{S}^{n - s - 1})}{-t + n - s} C\varepsilon^{-n}r^{n - t} \|J_x\|_{\infty} 2^s (C_1)^{-t + n} \Big(\frac{\varepsilon}{r}\Big)^{-t + n}\\
		&\leq C'(n, s, t) \varepsilon^{-t} \leq C' 2^t d\big((x, \xi), S\big)^{-t}.
	\end{align*}
	where in the first line we identified $y$ with points in a geodesic ball of radius $C_1\varepsilon$ centred at $x \in S$, for a suitable uniform constant $C_1 > 0$, and $\eta$ with vectors in the fibre $N_xS$ via parallel transport $P_{y \to x}$; $J_x(y, \eta)$ is the Jacobian of the volume form $d\vol_{g_N}$ in the coordinates $(y, \eta)$ (uniformly bounded from above). In the second line we used \eqref{eq:mollifier0}, \eqref{eq:metric-taxi-metric}, and the monotonicity of $\chi$, while in the third one we changed the coordinates by $\eta' = \frac{\eta}{r}$ and $y' = \frac{y}{r}$. In the fourth line we estimated $\chi$ by $1$ from above, integrated in $y'$ and bounded from above the volume of the unit ball in $\mathbb{R}^s$ by $2^s$; in the fifth we used the polar coordinate system in $N_xS$; throughout $\vol(\mathbb{S}^k)$ denotes the volume of the unit sphere $\mathbb{S}^k \subset \mathbb{R}^{k + 1}$. In the sixth line, we used the assumption that $d(\bullet, S)^{-t}$ is integrable (i.e. $t < n - s$). In the seventh line, we introduced the constant $C' = C'(n, s, t) > 0$. In the final estimate we used the assumption that $d\big((x, \xi), S\big) \leq 2\varepsilon$. (We note that the role of $r \in (0, 1)$ is to keep track of the natural scaling in this estimate.)
	
	Since $(x, \xi) \in T \setminus S$ with $d\big((x, \xi), S\big) \leq 2\varepsilon$ (with $\varepsilon > 0$ small enough), as well as $r \in (0, 1)$ were arbitrary, combining the preceding estimate with \eqref{eq:interpolate-0} completes the proof.
\end{proof}

\subsection{Relation with linking of closed orbits}\label{ssection:classical-helicity} Here we indicate how the results of \cite{Coles-Sharp-21} can be directly used in conjunction with Lemma \ref{thm:helicity-formula} to give another formula for the helicity in terms of linkings of closed orbits, generalising \cite[Theorem 1.1]{Coles-Sharp-21}. We first very briefly introduce some notation regarding the thermodynamic formalism that will be used only in this section -- we refer the reader to \cite[Section 3]{Coles-Sharp-21}, \cite{Hasselblatt-Katok-95} and \cite[Chapter 10]{Merry-Paternain-11} for more details. 

Denote by $\mc{PM}(X)$ the set of flow-invariant Borel probability measures on $\mc{M}$, and given $\nu \in \mc{PM}(X)$ denote by $h(\nu)$ the \emph{measure-theoretic entropy} of the time $1$ map, $\varphi_1$. Given a H\"older continuous function $\psi$, denote by $P(\psi)$ the \emph{pressure} of $\psi$, i.e.
\[P(\psi) = \sup \Big\{h(\nu) + \int_{\mc{M}} \psi\, d\nu \mid \nu \in \mc{PM}(X)\Big\}.\]
There exists a unique $\mu_\psi \in \mc{PM}(X)$ that attains the supremum above, called the \emph{equilibrium state}. It is well-known that the SRB measures are equilibrium states, more precisely we have (see \cite[Chapter 10.3]{Merry-Paternain-11})
\[\SRB = \mu_{-r^u}, \quad \SRBs = \mu_{r^s},\]
where we recall $r^{u/s}$ were defined in \eqref{eq:stable-unstable-flow}.

For $T > 0$, denote by $\mc{P}_T$ the set of closed orbits of $\varphi_t$ with period in $(T - 1, T]$; denote by $\mc{P}_T(0) \subset \mc{P}_T$ the subset of closed orbits which are trivial in homology $H_1(\mc{M}; \mathbb{R})$. Given $\gamma \in \mc{P}_T$, set $\mu_\gamma$ to be the probability Dirac measure on $\gamma$, that is, if $T_\gamma$ is the period of $\gamma$, then $\mu_\gamma(f) = \frac{1}{T_\gamma} \int_\gamma f$ for $f \in C^0(\mc{M})$. Given a H\"older function $\psi$, introduce the following weighted orbital probability measures \cite[Section 9]{Coles-Sharp-21}:
\begin{equation}\label{eq:weighted-orbital-measure}
	\mu_{\psi, T} := \frac{\sum_{\gamma \in \mc{P}_T} e^{\int_\gamma \psi} \mu_\gamma}{\sum_{\gamma \in \mc{P}_T} e^{\int_\gamma \psi}}, \quad \mu_{\psi, T}^0 := \frac{\sum_{\gamma \in \mc{P}_T(0)} e^{\int_\gamma \psi} \mu_\gamma}{\sum_{\gamma \in \mc{P}_T(0)} e^{\int_\gamma \psi}}.
\end{equation}
These measures are used to define the following weighted sums:
	\begin{align}\label{eq:classical-helicity}
	\begin{split}
		\mc{L}(T) &:= \int_{\M \times \M} \Lambda\,  d\mu^0_{-r^u, T} \times d\mu_{r^s, T + 1}\\
		&= \frac{\sum_{\gamma \in \mc{P}_T(0),\, \gamma' \in \mc{P}_{T + 1}} \frac{\lk(\gamma, \gamma')}{T_\gamma T_{\gamma'}} e^{-\int_\gamma r^u + \int_{\gamma '} r^s}}{\sum_{\gamma \in \mc{P}_T(0),\, \gamma' \in \mc{P}_{T + 1}} e^{-\int_\gamma r^u + \int_{\gamma '} r^s}},
	\end{split}
	\end{align}
	where we recall $\Lambda$ was introduced in \eqref{eq:lambda-def} (it is smooth outside of the diagonal), and in the second line we used the definition \eqref{eq:weighted-orbital-measure} of weighted orbital measures and the defining property of the linking form $L$ from \eqref{eq:linking-form-def-property}.
\begin{prop}\label{prop:classical-helicity}
	The following formula holds:
	\begin{align*}
		\mc{H}(X) &= \lim_{T \to \infty} \mc{L}(T).
	\end{align*}
\end{prop}
\begin{proof}
	By \cite[Theorem 4.1]{Coles-Sharp-21}, we get that $\mu_{r^s, T} \to \SRBs$ in the weak limit sense and moreover, by \cite[Theorem 6.7]{Coles-Sharp-21} that $\mu^0_{-r^u, T} \to \mu_{-r^u + f_{[\tau]}}$, where $[\tau] = [\tau](-r^u) \in H^1(\mc{M}; \mathbb{R})$ is a cohomology class depending on $-r^u$ and $f_{[\tau]}(x) = \tau(X(x))$. More precisely, $[\tau]$ is defined as the minimizer of the function $\beta: H^1(\mc{M}; \mathbb{R}) \to \mathbb{R}$ defined by $\beta([\varpi]) := P(-r^u + f_{[\varpi]})$ (as the notation suggests, $\mu_{-r^u + f_{[\tau]}}$ and $P(-r^u + f_{[\varpi]})$ do not depend on the choice of the primitives $\tau$ and $\varpi$):
	\[\beta([\tau]) = \inf_{[\varpi] \in H^1(\mc{M}; \mathbb{R})} P\big(-r^u + f_{[\varpi]}\big).\]
	That this infimum is uniquely attained follows from \cite[Proposition 5.1]{Coles-Sharp-21} by noting that $X$ is \emph{homologically full}, i.e. that every integral class in $H_1(\mc{M}; \mathbb{Z})$ is represented by a closed orbit, which in turn follows from the assumption that the winding cycle $[\omega^+]$ of $\SRB$ is zero and \cite[Proposition 3.1]{Coles-Sharp-21}. From \cite[Proposition 5.1]{Coles-Sharp-21} it also follows that the function $[\varpi] \mapsto \beta([\varpi])$ is strictly convex and by \cite{Lalley_87,Sharp_92} that it is real-analytic, and that its derivative at zero in the direction of $[\varpi] \in H^1(\mc{M}; \mathbb{R})$ is
	\[D \beta (0)([\varpi]) = \int_{\mc{M}} f_{[\varpi]}\, \SRB = \int_{\mc{M}} \varpi(X)\, \SRB = 0,\]
	as $[\omega^+] = 0$. By uniqueness of the infimum it follows that $[\tau] = 0$ and so $\mu^0_{-r^u, T} \to \SRB$ weakly.
	
	Next, for $R > 0$ set $B_R := \{(x, y) \in \mc{M} \times \mc{M} \mid d(x, y) < R\} \supset \Delta(\M)$. By \cite[Lemma 9.9]{Coles-Sharp-21} and as a consequence of Fubini's theorem as in \cite[proof of Lemma 9.9]{Coles-Sharp-21}, there exist $Q, \alpha > 0$ such that for all $R, T > 0$ (note here that the supports of $\mu^0_{-r^u, T}$ and $\mu_{r^s, T + 1}$ do not intersect so $\Lambda$ is smooth in the domain of the following integral)
	\begin{equation}\label{eq:weighted-orbital-measure-regularity}
		\int_{B_R} |\Lambda|\,  d\mu^0_{-r^u, T} \times d\mu_{r^s, T + 1} \leq QR^\alpha.
	\end{equation}
	Take an arbitrary $\delta > 0$. By \eqref{eq:weighted-orbital-measure-regularity} and Theorem \ref{thm:helicity-formula} we may take $R > 0$ small enough such that for all $T > 0$
	\[\int_{B_R} |\Lambda|\,  d\mu^0_{-r^u, T} \times d\mu_{r^s, T + 1} < \delta \quad \mathrm{and} \quad \Big|\int_{B_R} \Lambda\,  \SRB \times \SRBs \Big| < \delta,\]
	respectively. Next, by the preceding paragraph for $T > 0$ large enough we have
	\[\Big|\int_{\mc{M} \times \mc{M} \setminus B_R} \Lambda\,  d\mu^0_{-r^u, T} \times d\mu_{r^s, T + 1} - \int_{\mc{M} \times \mc{M} \setminus B_R} \Lambda\,  \SRB \times \SRBs \Big| < \delta.\]
	Combining the two previous inequalities and by the triangle inequality, for $R > 0$ small enough and $T > 0$ large enough we get:
	\[\Big|\int_{\mc{M} \times \mc{M}} \Lambda\,  d\mu^0_{-r^u, T} \times d\mu_{r^s, T + 1} - \int_{\mc{M} \times \mc{M}} \Lambda\,  \SRB \times \SRBs \Big| < 3\delta.\]
	Since $\delta > 0$ was arbitrary, using Theorem \ref{thm:helicity-formula} concludes the proof.
\end{proof}

\begin{Remark}\rm
	As we have already mentioned there are no known examples of Anosov flows with zero helicity. The formulas in \cite[Theorem 1.1]{Coles-Sharp-21} and Proposition \ref{prop:classical-helicity} give another point of view to this problem, however since the linkings of closed orbits can be both positive and negative in theory there could be cancellations in \eqref{eq:classical-helicity} that in the limit give zero.
\end{Remark}

\begin{Remark}\rm
	In the recent article Dang-Rivi\`ere \cite[Theorem 1.2]{Dang-Riviere-20} show that the linking of two distinct closed geodesics (which are homologically trivial) in the unit sphere bundle of a negatively curved surface can be expressed as the value at zero of a certain Poincar\'e series involving lengths of orthogeodesics.
\end{Remark}

\begin{Remark}\rm  In \cite{marty2023}, Marty proves the existence of a unique continuous symmetric bilinear form on the space of null homologous Borel invariant probability measures of a transitive Anosov flow such that it extends the linking number between two null homologous closed orbits. It is an interesting question to decide if $\mathcal H(X)$ agrees with Marty's linking between the SRB measures (when they are both null homologous).
\end{Remark}

\subsection{First variation of the helicity}
Here we compute the first variation formula for helicity and derive some consequences about the set of Anosov flows with zero helicity. We will use the notation and the perturbation theory derived in Section \ref{sec:perturbation}. We will assume \eqref{eq:tau-pm} for the vector field $X_0$. As follows from the second equality of \eqref{eq:G-helicity}, we have for $X \in \mc{W}$ near $X_0$
\begin{equation}\label{eq:G-helicity-2}
	\mc{H}(X) = \int_{\M} Gd^*\iota_{X}\SRBs \wedge \iota_{X} \SRB.
\end{equation} 
For technical reasons, we will now assume that $X_0$ preserves a smooth volume $\Omega$. Recall that the spaces $\mc{W}^\pm$ and $\mc{W} = \mc{W}^+ \cap \mc{W}^-$ defined in \eqref{eq:def-W-W-W+}, are locally $C^1$ Banach submanifolds of codimensions $b_1(\mc{M})$ and $2b_1(\M)$, respectively, by Lemma \ref{lemma:banach-manifold}. We prove: 

\begin{prop}\label{prop:helicity-zero}
	Assume $X_0$ preserves a smooth volume $\Omega$ such that $[\iota_{X_0} \Omega] = 0$ and $\mc{H}(X_0) = 0$. Then, the set $\{X \in \mc{W} \mid \mc{H}(X) = 0\} \subset \mc{W}$ is locally a $C^1$ Banach submanifold of codimension $1$.
\end{prop}
\begin{proof}
	We first show that the map $\mc{W} \ni X \mapsto \mc{H}(X) \in \mathbb{R}$ is $C^1$-regular in suitable topologies (for $X \in C^N(\M; T\M)$ for $N$ large enough). Coming back to Lemma \ref{lemma:perturbation-theory}, it follows from \cite[Section 2]{Bonthonneau-19} that the function $m(x, \xi)$ and the analogous function $m'(x, \xi)$ for the flow $-X_0$ defined on $T^*\M$ may be chosen with values in $[-\frac{1}{2}, 1]$, such that: $m$ is equal to $-\frac{1}{2}$ in a conic neighbourhood of $E_u^*$ and equal to $1$ outside of slightly larger conic neighbourhood of $E_u^*$, and $m'$ is chosen to satisfy the analogous property with respect to $E_s^*$. If $G'(x, \xi) \sim m'(x, \xi) \log(1 + |\xi|)$ is a logarithmically growing symbol on $T^*\M$, for $r$ large enough, it follows that the distributional wedge product is well-defined as a map $\mc{H}_{rG, -3}(\M; \Omega^1) \times \mc{H}_{rG', -3}(\M; \Omega^2) \to \mc{D}'(\M; \Omega^3)$. More precisely, for this wedge product to make sense, by using the definition of anisotropic spaces, as well as integration by parts it suffices to have (for simplicity, this is viewed on functions)
	\[
		(1 + \Delta_g)^{\frac{3}{2}} e^{-r \Op(G)} e^{-r\Op(G')} (1 + \Delta_g)^{\frac{3}{2}} \in \Psi^0(\M).
	\]
	In turn, by the composition rule for pseudodifferential operators it suffices to have $r(m + m') - 6 \geq 0$. By construction we have $m + m' \geq \tfrac{1}{2}$, so this is satisfied for large enough $r$. Thus, by noting that the pseudodifferential operator $Gd^*$ of degree $-1$ maps $\mc{H}_{rG, -3}(\M; \Omega^2)$ to $\mc{H}_{rG, -3}(\M; \Omega^1)$, the required regularity follows from the expression \eqref{eq:G-helicity-2} by applying Lemma \ref{lemma:proj-regularity}.
	
	Next, we compute the first derivative $D_{X_0} \mc{H}$. By the final formula of Lemma \ref{lemma:proj-regularity} we get for $Y \in C^N(\M; T\M)$ that
\begin{equation}\label{eq:first-variation-helicity}
	D_{X_0} \mc{H} (Y) = \int_{\mc{M}} Gd^* \Pi_2^+ \iota_Y \SRB \wedge \iota_{X_0} \SRBs + \int_{\mc{M}} Gd^* \iota_{X_0} \SRB \wedge \Pi_2^- \iota_Y \SRBs.
\end{equation}
Here we used \eqref{eq:G-formula} for $\alpha = \iota_{X_0} R_2^{\pm, H} \iota_Y \Omega_{\mathrm{SRB}}^{\pm}$ and the fact that $\iota_{X_0} \alpha = 0$ to get rid of the second term in the last formula in Lemma \ref{lemma:proj-regularity}.

	Since $X_0$ has zero helicity, it is not a contact flow. Thus by Lemma \ref{lemma:auxiliary-transversal}, there exists $Y \in C^\infty(\M; T\M) \cap T_{X_0} \mc{W}$ such that $\Pi_2^+ \iota_Y \Omega = d\alpha^+$ and $\Pi_2^- \iota_Y \Omega = 0$ where $\Pi_1^+ \beta =: \alpha^+ \in \Res^{1, \infty}$ is defined for some $\beta \in C^\infty(\M; \Omega^1)$ with $\iota_{X_0} \beta = 1$ (the analogous construction for $\Res_*^{1, \infty}$ was carried out in Proposition \ref{prop:semisimplicity-fails}) and so $\alpha^+$ satisfies $\iota_{X_0} \alpha^+ = 1$. Then the first variation formula \eqref{eq:first-variation-helicity} above gives, using also \eqref{eq:G-formula}:
	\[D_{X_0} \mc{H}(Y) = \int_{\mc{M}} \alpha^+ \wedge d\tau^- + 0 = 1,\]
	since $\alpha^+ \wedge d\tau^- = \SRBs$, which shows that the derivative $D_{X_0} \mc{H}$ is surjective onto $\mathbb{R}$ and the main claim follows from the Implicit Function Theorem.
\end{proof}

\begin{Remark}\rm \label{rem:vol-pres-helicity}
	If we restrict to $\mc{V}_\Omega = \{X \in C^N(\M; T\M) \mid \Lie_X \Omega = 0\}$, i.e. flows preserving the fixed smooth volume form $\Omega$, then it is straightforward (even much simpler than the dissipative case, since there is no need for anisotropic spaces) to see that: 1) the manifolds $\mc{W} := \mc{W}^+ \cap \mc{V}_\Omega = \mc{W}^- \cap \mc{V}_\Omega \subset \mc{V}_\Omega$ are of codimension $b_1(\M)$ (follows from \eqref{eq:F-pm-def} and \eqref{eq:F-pm-derivative}, and the Implicit Function Theorem); 2) the first derivative of the helicity at $X_0$ is non-zero, so $\{X \in \mc{W} \mid \mc{H}(X) = 0\} \subset \mc{W}$ is locally a $C^1$-regular Banach manifold (follows from \eqref{eq:G-helicity-2} by noting that $\SRB = \SRBs = \Omega$ is fixed).
\end{Remark}

\appendix

\section{Revision of elementary facts about $\eta_\pm$}\label{app:A} 

Here we recall some properties of the operators $\eta_\pm$ introduced in \S \ref{ssec:geometry-surfaces}; we will follow the notation introduced in that section. Let $(x, y)$ denote local isothermal coordinates on $U \subset M$, so $g|_U = e^{2\psi} (dx^2 + dy^2)$. Then on $SM|_{U}$ we have another set of coordinates:
\[U \times \mathbb{S}^1 \ni (x, y, \theta) \mapsto (x, y, e^{-\psi} \cos \theta, e^{-\psi} \sin \theta) \in SM|_U.\]
In these coordinates, \cite[p. 36]{Merry-Paternain-11} show that
\begin{align}\label{eq:isothermal-X-H}
\begin{split}
	X(x, y, \theta) &= e^{-\psi} \Big(\cos \theta \cdot \partial_x + \sin \theta \cdot \partial_y + \Big(-\partial_x \psi \cdot \sin \theta + \partial_y \psi \cdot \cos \theta\Big)\partial_\theta\Big),\\
	H(x, y, \theta) &= e^{-\psi} \Big(-\sin \theta \cdot \partial_x + \cos \theta \cdot \partial_y - \Big(\partial_x \psi \cdot \cos \theta + \partial_y \psi \cdot \sin \theta\Big)\partial_\theta\Big),
\end{split}
\end{align}
so that we compute $\eta_\pm = \frac{X \mp iH}{2}$, $\frac{\partial}{\partial \bar{z}} = \frac{1}{2}\Big(\partial_x + i\partial_y\Big)$, $\frac{\partial}{\partial z} = \frac{1}{2}\Big(\partial_x - i\partial_y\Big)$:
\begin{align}\label{eq:eta-}
	\eta_-(x, y, \theta) &= e^{-\psi} e^{-i\theta} \Big(\frac{\partial}{\partial \bar{z}} - i \frac{\partial \psi}{\partial \bar{z}} \cdot \partial_\theta\Big),\\
	\eta_+(x, y, \theta) &= e^{-\psi} e^{i\theta} \Big(\frac{\partial}{\partial z} + i \frac{\partial \psi}{\partial z} \cdot \partial_\theta\Big).
\end{align}
Now let $A$ be a section of $\mathcal{K}^{\otimes m}$, where $\mathcal{K}$ are the holomorphic $1$-forms (spanned locally by $dz$). Then the pullback $\pi^*A$ is a section of $\pi^*\mc{K}^{\otimes m} \subset \otimes^m T^*(SM)$ (here the pullback is in the sense of tensor pullback and we recall $\pi: SM \to M$ is the projection), and in fact we claim:
\begin{lemma} 
	The bundle $\pi^*\mc{K} \subset T^*(SM)$ is spanned by $\alpha + i\beta$.
\end{lemma}
\begin{proof}
	To see this, denote by $J$ the complex structure on the surface $(M, g)$, and note that by definition for any $u \in T_{(x, v)}SM$ we have
\[\alpha(x, v)(u) = g_x(v, d\pi(x, v)u) = \pi^*(g_x(v, \bullet)) (u), \quad \beta(x, v)(u) = g_x(Jv, d\pi(x, v)u) = \pi^*(g_x(Jv, \bullet)) (u),\]
so $\alpha$ and $\beta$ indeed define sections of $\pi^*\mc{K}$, and we get
\[(\alpha + i \beta)(x, v)(u) = g_x(v + iJ(x)v, d\pi(x, v) u).\]
It is easy to see, where $J = \begin{pmatrix}
0 & -1\\
1 & 0
\end{pmatrix}$ that
\[g_x(v + iJv, \partial_x) = e^{2\psi} (v_x - i v_y), \quad g_x(v + iJv, \partial_y) = e^{2\psi}(v_y + iv_x) = ie^{2\psi}(v_x - iv_y),\]
so we conclude that, we get
\[g_x(v + iJv, \bullet) = e^{2\psi} (v_x - iv_y) dz\]
and finally using that $e^{2\psi}(v_x^2 + v_y^2) = 1$
\begin{equation}\label{eq:canonical-bundle-pullback}
	(\alpha + i \beta)(x, v) = \frac{\pi^*dz}{dz(v)},
\end{equation}
which proves the claim. 
\end{proof}

Therefore if $A = fdz^m$ in $U$, we may write using \eqref{eq:canonical-bundle-pullback}:
\[\pi^*A = \pi_0^*f \cdot (\pi^*dz)^m = \pi_0^*f \cdot (dz(v))^m \cdot (\alpha + i\beta)^m = \underbrace{\pi_0^*f \cdot e^{-m\psi} e^{i m \theta}}_{\widetilde{a}:=} \cdot (\alpha + i \beta)^m,\]
where the function $\widetilde{a}$ defined by the local expression extends globally to $SM$ (since $\alpha$ and $\beta$ are global). Since $\pi_m^*(fdz^m)(z, v) = f(z) \cdot (dz(v))^m = \widetilde{a}(z, v)$ it follows that $\widetilde{a} = \pi_m^*A$, so
\[\pi^*A = \pi_m^*A \cdot (\alpha + i\beta)^m.\]

In \cite{Mettler-Paternain-19}, the authors write $\widetilde{a} = \frac{Va}{m} + ia = 2i a_m$, where $a = a_{-m} + a_m$ is real-valued (i.e. $a_{-m} = \overline{a_m}$), from where it follows
\begin{equation}\label{eq:A-a_m-relation}
	a_m = \frac{\pi_m^*A}{2i}.
\end{equation}
Therefore we obtain the relation between $\lambda := a$ and $A$:
\[\lambda = a = a_m + \overline{a_m} = 2\re(a_m) = 2\re\Big(\frac{\pi_m^*A}{2i}\Big) = -\re(i\pi_m^*A) = \im(\pi_m^*A).\]
Now using the convention that $|dz|^2 = \frac{1}{2}\Big(|dx|^2 + |dy|^2\Big) = e^{-2\psi}$, we get:
\[ |A(z)|^2_g = |f(z)|^2 \cdot |dz|^{2m} = |f(z)|^2 \cdot e^{-2m\psi},\]
and, using $dz(v) = e^{-\psi} e^{i\theta}$
\[4|a_m(z, v)|^2 = |\pi_m^*A(z, v)|^2 = |f(z)|^2 \cdot |dz(v)|^{2m} = |f(z)|^2 e^{-2m \psi}.\]
We conclude with the relation between the norms:
\begin{equation}\label{eq:A-a_m-norm-relation}
	4|a_m(z, v)|^2 = |f(z)|^2 \cdot e^{-2m\psi} = |A|_g^2.
\end{equation}
Next, we compute explicitly the expressions in local coordinates for the operators $\eta_\pm$ acting on tensor powers of the canonical bundle $\mc{K}$.
\begin{lemma}\label{lemma:eta-del-bar}
	The following identity holds for $m \geq 0$:
	\begin{align*}
		\eta_- \pi_m^*(f dz^m) &= \pi_{m-1}^*\bar{\partial}(f dz^m).
	\end{align*}
	Equivalently, we have:
	\begin{align*}
		\eta_- \pi_m^*(f dz^m) &= \pi_0^*\Big(\frac{\partial f}{\partial \bar{z}}\Big) e^{-2\psi} \cdot \pi_{m-1}^*(dz^{m-1}),\\
		\eta_+ \pi_m^*(f dz^m) &= e^{2m \psi} \frac{\partial	(f e^{-2m\psi})}{dz} \cdot \pi_{m + 1}^*(dz^{m+1}).
	\end{align*}
	Therefore $\bar{\partial} A = 0$ if and only if $\eta_- \pi_m^*A = 0$. If $m \geq 1$, then $\bar{\partial} A = 0$ if and only if $XV a = mHa$, or alternatively if and only if $HVa + mXa = 0$.
\end{lemma}
Here $\bar{\partial}(fdz^m) = \frac{\partial f}{\partial \bar{z}} \cdot d\bar{z} \otimes dz^{m}$ is well-defined since $dz$ is a holomorphic section of the holomorphic vector bundle $\mc{K}$. 
\begin{proof}
	Compute using \eqref{eq:eta-}, and $dz(v) = e^{-\psi} e^{i\theta}$:
	\begin{align*}
		\eta_- \pi_m^* (f dz^m) &= \eta_-(\pi_0^*f e^{-m\pi_0^*\psi} e^{i m\theta})\\
						&= \eta_-(\pi_0^*f) \cdot \pi_m^*(dz^m) + \pi_0^*f \cdot \eta_-(e^{-m\pi_0^*\psi} e^{i m\theta})\\
						&= \pi_0^*\Big(\frac{\partial f}{\partial \bar{z}}\Big) \cdot \underbrace{e^{-(m + 1)\pi_0^*\psi} e^{i(m - 1)\theta}}_{= e^{-2\pi_0^*\psi} \pi_{m-1}^*(dz^{m - 1})} + \pi_0^*f \cdot \eta_-(e^{-m\pi_0^*\psi} e^{i m\theta}).
	\end{align*}
	We claim that $\eta_-(e^{-m\psi} e^{im\theta}) = 0$. Indeed, we compute using \eqref{eq:eta-} that
	\begin{align*}
		\eta_-(e^{-m\psi} e^{im\theta}) = e^{-(m + 1)\psi} e^{i(m - 1)\theta} \cdot \Big(-m\frac{\partial \psi}{\partial \bar{z}} -i \frac{\partial \psi}{\partial \bar{z}} \cdot im\Big) = 0.
	\end{align*}
	Next, compute the right hand side:
	\begin{align*}
		\pi_{m-1}^*\bar{\partial}(fdz^m) &= \pi_0^*\Big(\frac{\partial f}{\partial \bar{z}}\Big) \cdot \pi_1^*(d\bar{z}) \cdot \pi_1^*(dz)^{m}\\ 
		&= \pi_0^*\Big(\frac{\partial f}{\partial \bar{z}}\Big) \cdot e^{-(m + 1)\psi} \cdot e^{i(m - 1)\theta}\\
		&= \pi_0^*\Big(\frac{\partial f}{\partial \bar{z}}\Big) \cdot e^{-2\psi} \pi_{m-1}^*(dz^{m-1}),
	\end{align*}
	which completes the proof of the first identity. The other identity follows analogously:
	\begin{align*}
		\eta_+ \pi_m^* (f dz^m) &= \eta_+(\pi_0^*f) \cdot \pi_m^*(dz^m) + \pi_0^*f \cdot \eta_+(e^{-m\psi} e^{im\theta})\\
		&= e^{-(m + 1)\psi} e^{i (m + 1) \theta} \pi_0^* \Big(\frac{\partial f}{\partial z}\Big) + \pi_0^*f \cdot e^{-\psi} e^{i\theta} e^{-m\psi} e^{im\theta} \Big(-m\frac{\partial \psi}{\partial z} + i \frac{\partial \psi}{\partial z} \times im\Big)\\
		&= e^{-(m + 1) \psi} e^{i(m + 1)\theta} \Big(\pi_0^*\Big(\frac{\partial f}{\partial z}\Big) - 2m \pi_0^*f \times \frac{\partial \psi}{\partial z}\Big)\\
		&= \pi_{m + 1}^*(dz^{m+1}) \cdot e^{2m\psi} \frac{\partial (f e^{-2m\psi})}{\partial z}.
	\end{align*}
	
	For the final conclusion, observe
	\[XVa - mHa = 0 \iff X(ia_m - i a_{-m}) - H(a_m + a_{-m}) = 0 \iff (X + iH) a_m - (X - iH) a_{-m} = 0,\]
	which holds if and only if $\eta_- a_m = 0$, which by \eqref{eq:A-a_m-relation} completes the proof. The other equivalence is similarly obtained.
\end{proof}

Recall that $X_-$ on $H_{-1} \oplus H_1$ is defined by $X_- (f_{-1} + f_1) = \eta_+ f_{-1} + \eta_- f_1$. Then:

\begin{prop}\label{prop:X-}
	It holds that:
	\begin{equation*}
	X_- \pi_1^* \gamma = -\frac{1}{2} \pi_0^* d^* \gamma, \quad \forall \gamma \in C^\infty(SM; \Omega^1).
	\end{equation*}
\end{prop}
\begin{proof}
	It suffices to consider the operator $\eta_-$ and in local isothermal coordinates, $\gamma = f dz$. By Lemma \ref{lemma:eta-del-bar}:
	\begin{align*}
		\eta_- \pi_1^*(f dz) = e^{-2\psi}\pi_0^*\Big(\frac{\partial f}{\partial \bar{z}}\Big).
	\end{align*}
	On the other hand, we compute:
	\begin{align*}
		\pi_0^* d^* (fdz) &= -\pi_0^*(\star d \star f dz) = i\pi_0^*(\star df \wedge dz)\\
		&= i \pi_0^*\Big(\star \Big(\frac{\partial f}{\partial \bar{z}}\Big) 2i dx \wedge dy\Big)\\
		&= -2e^{-2\psi} \pi_0^*\Big(\frac{\partial f}{\partial \bar{z}}\Big).
	\end{align*}
	Here we used that in isothermal coordinates $\star dz = -idz$, $d\bar{z} \wedge dz = 2i dx \wedge dy$, and that the volume form is equal to $e^{2\psi} dx \wedge dy$.
\end{proof}

\begin{prop}\label{prop:H1}
	We have $\dim \ker (\eta_-|_{H_0}) = \dim \ker (\eta_+|_{H_{0}}) = 1$ and moreover, it holds that $\dim \ker (\eta_-|_{H_1}) = \dim \ker (\eta_+|_{H_{-1}}) = \frac{b_1(M)}{2}$.
\end{prop}
	\begin{proof}
		For the first claim, by \eqref{eq:X-simple} note that $f \in \ker (\eta_-|_{H_0})$ if and only if $df = i\star df$. It follows that $\|df\|_{L^2(M)} = 0$ and so $f$ is constant. For the other case note that $\ker (\eta_+|_{H_0})$ is obtained by conjugation from $\ker (\eta_-|_{H_0})$.
		
		Next, we claim that the following map is well-defined and an isomorphism:
		\begin{equation}\label{eq:correspondence-H1}
			\mc{H}^1(M) \ni \gamma \mapsto (f_{-1}, f_1) \in \ker \eta_+|_{H_{-1}} \oplus \ker \eta_-|_{H_{1}},
		\end{equation}
		where $\pi_1^* \gamma = f_{-1} + f_1$ is the splitting of $\pi_1^* \gamma$ into Fourier modes. Since complex conjugation provides an isomorphism between $\ker \eta_+|_{H_{-1}}$ and $\ker \eta_-|_{H_{1}}$, the claim will follow.
		
		To see \eqref{eq:correspondence-H1}, by Proposition \ref{prop:X-} we have:
		\begin{align}\label{eq:system-H1}
		\begin{split}
			-\frac{1}{2} \pi_0^* (d^* \gamma) &= X_- \pi_1^* \gamma = \eta_+ f_{-1} + \eta_- f_1,\\
			- \frac{1}{2} \pi_0^* (d^* \star \gamma) &= X_- \pi_1^* \star \gamma = - X_- V \pi_1^* \gamma = -i(-\eta_+ f_1 + \eta_- f_1),
		\end{split}
		\end{align}
		where in the second line we also use $V\pi_1^* = -\pi_1^* \star$. Therefore, $\gamma$ is both closed and co-closed if and only if $\eta_- f_1 = \eta_+ f_{-1} = 0$. This shows that the map \eqref{eq:correspondence-H1} is well-defined and moreover an isomorphism, completing the proof.
	\end{proof}
	
\section{Conformal re-scaling of the geodesic vector field}\label{app:B}

In this appendix we study the behaviour of the geodesic vector field under a conformal re-scaling of a Riemannian metric on a surface. Let $g_1 = e^{-2f} g$ be a conformal scaling of the metric of the surface $(M, g)$, for some $f \in C^\infty(M)$. Consider the scaling diffeomorphism:
\[\ell_1: SM \to SM_1, \quad (x, v) \mapsto \Big(x, \frac{v}{\|v\|_{g_1}}\Big) = (x, e^f v),\]
where $SM_1$ denotes the unit sphere bundle of $(M, g_1)$. Let $X_1$ be the geodesic vector field on $SM_1$ and denote by $\pi_1$ the footpoint projection $SM_1 \to M$. Moreover, denote by $\alpha_1, \beta_1, \psi_1$ the global frame constructed in \S \ref{ssec:geometry-surfaces} on $SM_1$. Then:
\begin{lemma}\label{lemma:rescaling}
	For $(x, v) \in SM$, we have:
	\[X_f := \ell_1^*X_1 = e^{f}X + \pi_1^*(\star d (e^f)) V, \quad \ell_1^*V_1 = V.\]
	In particular $X_f$ is a time-change by $e^f$ of the thermostat flow $X - V (\pi_1^* (df)) V$ (as $\pi_1^*\star = -V\pi_1^*$). 
\end{lemma}
\begin{proof}
	Let us compute $\ell_1^*\psi_1, \ell_1^*\alpha_1, \ell_1^*\beta_1$ in terms of $\psi, \alpha, \beta$. We start with $\alpha_1$:
	\begin{align*}
		\ell_1^*\alpha_1(x,v)(\xi) = \alpha_1(x, e^f v)(d\ell_1(x, v) \xi) = \langle{\underbrace{d\pi_1 \circ d\ell_1}_{=d\pi}(x, v) \xi, e^f v}\rangle_{g_1} = e^{-f} \alpha(x, v)(\xi). 
	\end{align*}
	Then we have, if $\beta(x, v)(\xi) = \langle{d\pi(x, v) \xi, iv}\rangle_g$:
	\begin{align*}
		\ell_1^*\beta_1(x, v)(\xi) = \beta_1(x, e^f v)(d\ell_1(x, v) \xi) = \langle{d\pi_1 \circ d\ell_1(x, v) \xi, i e^f v}\rangle_{g_1} = e^{-f} \beta(x, v) (\xi).
	\end{align*}
	Recall that the $1$-form $\psi$ is defined as
	\[\psi(x, v)(\xi) = \langle{\mc{K}(x, v)\xi, iv}\rangle_g,\]
	where $\mc{K}$ is the connection map, defined by 
	\[\mc{K}(x, v)(\xi) = \frac{DZ}{dt}(0) \in T_xM,\]
	where for $c:(-\varepsilon, \varepsilon) \to TM$ is any curved satisfying $c(0) = (x, v)$, $\dot{c}(0) = \xi$, and $c(t) = (\gamma(t), Z(t))$, and $\frac{D}{dt}$ is the covariant derivative along $\gamma$ with respect to the Levi-Civita connection $\nabla$ of $(M, g)$. Let $\nabla^1$ be the Levi-Civita connection of $(M, g_1)$. Then by Koszul's formula we have for arbitrary vector fields $Y$ and $T$ (see \cite[Proposition 7.29]{Lee-18})
	\[\nabla^1_{Y}T = \nabla_{Y} T - Yf\cdot T - Tf\cdot Y + g(Y,T) \nabla f.\]
	It follows that
	\[\frac{D^1 Z(t)}{dt} = \nabla_{\dot{\gamma}}^1 Z = \nabla_{\dot{\gamma}} Z - df(\dot{\gamma}) \cdot Z - df(Z)\cdot \dot{\gamma} + \langle{\dot{\gamma}, Z}\rangle \cdot \nabla f.\]
	Recall that $Z(0) = v$ and $\dot{\gamma}(0) = d\pi(x, v) \xi$, and observe that $\ell_1 \circ c(t) = (\gamma(t), e^f Z(t))$ is a curve adapted to $(x, e^f v)$ and $d\ell_1(x, v) \xi$. Therefore:
	\begin{align*}
		\ell_1^*\psi_1(x, v)(\xi) &= \langle{\mc{K}_1(x, e^f v) (d\ell_1(x, v) \xi), ie^f v}\rangle_{g_1} = e^{-f} \Big\langle{\frac{D^1 (e^f Z)}{dt}|_{t = 0}, iv}\Big\rangle_g\\
		&= \Big\langle{\frac{D^1 Z}{dt}|_{t = 0}, iv}\Big\rangle_g\\
		&= \Big\langle{\frac{D Z}{dt}|_{t = 0}, iv}\Big\rangle_g - 0 - df(x, v)\cdot \langle{d\pi(x, v)\xi, iv}\rangle_g + \langle{d\pi(x, v)\xi, v}\rangle_g \cdot \langle{\nabla f(x), iv}\rangle_g\\
		&= \psi(x, v)(\xi) - df(x, v)\cdot \beta(x, v)(\xi) + df(x, iv)\cdot \alpha(x, v)(\xi)
	\end{align*}
	where we used that $Z(0) = v \perp iv$ in the second and third lines. We conclude that (since $df(x, iv) = -\star df(x, v)$, as both $i$ and $\star$ are rotations by $\frac{\pi}{2}$ counter-clockwise):
	\[\ell_1^*\psi_1 = \psi - \pi_1^*(df) \beta - \pi_1^*(\star df) \alpha.\]
	Now we use that
	\[\ell_1^*\alpha_1(\ell_1^*X_1) = e^{-f} \alpha(\ell_1^*X_1) = 1, \quad \ell_1^*\beta_1(\ell_1^*X_1) = \ell_1^*\psi_1(\ell_1^*X_1) = 0,\]
	to conclude that writing $\ell_1^*X_1 = aX + bH + cV$, by the above expressions:
	\[a = e^f, \quad b = 0, \quad c = \pi_1^*(\star df) e^f = \pi_1^*(\star d e^f),\]
	which completes the proof of the first claim. The formula for $\ell_1^*V_1$ readily follows as well, completing the proof.
\end{proof}

\bibliographystyle{alpha}
\bibliography{Biblio}
\end{document}